\documentclass[oneside,english,11pt]{amsart}
\usepackage[backend=biber,style=numeric,citestyle=numeric-comp,sorting=none]{biblatex}

\usepackage[english]{babel}
\usepackage[T1]{fontenc}
\usepackage{lmodern}
\usepackage{microtype}
\usepackage{xurl}
\usepackage{xcolor}
\usepackage{booktabs}
\usepackage{array}
\usepackage{longtable}
\usepackage[section]{placeins}
\usepackage{csquotes}
\usepackage{preamble}
\usepackage[parfill]{parskip}
\usepackage{adjustbox}

\SetSymbolFont{stmry}{bold}{U}{stmry}{m}{n}
\newtheorem{theorem}[thm]{Theorem}
\newtheorem{lemma}[thm]{Lemma}
\hypersetup{hidelinks}
\AtEveryBibitem{%
  \clearfield{abstract}\clearfield{keywords}\clearfield{issn}%
  \iffieldundef{doi}{}{\clearfield{url}}%
  \iffieldundef{journaltitle}{}{\clearfield{eprint}}%
  \iffieldundef{eprint}{}{\clearfield{url}}%
}
\renewbibmacro{in:}{\ifentrytype{article}{}{\printtext{\bibstring{in}\intitlepunct}}}
\DeclareFieldFormat[article,incollection,inproceedings,misc]{title}{#1}
\date{}

\providecommand{\Z}{\mathbb Z}

\DeclareMathOperator{\lcm}{lcm}

\title[New Directions in Irrationality]{\LARGE New Directions in Irrationality}
\author[Elyasheev Leibtag]{Elyasheev Leibtag\textsuperscript{1,2,*}}
\author[Hila Barkan]{Hila Barkan\textsuperscript{1,*}}
\author[Rotem Kalisch]{Rotem Kalisch\textsuperscript{1}}
\author[Shachar Weinbaum]{Shachar Weinbaum\textsuperscript{1}}
\author[Michael Shalyt]{Michael Shalyt\textsuperscript{1}}
\author[Ido Kaminer]{Ido Kaminer\textsuperscript{1}}
\thanks{\textsuperscript{*} These authors contributed equally to this work.}

\begin{document}
\raggedbottom
\begin{abstract}
Since Ap\'ery’s proof of the irrationality of \(\zeta(3)\), decades of work have explored integrals, hypergeometric identities, and recurrences with exceptional arithmetic properties, to discover the strongest quantitative bounds on irrationality measures for zeta values and related constants. 
Yet, most findings have relied on specially tailored families of such constructions, making their systematic extension difficult.
Here we show how these constructions can be connected and generalized into a common multidimensional structure based on conservative matrix fields (CMFs), enabling a systematic discovery of new irrationality proofs and quantitative bounds for irrationality exponents.
CMFs are finite-dimensional cocycles of compatible rational matrix shifts, and are shown to capture all the classical constructions developed since Apéry as one-dimensional trajectories. The initial point encodes the mathematical constants; varying the trajectory within a suitable convergence region produces different approximation schemes for the same constants.
We use this structure to improve the irrationality measures of famous constants, yielding for example
\[
    \mu(\zeta(2))<4.62675,
    \qquad
    \mu(\zeta(3))<5.4197,
\]
and proving that \(1,\zeta(3),\zeta(4)\) are linearly independent over \(\mathbb Q\), yielding
\[
    \mu\bigl(A\zeta(3)+B\zeta(4)\bigr)<66.3.
\]
for every $(A,B)\in\mathbb Q^2\setminus\{(0,0)\}$.
The same approach demonstrated over other CMFs provides irrationality proving trajectories for infinitely more mathematical constants.
We further construct higher-rank CMFs with explicit initial states for \(1,\zeta(2),\zeta(3),\ldots,\zeta(d+1)\), and identify the arithmetic conditions needed to obtain further, higher-order independence results.
\end{abstract}

\maketitle


\clearpage
\section{Introduction}
\label{sec:introduction}

When Ap\'ery announced his proof of the irrationality of $\zeta(3)$ in 1978~\cite{apery1979irrationalite,van1979proof}, he was reportedly asked where his formulas came from, and replied that they "grow in my garden". 
Over the following decades, Beukers, Zagier, Hata, Rhin and Viola, Fischler, Brown, Zudilin, and many others developed a range of constructions in an effort to find further irrationality proofs in Ap\'ery's garden \cite{beukers1979note,hata1995beukers,rhin1996permutation,hata2000zeta3,rhinViola2001,fischler2004bourbaki,zudilin2002arithmetic,zagier2006integral,zudilin2014zeta2}.

These efforts yielded substantial quantitative advances: Rhin--Viola obtained the bound $\mu(\zeta(3))<5.513891$~\cite{rhinViola2001}, and Zudilin obtained $\mu(\zeta(2))<5.095412$~\cite{zudilin2014zeta2}. 
Related constructions gave increasingly strong approximations to Catalan's constant~\cite{nesterenko2016catalan,viola2022catalan}, $\zeta(5)$~\cite{brownZudilin2022zeta5}, and combinations of zeta values~\cite{dauguetZudilin2014}. 
Ball and Rivoal proved that infinitely many odd zeta values are irrational~\cite{ball2001irrationalite}, and Zudilin proved that at least one of $\zeta(5),\zeta(7),\zeta(9),\zeta(11)$ is irrational~\cite{zudilin2002arithmetic}.
Despite these advances, no subsequent Ap\'ery-type, ``classical'', construction achieved a comparable irrationality breakthrough.

More recent approaches abandoned the classical constructions, and instead expanded the search for irrationality using methods that do not arise classically from Ap\'ery-like sequences. 
Zudilin's determinantal method replaces the requirement that individual linear forms tend to zero by estimates for structured determinants~\cite{zudilin2017determinantal}. 
Brown developed a many-parameter criterion for Mellin integrals using generalized Vandermonde determinants and transfinite diameter~\cite{brown2026mellin}. 
Fauzan's subsequent preprint uses normalized Hankel determinants to prove the irrationality of $\zeta(5)$~\cite{fauzan2026}.
OpenAI's proof of the irrationality of Catalan's constant also used a related approach~\cite{openai2026catalan}.
In another direction, Calegari, Dimitrov, and Tang applied arithmetic holonomy bounds to Zagier's earlier constructions, proving the linear independence of $1$, $\zeta(2)$, and $L(2,\chi_{-3})$ \cite{zagier2009apery,calegari2024linear,calegariDimitrovTang2025Holonomy}.

These breakthroughs raise a fundamental question: had classical Ap\'ery-type constructions reached their limits, or could they still yield new irrationality theorems, as hoped since Ap\'ery's 1978 proof?

This paper shows Apery's garden still grows irrationality proofs.
We connect and generalize all the classical constructions within a common multidimensional framework of conservative matrix fields (CMFs), and use this structure to prove new irrationality and linear-independence theorems. 
Developed over several years by generalizing low-rank families of rational approximations \cite{elimelech2024algorithm,weinbaum2025conservative,raz2025euler}, the framework now embeds all the classical constructions described above, each appearing as a distinct CMF trajectory. This embedding recovers the irrationality measures and provides a systematic way for their proofs and generalization.

\subsection{Main Results}
For an irrational real number $\alpha$, let $\mu(\alpha)$ denote its
irrationality measure; see Definition~\ref{def:irrationality-measure}.

Our first result is a proof of the irrationality of every nonzero rational combination of $\zeta(3)$ and $\zeta(4)$.
\begin{mthm}
\label{thm:intro-zeta34}
For every $(A,B)\in\mathbb Q^2\setminus\{(0,0)\}$,
\[
    A\zeta(3)+B\zeta(4)\notin\mathbb Q,
\]
and
\[
    \mu\bigl(A\zeta(3)+B\zeta(4)\bigr)<66.3.
\]
\end{mthm}

As the second result, we improve the upper bounds for the irrationality measures of $\zeta(2)$ and $\zeta(3)$.
\begin{mthm}
\label{thm:intro-zeta23}
\[
    \mu(\zeta(2))<4.62675,
    \qquad
    \mu(\zeta(3))<5.4197.
\]
\end{mthm}
These improve the recent bounds 
$\mu(\zeta(2))<5.0193784, \mu(\zeta(3))<5.5138800$ due to Giraudin \cite{giraudin2026zeta2,giraudin2026zeta3}, which in turn improve the previous bounds by Zudilin \cite{zudilin2014zeta2} and Rhin--Viola \cite{rhinViola2001}, respectively.

To show the prospects of the CMF framework beyond the main results on zeta values, we establish the irrationality of an infinite family of constants defined by the harmonic-weighted Gregory series, 
connecting Gregory's classical arctangent series with Euler sums~\cite{euler1776meditationes,flajolet1998euler}. 
We construct explicit rational approximations proving irrationality of values in this series, for every integer $M\ge7165$, improving by nearly eleven orders of magnitude the threshold obtained from a polylogarithmic criteria~\cite{kawashima2026linear}, and by more than 700 orders of magnitude the threshold obtained using general $G$-function methods~\cite{fischler2018rational}.

While this manuscript was in preparation, Lai \cite{lai2026zeta34} proved the linear independence of $1,\zeta(3),\zeta(4)$ (Theorem~\ref{thm:intro-zeta34}) by following Fauzan's usage of the determinant method, combined with classical simultaneous zeta values approximations.
Lai acknowledged that his proof was generated by an LLM.
We obtained this result using CMFs and without using an LLM. 
Our CMF approach is fully classical, in the sense that it provides an Ap\'ery-type sequence for every constant and every combination. 
Such sequences provide explicit Diophantine approximations and irrationality-measure bounds.
We formally verify these irrationality measure bounds in Lean \cite{demoura2021lean4}.

\subsection{Paper Structure}

Section~\ref{sec:preliminaries} introduces the formalism of irrationality measures and Diophantine approximations. 
Section~\ref{sec:cmfs} presents the concept of CMF and its relations to D-finite functions.
Section~\ref{sec:recovery} shows how all the classical results in the field can be unified by being embedded as trajectories in CMFs.
Section~\ref{sec:methodology} showcases the methodology of using CMFs to discover Ap\'ery-type sequences that can be used for irrationality proofs.
Section~\ref{sec:proof_strategy} proves the main results -- irrationality and linear independence.
Section~\ref{sec:other_implications} extends the use of CMFs to additional domains, proving the irrationality of the harmonic-weighted Gregory series.
Section~\ref{sec:outlook} discusses promising next steps in light of the broader context and current status of the field.

The appendix contains extra results and proof details. The code, exact certificates and Lean sources are available at \coderepo{}; see \hyperref[sec:code]{Code and data availability}.

\section{Preliminaries}
\label{sec:preliminaries}

We use the CMF conventions of \cite{weinbaum2025conservative}; see also
\cite[Appendices~E and~G]{raz2025euler}.  All matrix evaluations and
lattice paths below are assumed admissible, so that every matrix
encountered is defined and invertible.

\subsection{Diophantine Approximations}
\label{subsec:diophantine-approximation}

Let
\[
    r_n=\frac{p_n}{q_n}\longrightarrow\alpha,
    \qquad
    p_n\in\mathbb Z,\quad q_n\in\mathbb N,
\]
with $r_n\ne\alpha$, and $i\neq j \implies r_i\neq r_j$.  The fractions need not be reduced.  We write
\[
    \widetilde q_n=\frac{q_n}{\gcd(p_n,q_n)}
\]
for the reduced denominator.

\begin{defn}[Approximation exponent]
\label{def:delta}
The \emph{approximation exponent} of the sequence $r_n\to\alpha$ as above is
\[
    \delta_{(r_n)}
    =
    \liminf_{n\to\infty}
    \left(
        -1-\frac{\log|\alpha-r_n|}{\log\widetilde q_n}
    \right).
\]
\end{defn}

A positive approximation exponent implies that $\alpha$ is irrational.

\begin{defn}[Irrationality measure]
\label{def:irrationality-measure}
For irrational $\alpha$, its \emph{irrationality measure}
$\mu(\alpha)$ is the infimum of the real numbers $\mu$ such that, for
every $\varepsilon>0$,
\[
    \left|\alpha-\frac pq\right|>q^{-\mu-\varepsilon}
\]
for all $p\in\mathbb Z$ and all sufficiently large $q\in\mathbb N$.
\end{defn}

\begin{remark}[Approximation exponents and irrationality measures]
\label{rem:mu-from-delta}
All approximation sequences considered below have exponential reduced
denominator and error rates,
\[
    \widetilde q_n=\exp(\sigma n+o(n)),
    \qquad
    |\alpha-r_n|=\exp(-g n+o(n)),
    \qquad \sigma,g>0.
\]
Hence
\[
    \delta_{(r_n)}=\frac{g}{\sigma}-1.
\]
If $\delta_{(r_n)}>0$, then the standard irrationality-measure argument
gives
\[
    \mu(\alpha)\le 1+\frac{1}{\delta_{(r_n)}};
\]
see \cite[\S2, p.~196]{van1979proof}.  In particular, any proved lower
bound $\delta_{(r_n)}\ge d>0$ gives
\[
    \mu(\alpha)\le 1+\frac1d.
\]
\end{remark}

\subsection{Linear Recurrences, Cocycles, and Dual Limits}
\label{subsec:linear-cocycles}

The classical constructions considered in this paper are usually
written as scalar linear recurrences.  A recurrence
\[
    a_r(n)u_{n+r}+\cdots+a_0(n)u_n=0,
    \qquad a_j(n)\in\mathbb Q[n],
\]
produces a \emph{P-recursive}, or \emph{holonomic}, sequence
\cite{stanley1980differentiably,lipshitz1989d,
zeilberger1990holonomic}.  A convergent ratio of two rational solutions
of the same recurrence is commonly called an \emph{Ap\'ery limit}; see
\cite[\S\S2.2--2.3]{weinbaum2025conservative}.

We use the equivalent first-order matrix language, since it extends
directly to CMFs.  A matrix
\[
    T(n)\in\operatorname{GL}_r(\mathbb Q(n))
\]
defines the linear difference system
\[
    S_{n+1}=S_nT(n).
\]
Its transition matrix, which we also call its forward walk, is
\[
    \Phi_T(N)=T(0)\cdots T(N-1),
    \qquad \Phi_T(0)=I_r,
\]
so that $S_n=S_0\Phi_T(n)$.  In companion coordinates this is precisely
the scalar recurrence above.  Conversely, every rational first-order
system is, after a rational change of state coordinates, equivalent to
a companion system by the cyclic-vector theorem; this equivalence will be
made precise below \ref{def: companion form}.

The \emph{dual} or \emph{contragredient} walk is
\[
    \Phi_T^\vee(N)=\Phi_T(N)^{-\mathsf T}.
\]
For either the forward or dual walk $P_N$, fixed rational rows $u,w$
and rational columns $c,d$ give rational quantities
\[
    r_N=\frac{uP_Nc}{wP_Nd}
\]
whenever the denominator is nonzero.

For a nonzero vector or matrix $V$, we write $[V]$ for its projective
class.  Thus $[V_N]\to[V]$ means that there are nonzero scalars $a_N$
such that
\[
    a_NV_N\longrightarrow V\ne0.
\]
In particular, if $[P_N]\to[P_\infty]$ and $wP_\infty d\ne0$, then
\[
    \frac{uP_Nc}{wP_Nd}
    \longrightarrow
    \frac{uP_\infty c}{wP_\infty d}.
\]

The asymptotic theory of linear recurrences is classical, going back to
Poincar\'e--Perron theory and Birkhoff--Trjitzinsky
\cite{birkhoff1930formal,birkhoff1933analytic}; see also
\cite{wimp1985resurrecting}.  We need only the case in which one
solution is exponentially smaller than all independent solutions.

\begin{defn}[Weak Poincar\'e--Perron type]
\label{def:weak-pp}
The difference system generated by $T(n)$ is of \emph{weak Poincar\'e--Perron type} if there are a nonzero row $y$, a complementary subspace
$H\subset\mathbb R^{1\times r}$,
\[
    \mathbb R^{1\times r}=\mathbb R\cdot y\oplus H,
\]
and constants $C,g>0$ such that
\[
    \frac{\|y\Phi_T(n)\|}{\|y\|}
    \le
    Ce^{-gn}\frac{\|s\Phi_T(n)\|}{\|s\|}
    \qquad
    (s\in H\setminus\{0\},\ n\ge0).
\]
The solution $Y_n=y\Phi_T(n)$ is called \emph{recessive} (minimal).
\end{defn}

\begin{remark}
\label{rem:separated-spectrum}
If $T(n)$ tends to an invertible matrix whose eigenvalues have pairwise
distinct moduli, then its cocycle is of weak Poincar\'e--Perron type;
see \cite[Lemma~4.3]{weinbaum2025conservative}.  In the rational case,
when the limit exists,
\[
    T(n)=T_\infty+\frac{C_1}{n}+O(n^{-2}).
\]
\end{remark}

The dual walk recovers the initial values of the recessive solution.

\begin{thm}[Dual convergence]
\label{thm:dual-convergence}
Assume that $T(n)$ is invertible for all $n$, and is of weak Poincar\'e--Perron type with recessive solution  $Y_n=y\Phi_T(n)$, complement $H$, and constants $C,g>0$.
Then, for any column $\mathbf c\in\mathbb C^r$ with a nonzero component along the dominant solution of the dual system, the forward iteration of the dual system isolates the initial state $\mathbf y$ up to scaling. That is, there exists a sequence $a_n\in\mathbb C\setminus\{0\}$ such that
\[
    \lim_{n\to\infty}
    a_n\Phi_T^\vee(n)\mathbf c
    =
    \mathbf y.
\]
We refer to such columns $\mathbf c$ as \emph{non-exceptional} vectors.
\end{thm}

\begin{proof}
Set $\mathbf y=Y_0^{\mathsf T}=y^{\mathsf T}$, and $H_n=H\Phi_T(n)$. 

Let $\ell$ be the unique column vector satisfying
\[
y\, \ell=1, \qquad  h\,\ell =0\qquad \forall h\in H
\]
So defining $\ell_n=\Phi_T^{-1} (n) \ell$
\[
Y_n \, \ell_n=1 \qquad  h\,\ell_n =0\qquad \forall h\in H_n
\]
By hypothesis $\mathbf c$ is non-exceptional, i.e.\ its component along $\ell_n$ does not degenerate:
\[
    \kappa:=\liminf_{n\to\infty} \frac{|\mathbf c^{\mathsf T}\ell_n|}{\|\mathbf c\|\,\|\ell_n\|}>0 .
\]
Write $\mathbf c^{\mathsf T}\Phi_T^{-1}(n)=\alpha_n y+h_n$ with $h_n\in H$,
that is,
\[
    \mathbf c^{\mathsf T}=\alpha_nY_n+h_n\Phi_T(n).
\]
Multiplying by $\ell_n$ gives $\alpha_n=\mathbf c^{\mathsf T}\ell_n$, and
since $\|Y_n\|\,\|\ell_n\|\ge Y_n\ell_n=1$,
\[
    |\alpha_n|\,\|Y_n\|\ge\frac{|\mathbf c^{\mathsf T}\ell_n|}{\|\ell_n\|}
    \ge\frac{\kappa}{2}\|\mathbf c\|
\]
for all large $n$; in particular $\alpha_n\neq0$.
The weak Poincar\'e--Perron inequality with $s=h_n$, together with
$\|h_n\Phi_T(n)\|\le\|\mathbf c\|+|\alpha_n|\,\|Y_n\|$, gives
\[
    \frac{\|h_n\|}{|\alpha_n|}
    \le C\|y\|e^{-gn}\Bigl(1+\frac{\|\mathbf c\|}{|\alpha_n|\,\|Y_n\|}\Bigr)
    \le C\|y\|\Bigl(1+\frac{2}{\kappa}\Bigr)e^{-gn}.
\]
Hence, with $a_n=\alpha_n^{-1}$,
\[
    a_n\Phi_T^\vee(n)\mathbf c-\mathbf y
    =\bigl(h_n/\alpha_n\bigr)^{\mathsf T}\longrightarrow0 .
\]
\end{proof}

\begin{cor}[Dual ratios]
\label{cor:dual-rank-one}
Under the hypotheses of Theorem~\ref{thm:dual-convergence}, if $c_n$
is nonexceptional and $w\mathbf y\ne0$, then
\[
    \frac{u\Phi_T^\vee(n)c_n}
         {w\Phi_T^\vee(n)c_n}
    =
    \frac{u\mathbf y}{w\mathbf y}
    +O(e^{-gN+o(n)}).
\]
If in addition $[\ell_n]\to[\ell]$, then
\[
    [\Phi_T^\vee(n)]\longrightarrow[\mathbf y\ell].
\]
\end{cor}

\paragraph{Recessive initial values.}
\label{sec:initial-value-limit-principle}
For a companion recurrence, this is the classical backward-recurrence
principle for minimal solutions underlying Miller's algorithm; see
\cite{gautschi1967computational,zahar1977mathematical}.
If
\[
    Y_n=(y_n,\ldots,y_{n+r-1})
\]
is the recessive solution, the dual walk therefore recovers the
projective initial vector
\[
    [y_0:\cdots:y_{r-1}].
\]
This is the mechanism used below to recover constants from the initial
state of a CMF.

\begin{lemma}[Nonvanishing]
\label{lem:nonvanishing}
Let $A(n)\in \operatorname{GL}_r(\mathbb Q (n))$ and $v_n=\Phi_A(n)\,e_0$. Suppose that the CMF Ratios $x\,v_n$, for all rows
$x$, satisfy a common scalar recurrence of order $s$ whose trailing
coefficient does not vanish at any integer $n\ge0$. Let
$\rho_1,\ldots,\rho_s$ be rows such that the $s\times s$ matrix
\[
  \begin{pmatrix}\rho_1\\\vdots\\\rho_s\end{pmatrix}
  \bigl(v_0\ v_1\ \cdots\ v_{s-1}\bigr)
\]
is invertible. Then for every nonzero $x\in\operatorname{span}(\rho_1,\ldots,\rho_s)$
the numbers $L_n=x\,v_n$ do not vanish at $s$ consecutive indices.
\end{lemma}

\begin{proof}
Write $x=\sum_ic_i\rho_i$ with $c\ne0$. Then
$(L_0,\ldots,L_{s-1})=c\,[\rho_1;\ldots;\rho_s]\,[v_0,\ldots,v_{s-1}]\ne0$.
If $L_n=\cdots=L_{n+s-1}=0$ for some $n$, running the recurrence backwards,
which is possible since its trailing coefficient never vanishes, would
give $L_0=\cdots=L_{s-1}=0$, a contradiction.
\end{proof}

\subsection{Coboundary Equivalence}
\label{subsec:coboundary}

The natural equivalence between rational difference systems is a
rational change of state coordinates, also called a rational gauge
transformation.

\begin{defn}[Coboundary equivalence]
\label{def:coboundary}
The cocycles generated by $A(n)$ and $B(n)$ are
\emph{coboundary equivalent} if some
$U(n)\in\operatorname{GL}_r(\mathbb Q(n))$ satisfies
\[
    U(n)A(n)=B(n)U(n+1).
\]
\end{defn}

The relation telescopes to
\[
    \Phi_B(n)U(n)=U(0)\Phi_A(n),
\]
and, for the dual walks,
\[
    \Phi_B^\vee(n)U(n)^{-\mathsf T}
    =
    U(0)^{-\mathsf T}\Phi_A^\vee(n).
\]
\begin{defn}
    \label{def: companion form}
    The companion form of a matrix $M(n)$ is a coboundary-equivalent row companion $C(n)$.
\end{defn}

Since a rational matrix and its inverse have at most polynomial growth
away from their poles, coboundary equivalence preserves exponential
separation of solutions.

A scalar recurrence with coefficients in $\mathbb Q(n)$ is equivalent
to its companion system.  Conversely, the cyclic-vector theorem puts a
rational first-order system into companion form up to rational
coboundary equivalence; see
\cite[Theorem~B.2]{hendriksSinger1999}. Thus the scalar recurrence,
first-order difference-system, and cocycle languages used in the
literature describe the same one-dimensional objects up to this natural
change of coordinates.

\subsubsection{Dual Pairings}
\label{subsec:dual-pairing-prelim}

The dual pairing relates one system to the dual of another, possibly of a different rank and up to a scalar normalization.

\begin{defn}[Dual pairing]\label{def:dual-pairing}
Let $A(n)\in\operatorname{GL}_r(\mathbb Q(n))$ and $B(n)\in\operatorname{GL}_{r'}(\mathbb Q(n))$. A \emph{dual pairing} of $A$ and $B$ is a matrix $K(n)\in\mathbb Q(n)^{r\times r'}$ together with a scalar $c(n)\in\mathbb Q(n)^\times$ such that
\[
  A(n)\,K(n+1)=c(n)\,K(n)\,B(n)^{-\mathsf T}.
\]
\end{defn}
We call $K(n)$ the \emph{intertwining matrix} and $c(n)$ the
\emph{normalization factor}. The relation telescopes to
\begin{equation*}\label{eq:dual-pairing-walks}
\Phi_A(n)\,K(n)=\kappa_n\,K(0)\,\Phi_B^\vee(n),\qquad
  \kappa_n=\prod_{i<n}c(i).
\end{equation*}

\begin{lemma}[Minors of dual solutions]\label{lem:minors}
Assume $A(n)$ and $B(n)$ admit a dual pairing with $e=K(n)\,h_n$ as above. Let $r'=3$, let $Y_0,Z_0$ be row vectors and
\[
Y_n=Y_0\,\Phi_B(n),\qquad Z_n=Z_0\,\Phi_B(n)
\]
be the corresponding solutions.
If there exists a scalar $\lambda$ such that
\[
\rho\,K(0)=\lambda\,(Y_0\times Z_0)\]
then
\[
  \rho\,\Phi_A(n)e=\frac{\lambda\,\kappa_n}{\det\Phi_B(n)}\,
  \det\begin{pmatrix}Y_n\\Z_n\\h_n^{\mathsf T}\end{pmatrix}.
\]
\end{lemma}
\begin{proof}
Observe that for every invertible matrix $W$,
\[
(Y_0\times Z_0)\,W^{-\mathsf T}=(Y_0W\times Z_0W)/\det W\]
The proof follows by the telescoping relation and taking $W=\Phi_B(n)$, $w=h_n$.
\end{proof}

\section{Conservative Matrix Fields}
\label{sec:cmfs}
A conservative matrix field (CMF) is the multidimensional version of the preceding difference systems. Equivalently, it is a rational discrete flat connection on a parameter lattice. 
Early instances of such structures used for diophantine approximations appeared in Gosper's pioneering work~\cite{Gosper1990}.
We use the conservative matrix field terminology first introduced in \cite{
elimelech2024algorithm}.
\begin{defn}[Conservative matrix field (CMF)]
\label{def:CMF}
A rank-$r$, $d$-dimensional \emph{conservative matrix field} over
$\mathbb Q$ is a family
\[
    \mathcal M_v(x)\in\operatorname{GL}_r(\mathbb Q(x)),
    \qquad v\in\mathbb Z^d,
\]
satisfying
\[
    \mathcal M_{v+w}(x)
    =
    \mathcal M_v(x)\mathcal M_w(x+v).
\]
\end{defn}

Equivalently, writing $M_i(x)=\mathcal M_{e_i}(x)$, the generators
satisfy the flatness relations
\[
    M_i(x)M_j(x+e_i)
    =
    M_j(x)M_i(x+e_j).
\]
Thus parallel transport between two lattice points is independent of
the chosen admissible path; see
\cite[Definition~10 and Proposition~2.3]{weinbaum2025conservative}.

\subsection{The Trajectories of a CMF}

\begin{defn}[Trajectory]
\label{def:trajectory}
For $s\in\mathbb Q^d$ and
$t\in\mathbb Z^d\setminus\{0\}$, the \emph{trajectory} through $s$
in direction $t$ is the restriction of the CMF to the rank-one lattice
\[
    s+\mathbb Zt.
\]
It is the one-dimensional rational difference system generated by
\[
    T_{s,t}(n)=\mathcal M_v(s+nt),
\]
and its transition matrix is
\[
    \Phi_{s,t}(N)=\mathcal M_{Nt}(s).
\]
\end{defn}

Thus the linear recurrences and Ap\'ery limits above are the
one-dimensional objects obtained by restricting a CMF to a direction,
possibly after passing to companion coordinates.

A \emph{compatible state} is a row-valued function $\mathcal S$
satisfying
\[
    \mathcal S(x+v)=\mathcal S(x)\mathcal M_v(x).
\]
Its restriction
\[
    Y_n=\mathcal S(x_0+nv)
\]
is a solution of the trajectory system.

In the D-finite CMFs used in this paper, the entries of
$\mathcal S(x)$ are values of the underlying function and its
derivatives.  Thus the state at the initial lattice point contains the
special values available to the trajectory, while rational CMF ratios of
that state give the constants and rational combinations considered
below.

\begin{prop}[Initial-value limit]
\label{prop:cmf-initial-limit}
Suppose that the restriction of a compatible state $\mathcal S$ to the
trajectory through $s$ in direction $t$ is recessive and satisfies
the hypotheses of Theorem~\ref{thm:dual-convergence}.  Then every
nonexceptional terminal column $c_N$ satisfies
\[
    [\Phi_{s,t}^\vee(N)c_N]
    \longrightarrow
    [\mathcal S(s)^{\mathsf T}].
\]
Consequently, for rational rows $u,w$ with
$w\mathcal S(s)^{\mathsf T}\ne0$,
\[
    \frac{u\Phi_{s,t}^\vee(N)c_N}
         {w\Phi_{s,t}^\vee(N)c_N}
    \longrightarrow
    \frac{u\mathcal S(s)^{\mathsf T}}
         {w\mathcal S(s)^{\mathsf T}}.
\]
\end{prop}

\begin{proof}
Apply Theorem~\ref{thm:dual-convergence} to
$Y_n=\mathcal S(s+nt)$.
\end{proof}

In particular, fixing the CMF, the initial point, and the CMF ratio fixes
the target for every direction for which the same compatible state is
recessive.  Varying the direction changes the approximation scheme,
not the target.

\subsection{Coboundaries and Directions}
\label{subsec:cmf-parameters}

Coboundary equivalence extends naturally to CMFs.

\begin{defn}[Coboundary equivalence of CMFs]
\label{def:cmf-coboundary}
Two CMFs $\mathcal M^1,\mathcal M^2$ are \emph{coboundary equivalent}
if some $U(x)\in\operatorname{GL}_r(\mathbb Q(x))$ satisfies
\[
    U(x)\mathcal M_v^1(x)
    =
    \mathcal M_v^2(x)U(x+v)
    \qquad(v\in\mathbb Z^d).
\]
\end{defn}

Integral translations of the parameter lattice give the basic example.

\begin{prop}[Integral shifts and parallel trajectories]
\label{prop:parallel-trajectories}
For $k\in\mathbb Z^d$, the shifted field
\[
    \mathcal M_v^{[k]}(x)=\mathcal M_v(x+k)
\]
is coboundary equivalent to $\mathcal M$, with coboundary
$U_k(x)=\mathcal M_k(x)$.  Consequently, admissibly connected parallel
trajectories are coboundary equivalent.
\end{prop}

\begin{proof}
The cocycle relation gives
\[
    \mathcal M_k(x)\mathcal M_v(x+k)
    =
    \mathcal M_{k+v}(x)
    =
    \mathcal M_v(x)\mathcal M_k(x+v).
\]
\end{proof}

For a compatible state,
\[
    \mathcal S(x_0+k)^{\mathsf T}
    =
    \mathcal M_k(x_0)^{\mathsf T}
    \mathcal S(x_0)^{\mathsf T}.
\]
Thus the state and its rational CMF ratios are transported by the same
coboundary.  In particular, changing the initial point by an integral
lattice translation does not produce an unrelated approximation
system; it gives the corresponding parallel system in different state
coordinates.

Positive multiples of a direction contain no new geometric data.

\begin{cor}[Reduction to directions]
\label{cor:angular-search}
For $m\in\mathbb N_{>0}$,
\[
    \Phi_{x_0,mv}(N)=\Phi_{x_0,v}(mN).
\]
Hence, for any fixed rational CMF ratio,
\[
    r_N(mv)=r_{mN}(v).
\]
With the liminf convention of Definition~\ref{def:delta},
\[
    \delta_{r_N(mv)}
    \ge
    \delta_{r_N(v)}.
\]
Equality holds whenever the logarithmic ratio defining
$\delta_{r_N(v)}$ has an ordinary limit.
\end{cor}

\begin{proof}
Both trajectory products equal $\mathcal M_{Nmv}(x_0)$, and the
remaining assertions follow by passing to the subsequence $mN$.
\end{proof}

Thus, after these natural equivalences are taken into account, the
essential geometric parameter of a CMF trajectory is its direction.

\subsection{D-finite Functions for Generating CMFs}
\label{subsec:dfinite-cmf-construction}
\label{subsec:holonomic-cmf}

We describe how a CMF is associated with a $D$-finite function; we refer the reader to \cite{weinbaum2025conservative} for further details.

Let $K\coloneqq\mathbb Q(x_1,\ldots,x_d)$ and let $\mathcal V$
be a finite-dimensional $K$-space of functions of $x=(x_1,\ldots,x_d)$.
All identities are understood where the functions and rational
coefficients are defined. Suppose the coordinate shifts
$(S_if)(x)=f(x+e_i)$ preserve $\mathcal V$ and act bijectively.
They commute and are semilinear over $K$:
\[
 S_i(af)=a(x+e_i)S_if\qquad(a\in K).
\]
Choose a basis $\mathcal F=(F_1,\ldots,F_r)$ of $\mathcal V$.
Extending the shift of the basis $K$-linearly defines an invertible
map whose matrix $M_i(x)=\mathcal M_{e_i}(x)$ has columns given by
\[
 S_iF_b=\sum_{c=1}^r F_c(M_i(x))_{cb},\qquad
 \mathcal F(x+e_i)=\mathcal F(x)M_i(x).
\]
Commutativity of the shifts and uniqueness of basis expansion give
\[
 M_i(x)M_j(x+e_i)=M_j(x)M_i(x+e_j),
\]
the conserving property of Definition~\ref{def:CMF}.
\paragraph{State vectors and parameter shifts.}

For a family $F(x;z)$, the coordinates $x$ are lattice parameters
and $z=(z_1,\ldots,z_m)$ are auxiliary variables.  Write
$\partial_j=\partial/\partial z_j$ and $\theta_j=z_j\partial_j$.
D-finiteness in $z$ means that all partial derivatives of $F$ span
a finite-dimensional space over $K(z)$.
For one auxiliary variable, a differential equation of order $r$ can
be written, away from its leading-coefficient zeros, as
\[
 \theta^rF=\sum_{b=0}^{r-1}a_b(x,z)\theta^bF,
 \qquad a_b\in K(z),\qquad \theta=z\frac{d}{dz}.
\]
Repeated differentiation reduces every higher derivative to
$F,\theta F,\ldots,\theta^{r-1}F$.  At a rational point $z_0$ where
these functions and the reduction coefficients are regular, the space
\[
 \mathcal V=\operatorname{Span}_K
 \{(\theta^bF)(x;z_0):b\ge0\}
\]
is therefore finite-dimensional, of dimension at most $r$.
If parameter shifts preserve this space and act bijectively, a basis
of $\mathcal V$ gives a CMF by the preceding construction.
D-finiteness in the auxiliary variables alone does not assert this
additional shift property.

When $F,\theta F,\ldots,\theta^{r-1}F$ evaluated at $z_0$
form a basis, the corresponding row state vector is
\[
 \mathcal S_F(x)=\bigl(F,\theta F,\ldots,\theta^{r-1}F\bigr)(x;z_0).
\]
Along an admissible trajectory $x_n=x_0+nv$, it gives the compatible state
\[
 \mathcal S_n=\bigl(F,\theta F,\ldots,\theta^{r-1}F\bigr)(x_0+nv;z_0),
 \qquad
 \mathcal S_{n+1}=\mathcal S_n\mathcal M_v(x_n).
\]
Thus $\mathcal S_n=\mathcal S_0\Phi_{x_0,v}(n)$, as in
Subsection~\ref{sec:cmfs}.  For several auxiliary variables, use
a basis $P_1F,\ldots,P_rF$ in the same way, where the $P_b$ are
fixed derivative monomials and the same choice is valid at the
parameter values involved.
We refer to CMFs constructed from these families as
\emph{special-function CMFs}.

When constructing CMFs for the special case of hypergeometric families, shift invariance is supplied by
\emph{contiguity relations}
\[
 F(x+e_i;z)=L_i(x)F(x;z),
\]
where $L_i(x)$ is a rational differential operator.
Applying the basis operators $P_b$ and reducing the resulting
derivatives in the same basis gives the columns of the shift matrix.
If the shifts commute and are invertible, these matrices satisfy
the conserving relation above.  Specialization at a rational $z_0$
requires the basis, reductions, and shift inverses to remain valid.
This is the finite-space construction of
\cite{weinbaum2025conservative}; a rational change of basis gives
coboundary-equivalent CMFs as in Definition~\ref{def:cmf-coboundary}.

\paragraph{Hypergeometric families.}
For hypergeometric functions, the two ingredients above are supplied
by differential equations and contiguity relations.  For the general
family $F(\mathbf a,\mathbf b;z)={}_pF_q(\mathbf a;\mathbf b;z)$,
with $\theta=z\,d/dz$, the differential equation is
\cite[Section~16.8]{NIST:DLMF}
\[
 \left[\theta\prod_{j=1}^q(\theta+b_j-1)
       -z\prod_{i=1}^p(\theta+a_i)\right]F=0.
\]
For $p\le q+1$, the series defines an analytic function near the
origin; more generally the identities can be read formally.
The equation reduces higher derivatives to a finite list.  The elementary
parameter shifts satisfy
\[
 F(\mathbf a+e_i,\mathbf b;z)
   =\frac{\theta+a_i}{a_i}F(\mathbf a,\mathbf b;z),\qquad
 F(\mathbf a,\mathbf b-e_j;z)
   =\frac{\theta+b_j-1}{b_j-1}F(\mathbf a,\mathbf b;z).
\]
Here $e_i$ and $e_j$ select the corresponding parameter within its
block, and the identities are used where the denominators and the
chosen functions are defined.  Taking the upper parameters and the
negatives of the lower parameters as lattice coordinates makes these
positive coordinate shifts.  Applying the derivative operators to
these identities, then reducing with the differential equation,
expresses the shifted state vector in the original basis.
Where these maps are invertible, the reverse shifts are obtained by
matrix inversion.  Rational specialization of $z$ then gives a CMF.
Thus the differential equation provides finite derivative closure,
while the contiguity relations provide the parameter shifts; both
are needed for the construction.  The same principle applies to
multivariate hypergeometric systems, using partial derivatives.

\paragraph{GKZ hypergeometric functions.}
Let $x=(x_1,\ldots,x_d)$ be parameters and
$z=(z_1,\ldots,z_m)$ auxiliary variables.  A family $F(x;z)$ is
$A$-hypergeometric if, for an integer matrix
$A=(a_1,\ldots,a_m)\in\mathbb Z^{d\times m}$ of rank $d$, it
satisfies the Gel'fand--Kapranov--Zelevinsky (GKZ) system
\[
 \sum_{j=1}^m a_{ij}z_j\partial_jF=x_iF
 \quad(1\le i\le d),\qquad
 \partial^kF=\partial^\ell F\quad(Ak=A\ell),
\]
where $k,\ell\in\mathbb Z_{\ge0}^m$ and
$\partial^k=\prod_j\partial_j^{k_j}$
\cite{gelfandKapranovZelevinsky1990,saitoSturmfelsTakayama2000}.
These equations make the space spanned by $F$ and its partial
derivatives over $\mathbb Q(x,z)$ finite-dimensional.
The corresponding contiguity relation, for a compatibly normalized
family, is
\[
 F(x-a_j;z)=\partial_jF(x;z),\qquad 1\le j\le m.
\]
Indeed, differentiation sends solutions at $x$ to solutions at
$x-a_j$; the normalization identifies the chosen solutions under
this map \cite[Section~2.5]{reicheltSchulzeSevenheckWalther2021}.
Choose fixed derivative monomials $P_1,\ldots,P_r$ giving a basis
\[
 \mathcal F(x;z)=(P_1F,\ldots,P_rF)(x;z).
\]
Applying $P_b$ to the contiguity relation and reducing
$P_b\partial_jF$ in this basis gives a rational matrix $C_j(x;z)$
such that
\[
 \mathcal F(x-a_j;z)=\mathcal F(x;z)C_j(x;z).
\]
Where these maps are invertible, matrix inversion gives the opposite
shifts. Thus the derivative family is shift invariant on
$x_0+A\mathbb Z^m$. In a basis of this parameter lattice, its
coordinate shifts have matrices $M_i=\mathcal M_{e_i}$ as above.
Specializing at an admissible rational $z=z_0$ gives a CMF, with
compatible state vector obtained by expressing
$\mathcal F(x;z_0)$ in these lattice coordinates.

\paragraph{Computing the state vector and shift matrices.}
The main computational task is to find a derivative basis and a
complete set of reduction rules. The initial differential equations
need not suffice: cancellations between their leading terms can
produce additional relations.  One must account for these relations
to prove that the proposed derivatives both span and are independent,
and to express each shifted derivative uniquely in that basis.
General methods use noncommutative Gr\"obner bases
\cite{saitoSturmfelsTakayama2000} or Macaulay-matrix linear algebra
\cite{OharaTakayama2015,Chestnov2022Macaulay}.

This construction makes systematic the passage from special-function
identities to CMFs used in earlier explicit constructions
\cite{David2023,elimelech2024algorithm,raz2025euler}.
For comparison, creative telescoping provides a complementary way to obtain differential
equations and recurrences from sums and integrals, accounting for their
boundary terms \cite{zeilberger1990holonomic,zeilberger1991method}.
For example, the Beukers and Zudilin $\zeta(2)$ entries in
Table~\ref{tab:all-reconstructions} illustrate the connection with
integral and hypergeometric representations.
Once the state vector and shift matrices are known, restriction to $x_0+\mathbb Zv$ gives the trajectory system of Definition~\ref{def:trajectory}.

\section{Embedding Known Results as CMF Trajectories}
\label{sec:recovery}

Some Diophantine approximations can be realized as trajectories in CMFs.
A preliminary version of this phenomenon was shown in \cite{raz2025euler}, where several classical rational approximations to $\pi$ were identified as different trajectories in a single low-dimensional CMF.
Here we show how all earlier constructions designed to produce strong Diophantine approximations can be embedded as CMF trajectories.

By an \emph{embedding} we mean that the approximation scheme appearing in the literature is reproduced by a CMF trajectory, possibly after a finite shift
of the index, a common scalar normalization, or a specified coboundary/gauge transformation.
These embeddings reveal common CMFs behind all constructions, even though they are presented
rather differently in the literature, 

\subsection{Example Embeddings: Beukers' and Zudilin's \texorpdfstring{$\zeta(2)$}{ζ(2)} Constructions}
\label{subsec:zeta2-recovery}

What became the most prevalent classical approach for creation of Ap\'ery-type sequences is using tailored integral constructions that directly generate the sequences.
CMFs provide a framework to organize such constructions.

As an example of the simplest CMF embedding, we analyze Beukers' construction of $\zeta(2)$ \cite{beukers1979note}. After the usual integrations by parts, its linear form is given, up to sign, by
\[
    J_n=
    \int_0^1\!\int_0^1
    \frac{x^n(1-x)^n y^n(1-y)^n}
         {(1-xy)^{n+1}}\,dx\,dy = A\zeta(2)+B.
\]
This integral construction is equivalent to the following matrix in CMF notations:
\begin{equation}
 T_{\mathrm B}(n)=
 \begin{pmatrix}
 -3 & \dfrac{(n+1)(5n+8)}{(n+2)^2}\\[5pt]
 5 & -\dfrac{(n+1)(8n+13)}{(n+2)^2}
 \end{pmatrix}.
 \label{eq:beukers-integral-trajectory}
\end{equation}
In particular, $J_{n+1}=-3J_n+5K_n$; eliminating $K_n$ gives
\[
 (n+1)^2J_{n+1}+(11n^2+11n+3)J_n-n^2J_{n-1}=0
 \qquad(n\ge1),
\]
which is Ap\'ery's recurrence after the sign change $J_n\mapsto(-1)^nJ_n$~\cite{jainTzermias2005}.

The multidimensional nature of the CMF becomes visible when the exponents in the integral are allowed to vary independently. A five-parameter family was studied by Rhin and Viola~\cite[Eq.~(1.4)]{rhin1996permutation}:
\begin{equation}
 \mathcal I(x)=\int_0^1\!\int_0^1
 \frac{t^{b-1}(1-t)^{d-b-1}
       u^{c-1}(1-u)^{e-c-1}}
      {(1-tu)^a}\,dt\,du,
 \qquad x=(a,b,c;d,e).
 \label{eq:five-parameter-beukers-integral}
\end{equation}
Their parameters $(h,i,j,k,l)$ correspond to
\[
 x=(i+j-l+1,h+1,k+1;h+i+2,k+j+2).
\]
Next, define the coordinate vectors of the lattice $e_a,e_b,e_c,e_d,e_e$, 
set
\[
 \Delta=d+e-a-b-c-1,
 \qquad P=(d-a)(e-a),
\]
and define
\[
 R_t(x)=
 \begin{pmatrix}
 t\Delta-P+bc & (a-1)\Delta\\
 P & (t-a+1)\Delta
 \end{pmatrix}.
\]
The following five rational matrices generate the CMF of this integral:
\begin{equation}
 \mathcal M_{e_a}(x)=\frac{1}{a\Delta}
 \begin{pmatrix}
 a\Delta-P+bc & a\Delta\\
 P & 0
 \end{pmatrix},
 \label{eq:integral-cmf-a-shift}
\end{equation}
\begin{equation}
 \begin{aligned}
 \mathcal M_{e_b}(x)&=\frac{R_b(x)}{(d-b-1)\Delta},
 &\mathcal M_{e_c}(x)&=\frac{R_c(x)}{(e-c-1)\Delta},\\[3pt]
 \mathcal M_{-e_d}(x)&=\frac{R_{d-1}(x)}{(d-b-1)\Delta},
 &\mathcal M_{-e_e}(x)&=\frac{R_{e-1}(x)}{(e-c-1)\Delta}.
 \end{aligned}
 \label{eq:integral-cmf-other-shifts}
\end{equation}
They satisfy
\[
 \mathcal M_v(x)\mathcal M_w(x+v)
 =\mathcal M_w(x)\mathcal M_v(x+w),
\]
for the five elementary shifts displayed above. Opposite shifts are obtained
by inversion, and arbitrary lattice shifts by composition. 

The conversion from integrals to CMFs is particularly transparent once the underlying hypergeometric functions are recognized.
Then Beukers' construction can be embedded in the CMF as
\begin{equation}
 \boxed{
 \mathcal M={}_3F_2(1),\qquad
 s_{\mathrm B}=(1,1,1;2,2),\qquad
 t_{\mathrm B}=(1,1,1;2,2).
 }
\end{equation}
$x_{0,\mathrm B}$ is the initial point and $v_{\mathrm B}$ is the trajectory.
Hata's construction~\cite[Eq.~(1.2)]{hata1995beukers} uses
\begin{equation}
 \boxed{
 \mathcal M={}_3F_2(1),\qquad
 s_{\mathrm H}=(1,1,1;2,2),\qquad
 s_{\mathrm H}=(12,15,14;30,28).
 }
 \label{eq:hata-cmf-data}
\end{equation}
Beukers and Hata thus use the same CMF and the same initial point:
their constructions differ in the direction of the trajectory.

Table~\ref{tab:all-reconstructions} extends this embedding into a CMF dictionary of classical approximation constructions, with Table ~\ref{tab:appendix-reconstructions} in the appendix attempting to provide a full embedding of the literature.

The same approach embeds more advanced results in the field: A larger CMF contains Zudilin's construction for $\zeta(2)$ \cite{zudilin2014zeta2} and the recent constructions by Giraudin  \cite{giraudin2026zeta2} (see the CMF attributes in Table ~\ref{tab:all-reconstructions}). 
This realization will be used below: Zudilin's recurrence and Giraudin's recurrence are each one direction in a larger parameter lattice, and the new $\zeta(2)$ construction of Theorem~\ref{thm:intro-zeta23} is obtained by changing this direction while keeping the same CMF and initial point (see Figure~\ref{fig:zeta2-delta-sphere}).

\clearpage
\subsection{Embedding Classical Constructions in a CMF}

\label{subsec:all-classical-recovery}
\begingroup
\fontsize{8.5}{10}\selectfont
\setlength{\tabcolsep}{2.2pt}
\renewcommand{\arraystretch}{1.18}
\begin{longtable}{@{}l
>{\raggedright\arraybackslash}p{2.75cm}
c r l
>{\raggedright\arraybackslash}p{2.3cm}
>{\raggedright\arraybackslash}p{3.23cm}@{}}
\caption{CMF embeddings of classical Ap\'ery-type Diophantine approximations: capturing each one by a CMF, initial point, and trajectory. 
The computed approximation exponents are provided in each case, always matching the value reported in the literature. 
Subscripts $\mathrm F,\mathrm B$ denote forward and dual walks.
The full table is provided in Table ~\ref{tab:appendix-reconstructions}.
See Appendices~\ref{sec:appendix-history} and~\ref{app:special-cmfs}
for sources and CMF definitions.}\label{tab:all-reconstructions}\\
\toprule
Constant & Reference & Year & $\delta$ & CMF & Initial point & Trajectory \\
\midrule
$\zeta(2)$ & Ap\'ery \cite{apery1979irrationalite} & 1979 & $0.094$
& ${}_3F_2(1)$ & $(-2,-2,3;1,1)$ & $(-1,-1,1;0,0)_{\mathrm F}$ \\
 & Beukers \cite{beukers1979note} & 1979 & $0.094$
& ${}_3F_2(1)$ & $(3,3,3;6,6)$ & $(1,1,1;2,2)_{\mathrm D}$ \\
 & Hata \cite{hata1995beukers} & 1995 & $0.217$
& ${}_3F_2(1)$ & $(3,3,3;6,6)$ & $(12,15,14;30,28)_{\mathrm D}$ \\
 & Rhin--Viola \cite{rhin1996permutation} & 1996 & $0.225$
& ${}_3F_2(1)$ & $(3,3,3;6,6)$ & $(12,13,14;24,28)_{\mathrm D}$ \\
 & Zudilin \cite{zudilin2002thirdOrder} (simultaneous) & 2002 & $-0.043$
& ${}_4F_3(1)$ & $(3,3,3,3;6,6,6)$ & $(1,1,1,1;2,2,2)_{\mathrm F}$ \\
 & Zudilin \cite{zudilin2014zeta2} & 2014 & $0.244$
& ${}_4F_3(1)$ & $(1,1,1,0;1,1,1)$ & $(8,7,6,-6;1,2,3)_{\mathrm F}$ \\
 & Dauguet--Zudilin \cite{dauguetZudilin2014} (simultaneous) & 2014 & $0.062$
& ${}_4F_3(1)$ & $(1,1,1,0;1,1,1)$ & $(10,9,8,-5;1,2,3)_{\mathrm F}$ \\
 & Niedbala Giraudin \cite{giraudin2026zeta2} & 2026 & $0.252$
& ${}_4F_3(1)$ & $(1,1,1,0;1,1,1)$ & $(169,146,123,-123;\allowbreak 23,46,69)_{\mathrm F}$ \\
\addlinespace
$\zeta(3)$ & Ap\'ery \cite{apery1979irrationalite} & 1979 & $0.082$
& ${}_4F_3(1)$ & $(-2,-2,3,3;\allowbreak 1,1,1)$ & $(-1,-1,1,1;0,0,0)_{\mathrm F}$ \\
 & Beukers \cite{beukers1979note} & 1979 & $0.081$
& ${}_4F_3(1)$ & $(-2,-2,3,3;\allowbreak 1,1,1)$ & $(-1,-1,1,1;0,0,0)_{\mathrm F}$ \\
 & Hata \cite{hata2000zeta3} & 2000 & $0.157$
& $G^{4,2}_{4,4}(1)$ & $(1,1,2,2;\allowbreak 1,1,1,1)$ & $(0,1,15,16;\allowbreak 8,7,8,9)_{\mathrm F}$ \\
 & Rhin--Viola \cite{rhinViola2001} & 2001 & $0.221$
& $G^{4,2}_{4,4}(1)$ & $(1,1,2,2;\allowbreak 1,1,1,1)$ & $(0,7,31,32;\allowbreak 18,17,16,19)_{\mathrm F}$ \\
 & Zudilin \cite{zudilin2002thirdOrder} (simultaneous) & 2002 & $-0.042$
& ${}_4F_3(1)$ & $(3,3,3,3;6,6,6)$ & $(1,1,1,1;2,2,2)_{\mathrm F}$ \\
 & Dauguet--Zudilin \cite{dauguetZudilin2014} (simultaneous) & 2014 & $0.038$
& ${}_4F_3(1)$ & $(1,1,1,0;1,1,1)$ & $(10,9,8,-5;1,2,3)_{\mathrm F}$ \\
 & Niedbala Giraudin \cite{giraudin2026zeta3} & 2026 & $0.222$
& $G^{4,2}_{4,4}(1)$ & $(1,1,2,2;\allowbreak 1,1,1,1)$ & $(0,363979,1612007,\allowbreak 1664004;\allowbreak 935996,883999,\allowbreak 832002,987993)_{\mathrm F}$ \\
\addlinespace
$\zeta(5)$ & Zudilin \cite{zudilin2002thirdOrder} & 2002 & $-0.221$
& ${}_8F_7(1)$ & $(8,3,3,3,3,3,3,3;\allowbreak 6,6,6,6,6,6,6)$ & $(3,1,1,1,1,1,1,1;\allowbreak 2,2,2,2,2,2,2)_{\mathrm F}$ \\
 & Brown--Zudilin \cite{brownZudilin2022zeta5} & 2022 & $-0.134$
& ${}_8F_7(1)$ & $(5,2,2,2,2,2,2,2;\allowbreak 4,4,4,4,4,4,4)$ & $(41,17,16,15,\allowbreak 14,13,12,11;\allowbreak 24,25,26,27,28,29,30)_{\mathrm D}$ \\
\addlinespace
$\gamma$ & Rivoal \cite{rivoal2009gamma} (third order) & 2009 & $-0.940$
& $G^{3, 2}_{2,3}(1)$ & $(0, 0; 0, 0, 0)$ & $(0, 0; 1, 1, 1)_{\mathrm D}$ \\
 & Rivoal \cite{rivoal2009gamma} (simultaneous) & 2009 & $-0.928$
& $G^{4, 3}_{3, 4}(1)$ & $(0, 0, 0; 0, 0, 0, 0)$ & $(0, 0, 0; 1, 1, 1, 1)_{\mathrm D}$ \\
\addlinespace
Catalan's $G$ & Zudilin \cite{zudilin2003catalan} (original) & 2003 & $-0.476$
& ${}_3F_2(1)$ & $(2,2,\tfrac32;\tfrac72,\tfrac72)$ & $(1,1,1;2,2)_{\mathrm F}$ \\
 & Zudilin \cite{zudilin2002catalanRemarks} (modified) & 2002 & $-0.476$
& ${}_3F_2(1)$ & $(1,1,\tfrac12;\tfrac52,\tfrac52)$ & $(1,1,1;2,2)_{\mathrm F}$ \\
 & Nesterenko \cite{nesterenko2016catalan} & 2016 & $-0.476$
& ${}_3F_2(1)$ & $(1,1,\tfrac32;\tfrac52,\tfrac52)$ & $(1,1,1;2,2)_{\mathrm F}$ \\
\bottomrule

\end{longtable}
\endgroup

The trajectory column gives the direction in our coordinates; a subscript
$\mathrm F$ or $\mathrm B$ indicates whether the corresponding approximants
are read from the forward or dual walk.  In several cases the same
recurrence admits both descriptions, in related CMFs.  The displayed value
of $\delta$ is numerically evaluated from the CMF approximants.  It is not
exactly the bound stated in the cited papers; some of the papers
prove a slightly weaker explicit estimate and do not record the apparent
numerical stronger evidence. 

One entry deserves separate mention.  Brown and Zudilin construct cellular
integrals on $\mathcal M_{0,8}$ which give simultaneous rational
approximations to $\zeta(3)$ and $\zeta(5)$
\cite{brownZudilin2022zeta5}.  They also give a dual description in terms of
a very-well-poised ${}_8F_7(1)$ family.  In our coordinates its basic initial
point is
\[
    s_{\mathrm{BZ}}
    =(5,2,2,2,2,2,2,2;4,4,4,4,4,4,4),
\]
and the parameter choice corresponding to their strongest approximation in
this family gives the direction
\[
    t_{\mathrm{BZ}}
    =(41,17,16,15,14,13,12,11;
      24,25,26,27,28,29,30).
\]

This example is useful because the associated state contains several zeta
values at once.  Brown describes the problem of finding useful linear forms
in these large parameter families as ``significantly harder than finding a
needle in a haystack'' \cite{brown2014irrationality}.  In the CMF
description, the dual walk gives a direct way to isolate the desired initial
value: an inverse-transpose column recovers the initial state, and a rational
CMF ratio gives an explicit sequence
\[
    r_N\in\mathbb Q,
    \qquad
    r_N\longrightarrow\zeta(5).
\]
Thus the same hypergeometric data can be read as a one-dimensional rational
approximation sequence without reproducing the special cancellations of the
forward presentation.

Once a construction is placed in CMF coordinates, the initial point
identifies the relevant constants while the trajectory determines the
approximation scheme.  This separation is what we use next.

\subsection{Connecting Different Constants Through CMF Representations}
\label{sec:initial-point}

Once embedding distinct classical Diophantine approximations for a range of constants in the same CMF (Table ~\ref{tab:appendix-reconstructions} Appendix), we identify classical constructions that lie in the same CMF, differing only in their initial point.
For example, Beukers' construction for $\zeta(2)$ and Nesterenko's construction for Catalan's constant follow
\[
    v=(1,1,1;2,2)
\]
in the ${}_3F_2(1)$ CMF, with respective initial points
\[
    s_{\mathrm B}=(3,3,3;6,6),
    \qquad
    s_G=(3,3,\tfrac72;\tfrac{13}2,\tfrac{13}2).
\]

\begin{figure}[t]
\centering
\includegraphics[width=0.9\linewidth]{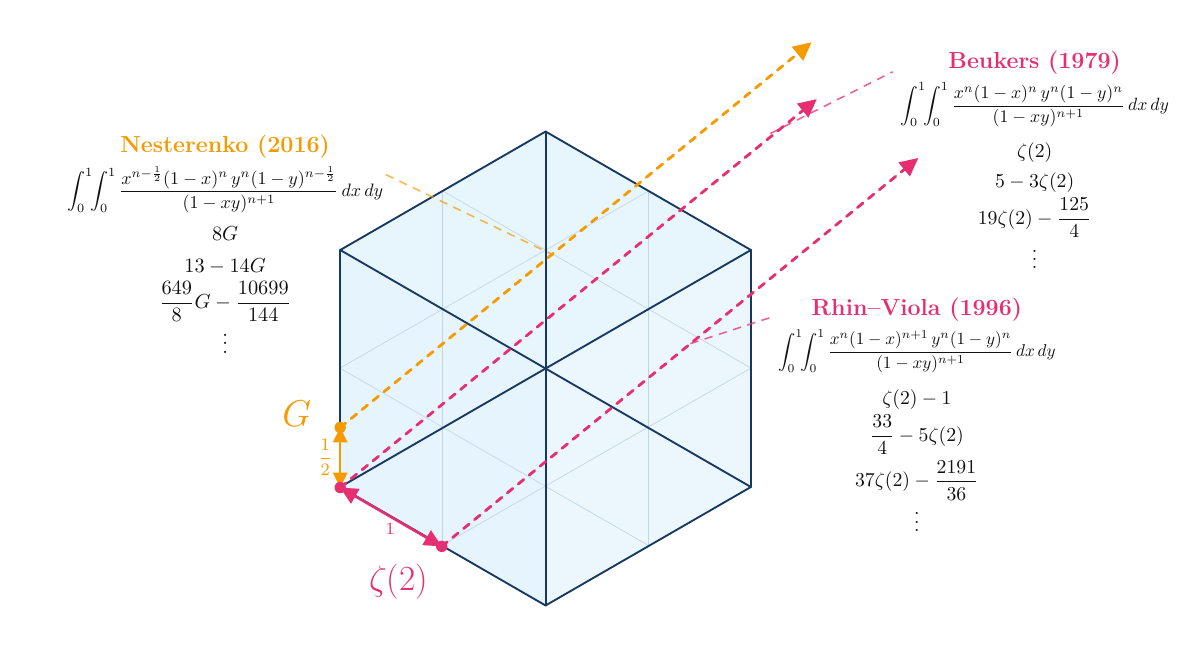}
\caption{Integral representations associated with the Beukers,
Nesterenko, and Rhin--Viola constructions.  The Beukers and Nesterenko
constructions give parallel trajectories in the ${}_3F_2(1)$ CMF with
direction $(1,1,1;2,2)$ and different initial points.}
\label{fig:z2-catalan-cube}
\end{figure}

The half-shifted Ap\'ery limits of Cohen and Zudilin
\cite{cohenZudilin2025variations} give a useful instance of the same
observation.  If a recurrence comes from the line $s+nt$, then replacing
$n$ by $n+\tfrac12$ simply replaces
\[
    s\longmapsto s+\tfrac12t,
\]
leaving the CMF and the direction unchanged.  For their tiny, small, and
big Ap\'ery recurrences the shifted data are
\[
\begin{array}{c|c|c|c}
 &\mathrm{CMF}&x_0+\tfrac12v&v\\
\hline
\mathrm{tiny}
&{}_2F_1(-1)
&(\frac12,\frac12;1)
&(1,1;2)
\\[2pt]
\mathrm{small}
&{}_3F_2(1)
&(\frac32,\frac32,\frac32;3,3)
&(1,1,1;2,2)
\\[2pt]
\mathrm{big}
&G^{4,2}_{4,4}(1)
&(\frac12,\frac12;\frac12,\frac12\mid0,0,0,0)
&(-1,-1;1,1\mid0,0,0,0).
\end{array}
\]

By Theorem~\ref{thm:dual-convergence}, the shifted limits are CMF ratios of the
corresponding initial states.  Write these states as
\[
    (U_0,U_1),\qquad
    (V_0,V_1),\qquad
    (W_0,W_1,W_2),
\]
respectively.  The relevant CMF ratios are
\[
    \frac{4U_0}{U_0+4U_1},\qquad
    \frac{16(9V_0+5V_1)}{2V_0+V_1},\qquad
    -768\,\frac{W_1+2W_2}{W_0}.
\]
Cohen and Zudilin identify these quantities with
\[
    \left(\frac{\Gamma(1/4)}{\Gamma(3/4)}\right)^2,\qquad
    \left(\frac{\Gamma(1/4)}{\Gamma(3/4)}\right)^4,\qquad
    -3\frac{\eta_+}{\omega_+},
\]
using different arguments in each case, such as continued-fraction identities
and Ap\'ery acceleration in the smaller examples, and a Beukers-type
integral, creative telescoping, and initial-period evaluations in the big
case.  The CMF description separates this case-by-case evaluation from the
identification of the shifted limit: the shift moves the initial point, and
the dual dynamics determines which state has to be evaluated.

\section{New Directions in Irrationality}
\label{sec:methodology}

The dictionary in Table~\ref{tab:all-reconstructions} and Table~\ref{tab:appendix-reconstructions} embeds classical Ap\'ery-type Diophantine approximations inside parameterized CMF families. 

This section explains how to use CMFs to discover better approximation schemes that can be used for irrationality proofs.
Given a target constant, we identify D-finite functions that evaluates it, and construct a CMF using its contiguous relations.
We then optimize the irrationality measures for that constant by exploring the space of trajectories in its lattice.

\begin{table}[htbp]
\centering
\small
\begin{adjustbox}{max width=\textwidth,center}
\thinmuskip=1mu
\medmuskip=1mu
\begin{tabular}{@{}lcclll@{}}
\toprule
Constant & Record $\delta$ & Discovered $\delta$ & CMF & Initial point & Trajectory \\
\midrule
$\zeta(3)$ & $0.222$ & $0.226$ & $G^{5,2}_{5,5}(1)$ & $(1,1,0,0,2;0,0,0,0,0)$ & $(11,12,-16,-17,10;0,1,2,-1,-2)_{\mathrm D}$ \\
$\zeta(2)$ & $0.252$ & $0.278$ & ${}_4F_3(1)$ & $(1,1,1,0;1,1,1)$ & $(10,9,11,-11;2,1,3)_{\mathrm F}$ \\
$G$ & $-0.373$ & $-0.275$ & ${}_4F_3(1)$ & $\left(4,4,-\frac{1}{2},0;\frac{7}{2},\frac{7}{2},4\right)$ & $(11,9,-2,-11;-1,0,2)_{\mathrm D}$ \\
$A\zeta(3)+B\zeta(4)$ & --- & $0.016$ & $G^{5,2}_{5,5}(1)$ & $(1,1,0,0,2;0,0,0,0,0)$ & $(11,12,-16,-17,10;0,1,2,-1,-2)_{\mathrm D}$ \\
$\zeta(2)+\zeta(3)$ & $-0.101$ & $-0.000329$ & ${}_4F_3(1)$ & $(7,8,6,9;2,3,16)$ & $(11,13,15,16;2,4,30)_{\mathrm F}$ \\
$L(2,\chi_{-3})$ & $-0.475$ & $-0.060$ & ${}_4F_3(\frac{1}{4})$ & $\left(\frac{1}{2},\frac{1}{2},\frac{1}{2},\frac{1}{2};\frac{3}{2},\frac{3}{2},\frac{3}{2}\right)$ & $(5,7,6,8;17,18,8)_{\mathrm D}$ \\
    $L(2,\chi_{-3})-\frac{2}{3}\zeta(2)$ & --- & $0.019$ & $G^{4,2}_{4,4}(1)$ & $\left(-1,0,2,\frac{8}{3};0,0,0,\frac{2}{3}\right)$ & $(  -10,-9,14,6;0,1,2,-2)_{\mathrm D}$ \\
    $L(2,\chi_{-3})+\frac{2}{3}\zeta(2)$ & --- & $0.017$ & $G^{4,2}_{4,4}(1)$ & $\left(-1,0,2,\frac{7}{3};0,0,0,\frac{1}{3}\right)$ & $(  -10,-9,14,6;0,1,2,-2)_{\mathrm D}$ \\
\bottomrule
\end{tabular}%
\end{adjustbox}
\caption{New CMF constructions with computed finite-depth exponents.
Subscripts F and D denote forward and dual walks.
$G$ is Catalan's constant and $\chi_{-3}$ is the nontrivial character modulo 3.
$\delta$ values are evaluated at $N=8000$}
\label{tab:new-results}
\end{table}

\subsection{A Systematic Search for Formulas in Ap\'ery's Garden}
\label{subsec:new-cmf-selection}
For any D-finite CMF, the initial special-function state provides a way to identify target constants: Proposition~\ref{prop:cmf-initial-limit} identifies the limits of ratios in a dual walk under the hypotheses of Theorem~\ref{thm:dual-convergence}. Rational numerator and denominator rows select ratios of linear combinations of the components of that state. Thus the choice of initial point and rows is part of the construction.

Once these choices are fixed, we vary the lattice direction within the domain where the walk is defined and the convergence hypotheses hold. Changing the trajectory produces different approximation sequences for the same target. 
Changing the initial point or the rows requires identifying the target
again. 
If one starts with a known forward approximation scheme, a dual special-function realization, connected through a rational coboundary or a pairing that preserves the approximation ratios, makes this initial-state method available.

For example, the forward ${}_4F_3(1)$ realization of Ap\'ery's $\zeta(3)$ approximants in Table~\ref{tab:all-reconstructions} has a dual presentation with data
\begin{equation}
\begin{aligned}
 \mathcal G&=G^{4,2}_{4,4}(1),\\
 s_0&=(-2,-2,3,3;0,0,0,0),\\
 d_0&=(-1,-1,1,1;0,0,0,0)_{\mathrm B}.
\end{aligned}
\label{eq:z3-seed-dual-setup}
\end{equation}

The forward realization explains the hypergeometric form of the classical denominators:
\begin{equation}
 q_m=\sum_{k=0}^{m}\binom{m}{k}^{\!2}\binom{m+k}{k}^{\!2}
 ={}_4F_3\!\left(
 \begin{matrix}-m,-m,m+1,m+1\\1,1,1\end{matrix};1\right).
 \label{eq:apery-hypergeometric-denominators}
\end{equation}
See \cite[Section~1.1, equation~(19)]{zudilin2011arithmeticHypergeometric}
for the binomial formula; the terminating series gives the second equality. Its parameter shifts are those of the forward trajectory. We retain the dual initial point and the numerator and denominator rows, and vary directions satisfying the same convergence hypotheses to seek better approximations.

\subsection{Improving Irrationality Measure by Changing CMF Trajectory}
\label{subsec:changing-traj}

The complementary operation is to keep the CMF and initial point fixed and change the direction. The recovery table shows that the same CMF may contain many approximation schemes for one constant, with substantially different
quality.  Numerically, nearby directions often have comparable finite-depth exponents, so the recovered constructions appear inside regions of related approximations rather than as isolated recurrences.

The $\zeta(2)$ case makes this particularly visible.
Figure~\ref{fig:zeta2-delta-sphere} shows $\delta_N$ over slices of the
${}_3F_2(1)$ and ${}_4F_3(1)$ direction spaces.  The recovered Beukers,
Hata, Rhin--Viola, and Zudilin constructions lie in regions of nearby
directions with comparable approximation quality; white marks the threshold
$\delta_N=0$.

\begin{figure}[htbp]
    \centering
    \setlength{\tabcolsep}{2pt}%
    \makebox[\textwidth][c]{%
        \resizebox{1.15\textwidth}{!}{%
            \begin{tabular}{@{}cc@{}}
                \raisebox{0\height}{%
                    \includegraphics[height=0.57\textwidth]
                    {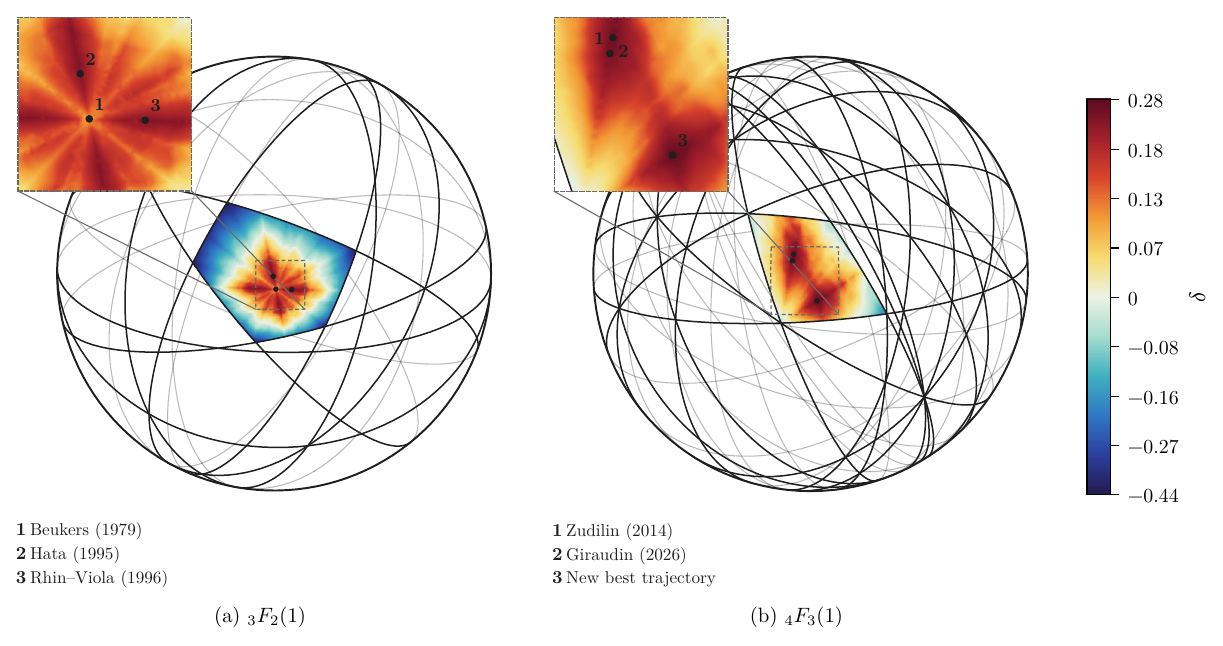}
                }
            \end{tabular}%
        }%
    }
    \caption{$\delta$ (numerical) for $\zeta(2)$ on slices of the ${}_3F_2(1)$~(a) and
    ${}_4F_3(1)$~(b) CMFs.}
    \label{fig:zeta2-delta-sphere}
\end{figure}

For $\zeta(2)$ this leads directly to the main improvement.  We keep
Zudilin's initial point
\[
    s=(0,1,1,1;1,1,1)
\]
and replace his direction
\[
    t_{\mathrm Z}=(-6 , 8,7,6;1,2,3)
\]
by
\[
    v_{\mathrm{new}}=(-11, 10,9,11;2,1,3).
\]
The resulting trajectory is analyzed in Theorem~\ref{thm:z23-measures}, where we
prove
\[
    \delta>0.2757,
    \qquad
    \mu(\zeta(2))<4.62675.
\]
Thus, in CMF coordinates, the new record construction is obtained from
Zudilin's by changing only the direction.

The same experiment can be carried out starting from other recovered constructions.
Table~\ref{tab:new-results} records several cases in which the CMF and initial point are fixed and another direction gives a stronger approximation scheme. 

The computations use exact rational matrix products implemented in \texttt{ramanujantools} \cite{ramanujantools}, with the trajectory data recorded in the accompanying \texttt{ramanujandata} repository.

Specific systematic algorithms for exploring CMF direction spaces will be discussed and reported separately.

\subsection{Lifting D-finite CMFs to a Higher Order}

A second way to extend a construction is to embed its special-function family in one with more parameters. A parameter-cancellation identity recovers the original family on a constrained set of parameter values; releasing the constraint opens directions unavailable in the original
parameter space. The identity establishes the embedding of the functions. The associated CMF and its admissible trajectories must still be constructed and checked.

For the Meijer--$G$ realization above, adjoining an equal upper and lower parameter gives
\[
 G^{5,2}_{5,5}\!\left(
 \begin{matrix}a_1,a_2,a_3,a_4,c\\b_1,b_2,b_3,b_4,c\end{matrix};z\right)
 =G^{4,2}_{4,4}\!\left(
 \begin{matrix}a_1,a_2,a_3,a_4\\b_1,b_2,b_3,b_4\end{matrix};z\right).
\]
The added upper parameter belongs to the denominator group of the Mellin--Barnes integrand, so its gamma factor cancels the added lower one. This embeds the Ap\'ery parameter family; allowing the two new parameters to vary independently opens further directions.

An enlarged family can be used in two ways. One can retain an embedded initial state and seek new trajectories for an existing target, or choose a different initial point whose state contains additional constants. In the latter case, suitable rational rows can give
simultaneous approximations with a common denominator. Neither more parameters nor a larger state alone guarantees improved approximation exponents; the benefit is a larger collection of candidates to examine.

\begin{figure}[htbp]
\label{fig:zeta3-delta-sphere}
    \centering
    \setlength{\tabcolsep}{2pt}%
    \makebox[\textwidth][c]{%
        \resizebox{1.15\textwidth}{!}{%
            \begin{tabular}{@{}cc@{}}
                \raisebox{-0.5\height}{%
                    \includegraphics[height=0.57\textwidth]
                    {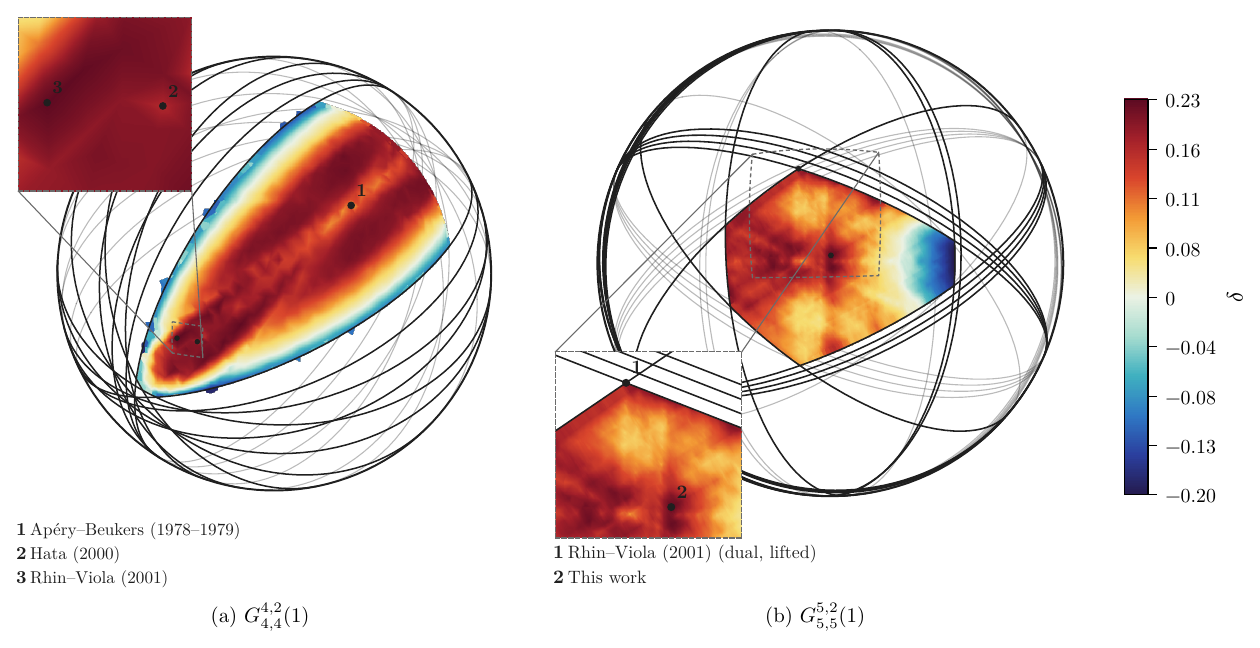}
                }
            \end{tabular}%
        }%
    }
    \caption{$\delta$ (numerical) for $\zeta(3)$ on slices of
    the $G^{4,2}_{4,4}(1)$ (a) and $G^{5,2}_{5,5}(1)$ (b) CMFs.}
\end{figure}

\section{Proving the Irrationality Exponent Using CMFs}
\label{sec:proof_strategy}

In this section we prove the main results of the paper. We first state the theorems using the CMF framework, then explain how to ''read'' these theorems and prove them using the CMF.
\begin{theorem}[CMF Theorem 1]
    \label{thm:cmf1}
    
    The CMF data
    \begin{equation*}
 \boxed{
 \mathcal M={}_5F_4(1),\qquad
 s_{\mathrm Z_{3,4}}=(0,0,1,1,1; 1,1,1,1),\qquad
 t_{\mathrm Z_{3,4}}=(-14,-15,13,14,12;1,2,3,4).
 }
\end{equation*}
    proves that for any rational numbers $A,B$, not both zero, the constant $A\zeta(3)+B\zeta(4)$ has $\delta>0.015$, and is therefore irrational. Moreover, for $B=0$ we approximate the irrational number $\zeta(3)$ and can improve the exponent to $\delta>0.226$.
\end{theorem}
\begin{theorem}[CMF Theorem 2]
    \label{thm:cmf2}
    
    The CMF data
    \begin{equation*}
 \boxed{
 \mathcal M={}_4F_3(1),\qquad
 s_{\mathrm Z_{2}}=(0,1,1,1; 1,1,1),\qquad
 t_{\mathrm Z_{2}}=(-11,10,9,11;2,1,3).
 }
\end{equation*}
    proves that constant $\zeta(2)$ has $\delta>0.2757$.
\end{theorem}

The CMF data $(\mathcal M,s_0,t_0)$ provides a trajectory matrix, see \ref{def:trajectory}.
\[
M(n)=\mathcal M_{t_0}(s_0+nt_0)
\]
whose walk produces the approximations. To show that they give a positive approximation exponent, we use the CMF structure to prove three things: (i) the limit is the intended constant, (ii) the error decays at an explicit exponential rate, (iii) the numerators and denominators share a common factor large enough to beat the growth of the denominators.

We show the CMF structure lends itself to proving the claims above in the following way.

\subsection{The CMF proof scheme}
\label{cmf proof stratergy}
We identify $M(n)$, through a dual pairing, see Definition~\ref{def:dual-pairing}, with a trajectory of a
special-function CMF $\mathcal Q=\mathcal Q(G)$ defined by a $D$-finite
function $G$:
\[
  \boxed{\;(\mathcal M,\ s_0,\ t_0)_F\;}
  \quad\xleftrightarrow[\ \text{dual pairing}\ ]{\ K(n),\ c(n)\ }\quad
  \boxed{\;(\mathcal Q,\ s_1,\ t_1)_D\;}
\]
We then use this identification to prove these claims as follows.

\paragraph{Identifying the constant.}
\label{par: constant}
We show the function state 
\[\mathcal G_n=(G_n,\theta G_n,\dots,\theta^{r'-1}G_n)\]
along the trajectory $s_1+nt_1$ is the recessive solution to the system. 
Hence, by Theorem~\ref{thm:dual-convergence} the dual walk then converges projectively to the initial state $\mathcal G_0$. The dual pairing then relates these constants to the forward walk.

\paragraph{The rate of convergence.}
The exponential growth rates are simply given by the eigenvalues of $Q(\infty)$, the matrix at infinity. Since the CMF limits are ratios of solutions, we can normalize the matrix $Q(n)$ to solve the characteristic polynomial at infinity.

\paragraph{Controlling the arithmetic growth.}\label{par: arithmetic control}
To bound the growth of the reduced denominators, we write the approximants as ratios of integers after an explicit normalization, and divide out a large common factor. For this, we aim to present the solutions in closed form, so that the primes dividing each term can be counted. The main tool is Legendre's formula
\[
v_p(m!)=\sum_{i\ge1}\lfloor m/p^i\rfloor.
\]
For primes $p^2>m$, only the first term survives, and the $p$-adic valuation of a binomial coefficient $\binom ab$ is given by the floor function 
\[
\lfloor a/p\rfloor-\lfloor b/p\rfloor-\lfloor(a-b)/p\rfloor.
\] 
Binomial coefficients of this kind appear in these hypergeometric constructions. 
A sum is divisible by $p^e$ if every term is, so minimizing over the terms gives a step function $\gamma(n/p)$ that bounds the valuation of the sum, then $\mathcal D_n=\prod_pp^{\,\gamma(n/p)}$ is a common divisor, and by the prime number theorem $\frac1n\log\mathcal D_n\to\int_0^\infty\gamma(x)\,x^{-2}\,dx$.
The closed forms are found from the structure of the construction and proved by the recurrence argument: the companion form of $M(n)$ gives the recurrence for the solution, and the closed forms satisfy it and agree at the initial values.

A single presentation usually misses divisibility that comes from cancellation between terms. One therefore expresses the same approximants in alternative forms, counts primes in each, and keeps for each prime the best count. 
The summands of the alternative presentations are
not necessarily integers, but they still bound the valuations from below.
The CMF gives three ways to present the same terms:
\begin{enumerate}
\item \emph{Parameter symmetries.} The classical permutation-group method
\cite{Hata90,rhin1996permutation,rhinViola2001,zudilin2014zeta2} uses the symmetries of the hypergeometric or integral representation, which rewrite the same numbers with different terms.
\item \emph{The dual presentation.} Beyond these symmetries, the dual pairing expresses the approximants as minors of dual solutions.
\item \emph{Walk decompositions.} A trajectory matrix lives in a CMF, so by conservation its walk can be replaced by the product along any lattice path with the same endpoints. Decomposing the path into blocks exposes the factors added by each block, and the decomposition can be adapted to the initial condition or to individual primes. This use of conservation goes back to \cite{David2023CMF}.
\end{enumerate}

\paragraph{Irrationality exponent.}
\label{par: delta}
From the three steps above we get rational numbers $ p_n,q_n\in \mathbb Q$, and real numbers $0<h,g$ satisfying
\[
  \Bigl|\alpha-\frac{p_n}{q_n}\Bigr|\le e^{-(g-o(1))n},
  \qquad
|\widetilde{q}_n|\le e^{(h+o(1))n}
\]
where $\widetilde{q}_n$ represents the reduced denominators.
Moreover, $p_n,q_n$ are both solutions to a holonomic recurrence given by the companion form of $M(n)$, see \ref{def: companion form}. The error term $L_n=\alpha q_n-p_n$ satisfies the same recurrence, and we use Lemma~\ref{lem:nonvanishing} to show that the approximations are nontrivial. 

By the definition in Section~\ref{subsec:diophantine-approximation}, this gives the approximation exponent:
\[
  \delta=\frac gh-1,
\]
so $\alpha$ is irrational when $g>h$, and $\mu(\alpha)\le\frac{g}{g-h}$. 
\begin{remark}[Finding the dual]\label{rem:finding-dual}
The dual CMF comes from the adjoint of the differential equation defining $\mathcal M$. For the hypergeometric CMFs here, its solutions are Meijer $G$-functions with related parameters.
\end{remark}
\subsection{Linear Independence of
\texorpdfstring{$1,\zeta(3),\zeta(4)$}{1, ζ(3), ζ(4)}}
\label{subsec:z34}
We use the CMF proof method to prove the following theorem.

\begin{theorem}\label{thm:z34}
The numbers $1$, $\zeta(3)$ and $\zeta(4)$ are linearly independent over $\mathbb Q$.
\end{theorem}

We follow the scheme of Section~\ref{cmf proof stratergy}, with the dual
pairing

{\setlength{\fboxsep}{2pt}%
\[
  \boxed{\bigl({}_5F_4(1),\lbl{s_0}\pv{0,0,1,1,1}{1,1,1,1},
         \lbl{t_0}\pv{-14,-15,13,14,12}{1,2,3,4}\bigr)}
  \xleftrightarrow{K(n),c(n)}
  \boxed{\bigl(G^{5,2}_{5,5}(1),\lbl{s_1}\pv{0,0,1,1,2}{0,0,0,0,0},
         \lbl{t_1}\pv{-16,-17,11,12,10}{-2,-1,0,1,2}\bigr)}
\]}
Along the two trajectories we have the functions, their states, and their
initial values:
\[
\begin{aligned}
  F_n&={}_5F_4(s_0+n\,t_0;\,1),
  &\qquad
  G_n&=G^{5,2}_{5,5}(1;\,s_1+n\,t_1),\\[2pt]
  \mathcal F_n&=\bigl(F_n,\ \theta F_n,\ \theta^2F_n,\ \theta^3F_n\bigr),
  &\qquad
  \mathcal G_n&=\bigl(G_n,\ \theta G_n,\ \theta^2G_n\bigr),\\[2pt]
  \mathcal F_0&=\bigl(1,\ 0,\ 0,\ 0\bigr),
  &\qquad
  \mathcal G_0&=\bigl(1-2\zeta(3)+3\zeta(4),\ \ 1-2\zeta(3),\ \ 1\bigr).
\end{aligned}
\]
The trajectory matrix of the dual is block upper triangular,
\[
  \mathcal Q_{t_1}(s_1+n\,t_1)=\begin{pmatrix}Q(n)&*\\0&-1\end{pmatrix},
  \qquad
  \mathcal G_{n+1}=\mathcal G_n\,Q(n).
\]
\begin{remark}[Linear forms from the dual solution]\label{rem:z34-type-one}
The recessive dual solution already gives linear forms: 
\[
  \mathcal G_n=\zeta(4)\mathbf a_n+\zeta(3)\mathbf b_n+\mathbf d_n ,
\]
and after clearing denominators, each coordinate is an integer linear form in $1,\zeta(3),\zeta(4)$ tending to $0$.
This form alone does not prove linear independence: it decays at rate about $26$ while its coefficients grow at rate about $87$, so Nesterenko's criterion \cite{Nesterenko85} gives only
\[
\dim_{\mathbb Q}\langle1,\zeta(3),\zeta(4)\rangle\ge1.3\]
implying the irrationality of one of $\zeta(3),\zeta(4)$.
\end{remark}

\begin{lemma}\label{prop:z34-connection}
The trajectory matrix $M(n)$ and the block $Q(n)$ are related by a dual
pairing,
\[
  M(n)\,K(n+1)=c(n)\,K(n)\,Q^\vee(n),
\]
and $e_0=K(n)\,h_n$ for
\[
h_n=\bigl(1,\,10n+1,\,(10n+1)^2\bigr)^{\mathsf T}/\prod_{j=8}^{12}(jn+1).\]
\end{lemma}

\begin{proof}
Exact verification is in the supplementary material. We record the value at the starting point,
$K(0)=\left(\begin{smallmatrix}0&0&1\\0&0&0\\-1&2&-1\\-2&2&0\end{smallmatrix}\right)$,
which is used below.
\end{proof}

\paragraph{Identifying the limit.}
\label{subsubsec:z34-constants}

Define the initial conditions
\[
  \rho_q=(12,0,0,0),\qquad \rho_{p,3}=(0,0,-6,3),\qquad
  \rho_{p,4}=(0,0,0,-2),
\]
\begin{lemma}\label{lem:z34-limits}
Let
\[
  q_n=\rho_q\,\Phi_M(n)\,e_0,\qquad
    p_{s,n}=\rho_{p,s}\,\Phi_M(n)\,e_0\qquad(s=3,4).
\]
Then the ratios converge:
\[
  \lim_{n\to\infty}\frac{p_{s,n}}{q_n}=\zeta(s)\qquad(s=3,4).
\]
\end{lemma}

\begin{proof}
The state $\mathcal G_n$ is the recessive solution of $Q(n)$ (Appendix~\ref{app:zeta34-proof}).  The Dual Convergence Theorem~\ref{thm:dual-convergence} applies, therefore
\[
  \lim_{n\to\infty}\frac{p_{s,n}}{q_n}
  =\frac{\rho_{p,s}\,K(0)\,\mathcal G_0^{\mathsf T}}{\rho_q\,K(0)\,\mathcal G_0^{\mathsf T}}
  =\frac{12\,\zeta(s)}{12}=\zeta(s).
\]
\end{proof}

\paragraph{The rate of convergence.}
\label{subsubsec:z34-error}

\begin{lemma}\label{lem:z34-selected-error}
For $(A,B)\in\mathbb Q^2\setminus\{(0,0)\}$,
\[
  \Bigl|A\zeta(3)+B\zeta(4)-\frac{Ap_{3,n}+Bp_{4,n}}{q_n}\Bigr|
  \le e^{-(g-o(1))n},\qquad g>112.8075 .
\]
\end{lemma}

\begin{proof}
See Appendix \ref{app:zeta34-proof}
\end{proof}
\paragraph{Controlling the arithmetic growth.}
\label{par: gcd z34}
The symbolic verifications and exact computations supporting this arithmetic argument are provided in the supplementary material.

Finding an integer form for the denominators is straightforward.
\[
q_n=12F_n = 12\sum_{k=0}^{14n}t_{n,k},\qquad  
t_{n,k}=\frac{(-14n)_k(-15n)_k(1+12n)_k(1+13n)_k(1+14n)_k}{k!\,(1+n)_k(1+2n)_k(1+3n)_k(1+4n)_k}.
\]
\begin{lemma}[Denominator sum presentation]\label{prop:z34-q-iv}
Let $C_n$ be the expression
\[
C_n=\binom{12n}{10n}\binom{13n}{10n}
       \binom{14n}{10n}\binom{15n}{n}
\]
Then
\[
 \mathbf q_n=\frac{1}{12} C_nq_n
 =\sum_{k=0}^{14n}B_{n,k}, \qquad B_{n,k}=t_{n,k}C_{n}\in\mathbb Z_{>0}
\]
\end{lemma}
\begin{proof}
    The binomial factor $C_n$ clears the denominators of the summands $t_{n,k}$.
\end{proof}
For $s=3,4$ let
\begin{equation*}
  \mathbf p_{s,n}=\sum_{k=-n}^{15n}
  \Bigl(B_{n,k}H^{(s)}_{4n+k}-\frac1{s-1}\,B'_{n,k}H^{(s-1)}_{4n+k}\Bigr),
\end{equation*}
where $B'_{n,k}$ is the derivative of $B_{n,k}$ in $k$
(Appendix~\ref{app:zeta34-proof}).

\begin{lemma}[Numerator sum presentation]\label{prop:z34-p-iv}
For $s=3,4$ and all $n\ge0$, $\mathbf p_{s,n}=\frac{C_n}{12}\,p_{s,n}$.
\end{lemma}

\begin{proof}
This is by the recurrence argument. Both sides satisfy the recurrence given by the companion form of $M(n)$, and they agree at $n=0,1,2$, the exact certificate is in the supplied notebook.
\end{proof}

\begin{lemma}[Integer normalization]\label{prop:z34-int} With $D_n=6d_{19n}^4$,
\[
D_n(\mathbf q_n,\mathbf p_{3,n},\mathbf p_{4,n})\in\mathbb Z^3.
\]
\end{lemma}
\begin{proof}
$B_{n,k}\in\mathbb Z$ and $d_{16n}B'_{n,k}\in\mathbb Z$, and the harmonic indices in $\mathbf p_{s,n}$ lie in $[3n,19n]$, so $d_{19n}^rH^{(r)}_j\in\mathbb Z$. For a detailed calculation, see the supplementary notebook.
\end{proof}

We can now see how much cancellation is needed to pass the irrationality threshold.
\begin{table}[H]
\centering
\renewcommand{\arraystretch}{1.25}
\small
\begin{tabular}{llr}
\toprule
& source & rate\\
\midrule
growth of $\mathbf q_n=\sum_kB_{n,k}$ & Stirling, largest term & $\sigma<65.4516$\\
normalization $D_n=6d_{19n}^4$ & prime number theorem & $76$\\
\midrule
growth of $D_n\mathbf q_n$ & sum of the two & $<141.4516$\\
rate of convergence $g$ &  asymptotics & $>112.8075$\\
\midrule
\textbf{common divisor needed} & difference & $\boldsymbol{>28.6441}$\\
\bottomrule
\end{tabular}
\end{table}
To account for more cancellations, we present these integers via the Dual representation. This gives an alternative expression for the same integer quantities.

Looking at the dual CMF $\mathcal{Q}$, write the initial state as
\[\mathcal G_0=\zeta(4)\,\mathbf a_0+\zeta(3)\,\mathbf b_0+\mathbf d_0,\] with the rational vectors
\[
  \mathbf a_0=(3,0,0),\qquad \mathbf b_0=(-2,-2,0),\qquad \mathbf d_0=(1,1,1),
\]

The kernel of $G^{5,2}_{5,5}$ along the trajectory is the rational function
\[
  R_n(s)=(-1)^n\,\frac{(1+2n-s)_{14n}\,(1+n-s)_{16n}}
  {(s)_{11n+1}\,(s+n)_{11n+1}\,(s+2n)_{8n+2}} .
\]

\begin{lemma}[Dual sum presentation]
\label{lem: dual sum}
Let $J_n=\max(12n,10n+1)$ and define
\[
  (-s)^kR_n(s)
  =\sum_{j=0}^{J_n}\sum_{i=1}^3
    \frac{c^{(k)}_{n,j,i}}{(s+j)^i}.
\]
    For $k=0,1,2$ define
\begin{align*}
  \mathbf a_{n,k}&=3\sum_jc^{(k)}_{n,j,3},\\
  \mathbf b_{n,k}&=2\sum_jc^{(k)}_{n,j,2},\\
  \mathbf d_{n,k}&=-\sum_j\sum_{i=1}^3i\,c^{(k)}_{n,j,i}\,H^{(i+1)}_{2n+j}.
\end{align*}
Then for all $n\ge0$,
\[
  \mathbf a_n=\mathbf a_0\,\Phi_Q(n),\qquad \mathbf b_n=\mathbf b_0\,\Phi_Q(n),\qquad
  \mathbf d_n=\mathbf d_0\,\Phi_Q(n).
\]
\end{lemma}
\begin{proof}
    We use the recurrence trick again. We show that both sides satisfy the same order-three recurrence. We first compute the recurrence for the $k$-th coefficients via companion forms coming from the trajectory matrix $Q(n)$. The exact certificate in the supplied notebook.
\end{proof}

\begin{lemma}[A second presentation]\label{prop:z34-dual}
Let 
\[
\tilde u_n=(1,10n+1,(10n+1)^2)
\]
and 
 \[
  \Lambda_n=\frac{(8n)!\,(9n)!^2\,(11n)!^3\,(12n)!^3\,(13n)!^2}
                 {6\,(16n)!^2\,(17n)!^2\,(18n)!^2\,(19n)!}.
\]
Then
\[
  \mathbf q_n=\Lambda_n\det\begin{pmatrix}\mathbf b_n\\\mathbf a_n\\\tilde u_n\end{pmatrix},\quad
  \mathbf p_{3,n}=-\Lambda_n\det\begin{pmatrix}\mathbf d_n\\\mathbf a_n\\\tilde u_n\end{pmatrix},\quad
  \mathbf p_{4,n}=\Lambda_n\det\begin{pmatrix}\mathbf d_n\\\mathbf b_n\\\tilde u_n\end{pmatrix}.
\]
\end{lemma}
\begin{proof}
The proof follows from the Minors Lemma \ref{lem:minors} and is direct linear algebra calculations carried out in the supplementary notebook.
\end{proof}
We now count primes in the forward and the dual presentations, keeping for each prime the larger count, as described in Section~\ref{par: arithmetic control}.

\begin{lemma}[Common divisor]\label{prop:z34-divisor}
The integers $6d_{19n}^{4}\,(\mathbf q_n,\mathbf p_{3,n},\mathbf p_{4,n})$ have a common divisor $\mathcal D_n$ such that
\[
\liminf_{n\to\infty}\frac1n\log\mathcal D_n\ge30.3467.
\]

\end{lemma}

\begin{proof}
We use Legendre's formula to count the primes in each presentation, and keep for each prime the larger count. The resulting rates are:
\begin{center}
\renewcommand{\arraystretch}{1.25}
\small
\begin{tabular}{lr}
\toprule
count & rate\\
\midrule
needed  & $>28.6441$\\
\midrule
forward presentation alone & $<26.29$\\
dual presentation alone & $<18.25$\\
\midrule
\textbf{both, best per prime} & $\boldsymbol{\ge30.3467}$\\
margin & $>1.70$\\
\bottomrule
\end{tabular}
\end{center}
Note that neither presentation alone reaches the threshold.
The exact evaluations, including tail bounds for the individual counts and the resulting density certificate, are provided in the supplementary material.
\end{proof}

\paragraph{Nonvanishing.}

\begin{lemma}\label{lem:z34-nonvanishing}
For $(A,B)\in\mathbb Q^2\setminus\{(0,0)\}$ and $\alpha=A\zeta(3)+B\zeta(4)$, the errors
\[
  L_n=\alpha\,q_n-Ap_{3,n}-Bp_{4,n}
\]
do not vanish at three consecutive indices.
\end{lemma}

\begin{proof}
The sequences $q_n,p_{3,n},p_{4,n}$, and hence $L_n$, satisfy the recurrence given by the companion form of $M(n)$. The $3\times3$ matrix of their values at $n=0,1,2$ is invertible, and the recurrence can be run backwards at every index $n\ge0$. Since
\[
\alpha\rho_q-A\rho_{p,3}-B\rho_{p,4}\ne0,
\] Lemma~\ref{lem:nonvanishing} gives the claim.
\end{proof}
\paragraph{Proof of Theorem \ref{thm:intro-zeta34}.}

We now conclude with a proof of the main theorem.
\begin{proof}[Proof of Theorem \ref{thm:z34}]
Without loss of generality, we may assume $A, B$ to be integers. Let $\alpha=A\zeta(3)+B\zeta(4)$ and consider the approximations 
\[
r_n=(Ap_{3,n}+Bp_{4,n})/q_n,
\] 
with the integer representatives
\[
  \widetilde q_n=\frac{D_n\mathbf q_n}{\mathcal D_n},\qquad
  \widetilde p_n=\frac{D_n(A\mathbf p_{3,n}+B\mathbf p_{4,n})}{\mathcal D_n},
  \qquad D_n=6d_{19n}^4 .
\]
The four claims hold:
\begin{enumerate}
\item the limit is $\alpha$ (Lemma~\ref{lem:z34-limits}).
\item the rate of convergence is $g>112.8075$
  (Lemma~\ref{lem:z34-selected-error}).
\item the representatives are integers, and they grow at rate
  $h<141.4516-30.3467<111.105$ 
  (Lemma~\ref{prop:z34-divisor}).
\item the errors cannot vanish at three consecutive indices
  (Lemma~\ref{lem:z34-nonvanishing}).
\end{enumerate}
The approximation exponent (Definition~\ref{def:delta}) is $\delta=\frac gh-1>0.015$. Hence $\alpha$ is irrational, and 
\[\mu(\alpha)\le\frac{g}{g-h}<
\frac{112.8075}{112.8075-111.105}
=66.2599<66.3.
\]
\end{proof}

\subsection{A Bound for the Irrationality Measures of
\texorpdfstring{$\zeta(2)$ and $\zeta(3)$}{ζ(2) and ζ(3)}}
\label{subsec:z23}
We use the CMF proof method to prove the following theorem.

\begin{theorem}[Irrationality measures]\label{thm:z23-measures}
The irrationality measures (Definition~\ref{def:irrationality-measure}) of
$\zeta(2)$ and $\zeta(3)$ satisfy
\[
  \mu\bigl(\zeta(2)\bigr)<4.62675,
  \qquad
  \mu\bigl(\zeta(3)\bigr)<5.4197 .
\]
\end{theorem}


For $\zeta(2)$ we follow the scheme of Section~\ref{cmf proof stratergy},
with the dual pairing
{\setlength{\fboxsep}{2pt}%
\[
  \boxed{\bigl({}_4F_3(1),\lbl{s_0}\pv{0,1,1,1}{1,1,1},
         \lbl{t_0}\pv{-11,10,9,11}{2,1,3}\bigr)}
  \xleftrightarrow{K(n),c(n)}
  \boxed{\bigl(G^{4,1}_{4,4}(-1),\lbl{s_1}\pv{0,1,1,2}{0,0,0,0},
         \lbl{t_1}\pv{-11,10,9,11}{0,2,1,3}\bigr)}
\]}
Along the two trajectories we have the functions, their states, and their
initial values:
\[
\begin{aligned}
  F_n&={}_4F_3(s_0+n\,t_0;\,1),
  &\qquad
  G_n&=G^{4,1}_{4,4}(-1;\,s_1+n\,t_1),\\[2pt]
  \mathcal F_n&=\bigl(F_n,\ \theta F_n,\ \theta^2F_n\bigr),
  &\qquad
  \mathcal G_n&=\bigl(G_n,\ \theta G_n,\ \theta^2G_n\bigr),\\[2pt]
  \mathcal F_0&=\bigl(1,\ 0,\ 0\bigr),
  &\qquad
  \mathcal G_0&=\bigl(\zeta(2)-\zeta(3)-1,\ \ \zeta(2)-1,\ \ -1\bigr),
\end{aligned}
\]

\begin{lemma}\label{lem:z2-connection}
The trajectory matrices $M(n)$ and $Q(n)$ are related by a dual pairing,
\[
  M(n)\,K(n+1)=c(n)\,K(n)\,Q^\vee(n),
\]
with $K(n)\in\operatorname{GL}_3(\mathbb Q(n))$ and $c(n)\in\mathbb Q(n)$.
\end{lemma}

\begin{proof}
Exact verification is in the supplementary material. We record the value at the starting point,
$K(0)=\left(\begin{smallmatrix}0&0&-1\\-1&2&-1\\-2&2&0\end{smallmatrix}\right)$,
which is used below.
\end{proof}

\paragraph{Identifying the limit.}
Define the initial conditions
\[
  \rho_q=(2,0,0),\qquad \rho_{p}=(0,2,-1).
\]
\begin{lemma}\label{lem:z2-limits}
Let
\[
  q_n=\rho_q\Phi_M(n)e_0,\qquad
  p_{n}=\rho_{p}\Phi_M(n)e_0
\]
Then the ratios converge:
\[
  \lim_{n\to\infty}\frac{p_{n}}{q_n}=\zeta(2).
\]
\end{lemma}

\begin{proof}
The state $\mathcal G_n$ is the recessive solution of $Q(n)$
(Appendix~\ref{app:zeta2-proof}).
The Dual Convergence Theorem~\ref{thm:dual-convergence} applies
therefore
\[
  \lim_{n\to\infty}\frac{p_{n}}{q_n}
  =\frac{\rho_{p}\,K(0)\,\mathcal G_0^{\mathsf T}}{\rho_q\,K(0)\,\mathcal G_0^{\mathsf T}}
  =\zeta(2).
\]
\end{proof}

\paragraph{The rate of convergence.}

\begin{lemma}\label{lem:z2-selected-error}
\[
  \Bigl|\zeta(2)-\frac{p_{n}}{q_n}\Bigr|
  \le e^{-(g-o(1))n},\qquad g>62.636575 .
\]
\end{lemma}

\begin{proof}
See Appendix \ref{app:z2-error}
\end{proof}

\paragraph{Controlling the arithmetic growth.}
\label{subsubsec:zeta2-blocks}
To bound the growth rate of the reduced denominators, we find a closed form for the approximants, as we did for $\zeta(3)$ and $\zeta(4)$ in Section~\ref{par: gcd z34}, the technical part is then to count the primes dividing its terms. 
This counting is carried out by exact computation, summarized in Appendix~\ref{app:z2-blocks}, with the certificates and the Lean formalization in the supplementary material.

Here, instead of finding a closed form for $p_n$ directly, we use the conservation of the CMF: The walk can be replaced by a product along a different lattice path. We decompose the walk such that for each path the product collapses into an explicit binomial sum. This closed form is used for counting primes. The details are in Appendix~\ref{app:z2-blocks}, and the whole argument is formalized in Lean~4, assuming only the prime number theorem \cite{demoura2021lean4}, see the supplementary material.

By conservation, the walk equals the product along the path of four blocks (Appendix~\ref{app:z2-blocks}).:
\[
\begin{array}{c|c|c}
\text{block}&\text{step}&\text{number of steps}\\\hline
A&(0,1,1,1;\,0,0,0)&3n\\
B&(0,0,0,0;\,2,1,3)&n\\
C&(0,7,6,8;\,0,0,0)&n\\
D&(-1,0,0,0;\,0,0,0)&11n
\end{array}
\]

\begin{lemma}[Integer normalization]\label{prop:zeta2-int}
Let 
\[
C_n=\prod_{u=9,10,11}\binom{un}{8n},\qquad \mathcal D^{\mathrm{car}}_n=\prod_{\sqrt{50n}<p\le13n}p^{\phi(n/p)}
\] 
and let $\mathcal D^{\mathrm{sh}}_n$ be the product of the primes $p>\sqrt{50n}$ in the six shells of Appendix~\ref{app:z2-blocks}. Then for every $n\ge1$,
\[
  \frac{C_n\,d_{12n}\,d_{13n}}{\mathcal D_n}\,(p_n,q_n)\in\mathbb Z^2,
  \qquad \mathcal D_n=\mathcal D^{\mathrm{car}}_n\,\mathcal D^{\mathrm{sh}}_n .
\]
\end{lemma}

\begin{proof}
From the closed form, Legendre's formula on its binomial terms gives the carry divisor, and cancellation over $\mathbb F_p$ gives one more factor for each shell prime (Appendix~\ref{app:z2-blocks}).
\end{proof}

\begin{lemma}[Arithmetic growth]\label{lem:z2-growth}
The numbers
\[
  \widetilde q_n=\frac{C_n\,d_{12n}\,d_{13n}\,q_n}{\mathcal D_n},\qquad
  \widetilde p_n=\frac{C_n\,d_{12n}\,d_{13n}\,p_n}{\mathcal D_n}
\]
are integers, and $|\widetilde q_n|\le e^{(h+o(1))n}$ with $h<49.0986393$.
\end{lemma}

\begin{proof}
The sequence $C_nq_n$ grows at rate less than $38.168733$, and $d_{12n}d_{13n}$ contributes $25$.
By the prime number theorem, $\mathcal D^{\mathrm{car}}_n$ has growth rate $\gamma=\int_{1/13}^\infty\phi(x)\,x^{-2}\,dx>12.128427$ and $\mathcal D^{\mathrm{sh}}_n$ has growth rate $\frac{233}{120}$, the total length of the six shells divided by $n$ (Appendix~\ref{app:z2-blocks}). Hence
\[
  h<38.168733+25-\Bigl(12.128427+\frac{233}{120}\Bigr)< 49.098639333.
\]
\end{proof}
\paragraph{Proof of Theorem~\ref{thm:z23-measures}.}

We now conclude with a proof of the measure bounds.
\begin{proof}[Proof of Theorem~\ref{thm:z23-measures}]
\emph{The case $\zeta(2)$.} Consider the approximations $r_n=p_n/q_n$,
with the integer representatives
\[
  \widetilde q_n=\frac{C_nd_{12n}d_{13n}q_n}{D^{\mathrm{car}}_nD^{\mathrm{sh}}_n},\qquad
  \widetilde p_n=\frac{C_nd_{12n}d_{13n}p_n}{D^{\mathrm{car}}_nD^{\mathrm{sh}}_n}.
\]
The four claims hold:
\begin{enumerate}
\item the limit is $\zeta(2)$ (Lemma~\ref{lem:z2-limits}).
\item the rate of convergence is $g>62.636575$
  (Lemma~\ref{lem:z2-selected-error}).
\item the representatives are integers, and they grow at rate
  $h<38.168733+10.929906=49.098639$
  (Lemmas~\ref{prop:zeta2-int} and~\ref{lem:z2-growth}).
\item The errors can not vanish since $\zeta(2)$ is irrational.
\end{enumerate}
The approximation exponent (Definition~\ref{def:delta}) is
$\delta=\frac gh-1>0.2757$, and therefore
\[
  \mu\bigl(\zeta(2)\bigr)\le1+\frac1\delta=\frac{g}{g-h}
  <\frac{62.636575}{13.537936}<4.62675 .
\]

\emph{The case $\zeta(3)$.} Take $A=1$, $B=0$ in the construction of
Section~\ref{subsec:z34}, and consider the approximations
$r_n=p_{3,n}/q_n=P_{3,n}/Q_n$, with the integer representatives
\[
  \widetilde q_n=\frac{6\,C'_n\,d_{19n}^3\,Q_n}{\mathcal D'_n},\qquad
  \widetilde p_n=\frac{6\,C'_n\,d_{19n}^3\,P_{3,n}}{\mathcal D'_n},
\]
where $\mathcal D'_n$ is the common divisor of
Appendix~\ref{app:zeta3-measure}. The four claims hold:
\begin{enumerate}
\item the limit is $\zeta(3)$ (Lemma~\ref{lem:z34-limits}).
\item the rate of convergence is $g>112.8075$
  (Lemma~\ref{lem:z34-selected-error}).
\item the representatives are integers, and they grow at rate
  $h<53.287584+30\log\frac32+57-30.458547< 91.99299103$
  (Appendix~\ref{app:zeta3-measure}).
\item The errors can not vanish since $\zeta(3)$ is irrational.
\end{enumerate}
The approximation exponent is $\delta=\frac gh-1>0.2262$, and therefore
\[
  \mu\bigl(\zeta(3)\bigr)\le1+\frac1\delta=\frac{g}{g-h}
  <\frac{112.8075}{20.81451}<5.4197 .
\qedhere
\]
\end{proof}
\section{Additional CMF Implications: Harmonic-Weighted Gregory Series}
\label{sec:other_implications}

So far, we have shown how CMFs unify prior irrationality results within a single framework and leveraged them to generate high quality Diophantine approximation sequences for mathematical constants.
In this section, we demonstrate another category of CMF applications to irrationality proofs: establishing the irrationality of infinite families of values 
- specifically the harmonic-weighted Gregory series $B(M)$
(see Remark~\ref{rem:hh-other-families} below about extensions to other families) 
\begin{equation}\label{eq:eq:hh-series}
\begin{aligned}
 B(M)&=\sum_{n=0}^{\infty}\frac{(-1)^nO_n}{(2n+1)(M-1)^n},
\end{aligned}
\end{equation}
with $O_n=\sum_{m=0}^n(2m+1)^{-1}$ being the odd harmonic number.

This series is Gregory's arctangent series \cite{turnbull1939gregory} $\sqrt{M-1}\arctan(1/\sqrt{M-1})=\sum_n(-1)^n/((2n+1)(M-1)^n)$ weighted by the odd harmonic numbers.
Sums with harmonic weights go back to Euler \cite{euler1776meditationes}. Their alternating, odd-index versions were evaluated by 
Flajolet and Salvy \cite{flajolet1998euler} in terms of zeta values, Catalan's constant $G$ and $\pi\log2$.
$B(M)$ values are connected to various constants, for example $B(2)=\frac12G+\frac\pi8\log2$ and $B(4)=\frac98L(2,\chi_{-3})$
(this value was recently proven irrational by Calegari, Dimitrov and Tang \cite{calegari2024linear}).

Gregory's series also has a place in the history of irrationality proofs. At rational points $0<|r|\le1$ it gives the values $\arctan r$, and Lambert's continued fraction for $\tan x$ shows that $\tan x$ is irrational for every rational $x\ne0$ \cite{lambert1768}. Hence $\arctan r$ is irrational for every rational $r\ne0$; in particular $\pi=4\arctan1$ is irrational. This was the first proof of the irrationality of $\pi$.

The connection of $B(M)$ to hypergeometric functions is found by substituting $z=1/M$, as
\begin{equation}\label{eq:hh-forms}
\begin{aligned}
 B(M)&=\sqrt{1-z}\;{}_3F_2\!\left(\begin{matrix}\frac12,\frac12,\frac12\\
 \frac32,\frac32\end{matrix};z\right)
 =\sqrt{M-1}\int_0^{\arctan(1/\sqrt{M-1})}u\cot u\,du.
\end{aligned}
\end{equation}
The integral is the classical inverse-sine integral \cite{lewin1981polylogarithms}.

\begin{theorem}\label{thm:hh-main}

Let $M\ge7165$ be an integer. The CMF data
\[
 \boxed{\ \mathcal M={}_3F_2(1/M),\qquad
 s_0=\bigl(\tfrac12,\tfrac12,\tfrac12;\tfrac32,\tfrac32\bigr),\qquad
 t_0=(2,2,2;3,3)\ }
\]
proves that $B(M)$ has $\delta>0$, and is therefore irrational.
\end{theorem}

\begin{remark}[Comparison with general criteria]\label{rem:hh-comparison}
The $G$-function criterion of Fischler and Rivoal \cite{fischler2018rational} gives $B(M)\notin\mathbb Q$ for $M-1>2^{2412}\approx10^{726}$, and Kawashima's preprint \cite{kawashima2026linear} proposes a polylogarithm criterion - which gives $M>2^{26}e^{16}\approx5.96\cdot10^{14}$ (Theorems~\ref{thm:hh-fr} and~\ref{thm:hh-kaw} derive these bounds for $B(M)$). 
Theorem~\ref{thm:hh-main} lowers the threshold by roughly $722$ and $11$ orders of magnitude, respectively, and gives explicit approximating sequences. Experimentally, $M$ divisible by $2$ or $4$ appear to admit even smaller thresholds (Figure~\ref{fig:hh-exp}).
\end{remark}

\subsection{Proof of Theorem~\ref{thm:hh-main}}\label{subsec:hh-proof}

We follow the CMF proof scheme of Subsection~\ref{cmf proof stratergy}. Here the dual pairing (Definition~\ref{def:dual-pairing}) is trivial: $K(n)=I$ and $c(n)=1$, so the approximants are CMF ratios of the dual walk of $(\mathcal M,s_0,t_0)$ itself,
\[
 r_n=\frac{\rho_p\Phi^\vee(n)e_0}{\rho_q\Phi^\vee(n)e_0},\qquad \rho_p=(1,0,0),\quad \rho_q=(1,4,4).
\]
\begin{enumerate}
 \item \emph{Identifying the constant.} $r_n\to B(M)$ (Lemma~\ref{lem:hh-mobius}, Proposition~\ref{prop:hh-rate}).
 \item \emph{The rate of convergence.} $|B(M)-r_n|\le e^{-g(M)n+o(n)}$, from the eigenvalues of the limiting matrix (Proposition~\ref{prop:hh-rate}).
 \item \emph{Controlling the arithmetic growth.} A closed form for the approximants, as Hermite--Pad\'e approximants with a Rodrigues formula (Lemma~\ref{lem:dual-pade}, Proposition~\ref{prop:hh-rodrigues}), gives $r_n=A_n/D_n$ with integers and $|D_n|\le e^{\sigma(M)n+o(n)}$ (Lemma~\ref{lem:hh-integers}).
 \item \emph{Irrationality exponent.} $g(M)>\sigma(M)$ for $M\ge7165$ (Proposition~\ref{prop:hh-gap}), and the approximations are nontrivial by Lemma~\ref{lem:nonvanishing} (Lemma~\ref{lem:hh-nondeg}); Paragraph~\ref{par: delta} then gives $\delta\ge g/\sigma-1>0$.
\end{enumerate}

Proposition~\ref{prop:hh-gap} shows $g(M)>\sigma(M)$ for $M\ge7165$. Computations specific to this function and trajectory are in Appendix~\ref{app:gregory-series}.

Put $k=M-1$.

Let $T(n)=\mathcal M_{t_0}(s_0+n\,t_0)$ be the trajectory matrix (Definition~\ref{def:trajectory}), given explicitly in \eqref{eq:hh-T}, so that $\Phi^\vee(n)=(T(0)\cdots T(n-1))^{-\mathsf T}$. Along the trajectory we define
\[
  F_n={}_3F_2(s_0+n\,t_0;\,z),\qquad S_n=(F_n,\theta F_n,\theta^2F_n),\qquad S_{n+1}=S_nT(n).
\]
We first identify the approximants. The following lemma applies to any trajectory whose matrix, after scaling by a power of $z$, is polynomial in $z$.

\begin{lemma}[CMF ratios of the dual walk are Hermite--Pad\'e approximants]\label{lem:dual-pade}
Let $r \geq 2$ and $d, m \geq 0$ be integers, and let $T(n)\in\operatorname{GL}_r(\mathbb Q(z))$ satisfy $z^dT(n)\in\operatorname{Mat}_r(\mathbb Q[z])$ with entries of degree $\le m$ in $z$, and let $S_n=(S_{n,0},\dots,S_{n,r-1})$ be rows of formal power series with $S_{n+1}=S_nT(n)$ and $S_{n,j}\in z\mathbb Q[[z]]$ for $j\ge1$. Put $W_n=\prod_{i<n}z^dT(i)$ and
\[
 \hat x_n=\operatorname{adj}(W_n)^{\mathsf T}e_0=z^{-dn}\det(W_n)\,\Phi_T^\vee(n)e_0.
\]
Then $\hat x_n$ has polynomial entries of degree $\le(r-1)mn$, and for all rational rows $p,q$,
\[
 (q\hat x_n)(pS_0^{\mathsf T})-(p\hat x_n)(qS_0^{\mathsf T})=O(z^{dn+1}).
\]
If moreover $(r-1)m\le d$ and $qS_0^{\mathsf T}$ has a nonzero constant term, then $q\hat x_n\not\equiv0$.
\end{lemma}
\begin{proof}
The entries of $\hat x_n$ are $(r-1)\times(r-1)$ minors of $W_n$, whose entries have degree $\le mn$.
Write $\wedge$ for the exterior product and $\Lambda^2$ for the induced action of a matrix on $\Lambda^2\mathbb Q((z))^r$. By the Lagrange identity, the left-hand side equals $(q\wedge p)\cdot(\hat x_n\wedge S_0^{\mathsf T})$.
Since $S_0W_n=z^{dn}S_n$, we have $S_0^{\mathsf T}=z^{dn}W_n^{-\mathsf T}S_n^{\mathsf T}$, hence
\[
 \hat x_n\wedge S_0^{\mathsf T}=z^{dn}\det(W_n)\,\Lambda^2\bigl(W_n^{-\mathsf T}\bigr)\bigl(e_0\wedge S_n^{\mathsf T}\bigr).
\]
By Jacobi's complementary-minor theorem, the entries of $\det(W_n)\Lambda^2(W_n^{-\mathsf T})$ are, up to sign, $(r-2)\times(r-2)$ minors of $W_n$, hence polynomials. The entries of $e_0\wedge S_n^{\mathsf T}$ are $\pm S_{n,j}$ with $j\ge1$, which vanish at $z=0$.
For the last claim, suppose $q\hat x_n\equiv0$. Then each $p\hat x_n$ has degree $<dn+1$ and satisfies $(p\hat x_n)(qS_0^{\mathsf T})=O(z^{dn+1})$, so $p\hat x_n=0$. Hence $\hat x_n=0$, contradicting $\det W_n\ne0$.
\end{proof}

For our trajectory, \eqref{eq:hh-T} gives $r=3$, $d=3$ and $m=1$, \eqref{eq:hh-detT} gives $\det T(n)\ne0$, and $S_{n,1},S_{n,2}$ vanish at $z=0$.
Let $a=\rho_q\hat x_n$ and $f=\rho_p\hat x_n$ (both depend on $n$, which we omit from the notation), so that $r_n=f/a$ at $z=1/M$.

Since $\rho_qS_0^{\mathsf T}$ has constant term $1$, Lemma~\ref{lem:dual-pade} with $q=\rho_q$ shows that $a,f\in\mathbb Q[z]$ have degree $\le2n$, that $a\not\equiv0$, and that
\begin{equation}\label{eq:hh-pade}
 a\,(pS_0^{\mathsf T})-(p\hat x_n)\,(\rho_qS_0^{\mathsf T})=O(z^{3n+1})\qquad\text{for }p=\rho_p\text{ and }p=(1,2,0).
\end{equation}

\begin{lemma}\label{lem:hh-mobius}
Let $x=z/(z-1)$ and $\ell(u)=\log2-\frac12\log(1-u)$. For $-1\le x\le0$,
\[
 \frac{\rho_pS_0^{\mathsf T}}{\rho_qS_0^{\mathsf T}}=f_2(x):=\sum_{n\ge0}\frac{O_n}{2n+1}x^n=\int_0^1\frac{\ell(u)\,du}{2\sqrt u\,(1-xu)},
\]
\[
 \frac{(1,2,0)S_0^{\mathsf T}}{\rho_qS_0^{\mathsf T}}=f_1(x):=\sum_{n\ge0}\frac{x^n}{2n+1}=\int_0^1\frac{du}{2\sqrt u\,(1-xu)}.
\]
In particular, $B(M)=f_2(-1/k)$.
\end{lemma}

Put $\tilde a(x)=(1-x)^{2n}a(\frac x{x-1})$, $\tilde f(x)=(1-x)^{2n}f(\frac x{x-1})$ (degree $\le2n$, $\tilde f/\tilde a=f/a$) and
$Q(u)=u^{2n}\tilde a(1/u)$. By \eqref{eq:hh-pade}, the coefficients of $x^{2n+1},\dots,x^{3n}$ in $\tilde af_i$
vanish, and $\tilde f$ is the truncation of $\tilde af_2$ to degree $\le2n$. In terms of $Q$, for $0\le j<n$,
\begin{equation}\label{eq:hh-orth}
 \int_0^1Q(u)\,u^{j+1/2}\,du=0,\qquad\int_0^1Q(u)\,u^{j+1/2}\ell(u)\,du=0 .
\end{equation}
Thus $Q$ satisfies the type II multiple orthogonality conditions, and $\tilde f/\tilde a$ is the $f_2$ component of the simultaneous Hermite--Pad\'e approximation to $(f_1,f_2)$.

The relations \eqref{eq:hh-orth} determine $Q$ up to a constant factor. The next proposition gives $Q$ in closed form by a Rodrigues formula - which makes the integer normalization of Lemma~\ref{lem:hh-integers} possible.
\begin{proposition}
\label{prop:hh-rodrigues}
The normalization of $\hat x_n$ can be chosen so that $Q=Q_n$, where
\[
 Q_n(u)=\frac{u^{-1/2}}{n!^2}\Bigl(\frac d{du}\Bigr)^n\Bigl[(1-u)^n\Bigl(\frac d{du}\Bigr)^n\bigl[u^{2n+\frac12}(1-u)^n\bigr]\Bigr].
\]
With this normalization, write $Q_n=\sum_{l=0}^{2n}q_lu^l$. Then
\[
 \tilde a=x^{2n}Q_n(1/x),\qquad \tilde f=\sum_lq_lx^{2n-l}\sum_{t\le l}\frac{O_t}{2t+1}x^t .
\]
\end{proposition}%
\noindent The proofs of Lemma~\ref{lem:hh-mobius} and Proposition~\ref{prop:hh-rodrigues} are in Appendix~\ref{app:hh-approx}.

Next, we use the asymptotics of the trajectory. After a diagonal coboundary transformation (Definition~\ref{def:coboundary}), the trajectory matrix has a limit $\tilde T_\infty$, given in \eqref{eq:hh-Tinf}. Its eigenvalues are real, positive and distinct; let $g(M)$ be the logarithm of the ratio of the two smallest.

\begin{proposition}\label{prop:hh-rate}
$S_n$ is recessive, $e_0$ is nonexceptional in the sense of Theorem~\ref{thm:dual-convergence} 
and
$r_n\to B(M)$ and $|B(M)-r_n|\le e^{-g(M)n+o(n)}$.
\end{proposition}
The proof is in Appendix~\ref{app:hh-rate}.

Next, we normalize the approximants to integers.
\begin{lemma}\label{lem:hh-integers}
Let $L_m=\lcm(1,3,\dots,m)$ for odd $m$. Then
$D_n=16^nL_{4n+1}^2\sum_lq_l(-1)^lk^l$ and $A_n=16^nL_{4n+1}^2k^{2n}\tilde f(-1/k)$ are integers,
$D_n>0$, $r_n=A_n/D_n$, and
\[
 \limsup_{n\to\infty}\frac{\log|D_n|}n\le\sigma(M):=2\log M+\log16+2\log\tfrac{27}4+8.
\]
\end{lemma}
The proof is in Appendix~\ref{app:hh-approx}.

\begin{lemma}\label{lem:hh-nondeg}
$D_nB(M)\ne A_n$ for infinitely many $n$.
\end{lemma}

This follows from Lemma~\ref{lem:nonvanishing}; see Appendix~\ref{app:hh-arith}.

\noindent The final step compares the error rate with the denominator growth.

\begin{proposition}\label{prop:hh-gap}
As $M\to\infty$,
\[
 g(M)-\sigma(M)=\log M-8-\log\tfrac{64}{27}+O(M^{-1/2}),
\]
and $g(M)>\sigma(M)$ for every integer $M\ge7165$.
\end{proposition}
The term $\log M$ reflects that the error is of order $M^{-3n}$ while $D_n$ grows like $M^{2n}$. The constant $8$ comes from the factor $L_{4n+1}^2$ by the prime number theorem, and $\log\frac{64}{27}$ collects the remaining constants. The threshold is therefore close to $\frac{64}{27}e^8\approx7066$; the value $7165$ comes from explicit bounds for the two eigenvalues (Lemma~\ref{lem:hh-roots}). The proof is in Appendix~\ref{app:hh-arith}.

\begin{proof}[Proof of Theorem~\ref{thm:hh-main}]
Let $M\ge7165$. By Proposition~\ref{prop:hh-gap}, $g>\sigma$. Lemma~\ref{lem:hh-integers} and Proposition~\ref{prop:hh-rate} give
$|D_nB-A_n|\le e^{(\sigma-g)n+o(n)}\to0$, and Lemma~\ref{lem:hh-nondeg} shows that this quantity is nonzero for infinitely
many $n$. If $B(M)=p/q$, then $q(D_nB-A_n)$ would be a nonzero integer for infinitely many $n$, a contradiction.
The bound $\delta_{(r_n)}\ge g/\sigma-1>0$ is derived in Appendix~\ref{app:hh-arith}.
\end{proof}

\begin{remark}[Further application]\label{rem:hh-other-families}
The proof of Theorem~\ref{thm:hh-main} applies to other families of constants.
Let the trajectory matrix of a CMF trajectory depend on a parameter $z$, and let the CMF ratios of its dual walk converge to $\alpha(z)$.
At $z=1/M$, suppose that
(i) the CMF ratios can be written as $A_n/D_n$ with integers $|D_n|\le e^{\sigma(M)n+o(n)}$, and
(ii) the trajectory matrix has a limit with eigenvalues of distinct moduli, the compatible state is recessive, and the terminal column is a cyclic vector of the limit.
Then Theorem~\ref{thm:dual-convergence}%
, with the argument of Appendix~\ref{app:hh-rate},
gives $|\alpha(1/M)-r_n|\le e^{-g(M)n+o(n)}$ and Lemma~\ref{lem:nonvanishing} gives $D_n\alpha(1/M)\ne A_n$ infinitely often, so $\alpha(1/M)\notin\mathbb Q$ whenever $g(M)>\sigma(M)$.
Condition~(ii) is a finite computation with the limiting matrix.
Lemma~\ref{lem:dual-pade} indicates when $g(M)>\sigma(M)$ for all large $M$: if $z^dT(n)$ has entries of degree at most $m$ in $z$ and $d>(r-1)m$, the error decays like $M^{-dn}$ while the denominators grow like $M^{(r-1)mn}$.
\end{remark}

\section{Discussion and Outlook}
\label{sec:outlook}


\subsection{Improving the Method}

The research strategy at the core of this work is inspired by experimental mathematics: we first connect a constant of interest to a D-finite function, construct a CMF using its contiguous relations, and then optimize the irrationality measure $\delta$ along trajectories in its lattice.

A natural direction could be to construct CMFs for more D-finite functions, such as Kampé de Fériet and A-hypergeometric systems, enabling the exploration of a much broader trajectory spaces.

\subsection{Higher-Order Irrationality Results}

The results of this paper, together with recent breakthroughs based on arithmetic holonomy, determinant methods, and AI-assisted mathematical discovery \cite{calegari2024linear, lai2026zeta34, openai2026catalan}, point toward a broader challenge: establishing irrationality results for higher-order constants.

The \(G_{d,d}^{d,2}(1)\) CMF encodes \(\zeta(3),\zeta(4),\dots\zeta(d-1)\) simultaneously. Could the next irrationality result be discovered as a trajectory in it?

\subsection{Relation of Experimental Math to "AI for Math"}

Before LLMs started writing major proofs several months ago \cite{openai2026planar} computers had very different uses in mathematical research. 
A primary use case, "experimental math", was also termed AI in many venues (in retrospect, confusingly).
This area of math and type of mathematical research are now completely disjoined from modern "AI for math".

There is in fact a key difference between experimental math and modern AI for math: Computer assistance in experimental math has always been used to enhance overall human understanding, by supporting it with examples in the technically demanding parts. 
In contrast, recent LLM achievements are (so far, mostly) replacing human understanding.

While it is too early to tell how LLM-use for math will evolve in the near future, we hope that the work here will help encourage computer usage and reporting that enhance human understanding, promoting responsible use of AI in mathematics.

In particular, recent AI discoveries included important achievements in Diophantine approximations \cite{openai2026catalan,fauzan2026}, including a direct competition to our work by an LLM-discovery honestly reported by Lai \cite{lai2026zeta34}.
Our work differs exactly in the important aspect of human understanding: The approach we developed over the past several years used experimental math to help unify decades of human effort into a compact framework that supports understanding through a simple representation.

Looking forward, our CMF approach does not capture certain key results \cite{fauzan2026, lai2026zeta34, openai2026catalan, calegari2024linear}.
We hope that near-future advances will help connect these efforts to a unified, more understandable framework. 
More broadly, attempts at such unifications and at creating simpler representations may help contribute a little to the broader effort to align AI with the goal of enhancing human understanding.

\subsection*{Use of AI Tools}
The mathematical concepts, CMF framework, and proof methods presented in this paper were developed by the authors over several years. 
In particular, the simultaneous Ap\'ery-type approximation construction and the proof of Theorem~\ref{thm:intro-zeta34} were discovered and written by the authors, independently of LLM-generated mathematical arguments.

AI tools assisted in developing code for the symbolic-computation software used for numerical experiments. 
After the authors had established the proof strategy and demonstrated it in Theorem~\ref{thm:intro-zeta34}, AI assistance contributed to the most technical parts 
of subsequent proofs:
the arithmetic-growth estimate in Lemma~\ref{lem:z2-growth}, supporting Theorem~\ref{thm:intro-zeta23}, and technical arguments in Appendix~\ref{app:gregory-series}, supporting the irrationality of 
harmonic-weighted Gregory series in Section~\ref{sec:other_implications}. 
These arguments were subsequently reconstructed, verified, and written in full by the authors.

AI tools also assisted in the formal verification of the two main theorems in Lean. 
The authors take full responsibility for all mathematical arguments, computations, and citations.

\section*{Code and data availability}\phantomsection\label{sec:code}
All computations used in the proofs are available at \coderepo{}
(release v1.0.0). Each notebook runs independently and states which
lemma it verifies.
\begin{itemize}
\item Theorem~\ref{thm:z34} (Section~\ref{subsec:z34}, Appendix~\ref{app:zeta34-proof}): \texttt{python/joint/}.
\item The $\zeta(2)$ bound (Section~\ref{subsec:z23}, Appendices~\ref{app:zeta2-proof}--\ref{app:z2-blocks}): \texttt{python/zeta2/}.
\item The $\zeta(3)$ bound (Appendix~\ref{app:zeta3-measure}): \texttt{python/pure\_zeta3/}.
\item The integral CMF for $\pi$ and the harmonic-weighted Gregory series: \texttt{additional/}.
\item Lean~4 sources, which assume the prime number theorem: \texttt{lean/}.
\end{itemize}

\section*{Acknowledgments}

We thank the members of the Ramanujan Machine team, past and present.
This work grew out of several years of joint investigation of
conservative matrix fields.  We are grateful for the discussions,
computations, experiments, and questions that shaped the point of view
developed here.

The Ramanujan Machine team gratefully acknowledges support from
Schmidt Sciences, LLC.  The first author began this work as a member
of the team and is supported by the FWO and the F.R.S.--FNRS under
the Excellence of Science (EOS) programme, project ID 40007542.

\printbibliography

@misc{NIST:DLMF,
  sortkey = {DLMF},
  title = {{NIST Digital Library of Mathematical Functions}},
  year = {2026},
  url = {https://dlmf.nist.gov/},
  note = {Release 1.2.8 of 2026-09-15. F.~W.~J. Olver, A.~B. Olde Daalhuis, D.~W. Lozier, B.~I. Schneider, R.~F. Boisvert, C.~W. Clark, B.~R. Miller, B.~V. Saunders, H.~S. Cohl, and M.~A. McClain, eds.}
}

@incollection{apery1979irrationalite,
  author = {Ap{\'e}ry, Roger},
  title = {Irrationalit{\'e} de {$\zeta(2)$} et {$\zeta(3)$}},
  booktitle = {Journ{\'e}es Arithm{\'e}tiques de Luminy},
  series = {Ast{\'e}risque},
  volume = {61},
  pages = {11--13},
  publisher = {Soci{\'e}t{\'e} Math{\'e}matique de France},
  year = {1979},
  url = {https://www.numdam.org/item/AST_1979__61__11_0/}
}

@misc{bai2026pi,
  author = {Bai, Yufei},
  title = {The irrationality measure of {$\pi$} is at most {$7.101862832357$}},
  year = {2026},
  eprint = {2609.11276},
  archivePrefix = {arXiv},
  primaryClass = {math.NT},
  url = {https://arxiv.org/abs/2609.11276},
  note = {Preprint, version 2, revised 11 September 2026}
}

@article{beukers1979note,
  author = {Beukers, Frits},
  title = {A note on the irrationality of {$\zeta(2)$} and {$\zeta(3)$}},
  journal = {Bulletin of the London Mathematical Society},
  volume = {11},
  number = {3},
  pages = {268--272},
  year = {1979},
  doi = {10.1112/blms/11.3.268}
}

@misc{brownZudilin2022zeta5,
  author = {Brown, Francis and Zudilin, Wadim},
  title = {On cellular rational approximations to {$\zeta(5)$}},
  year = {2022},
  eprint = {2210.03391},
  archivePrefix = {arXiv},
  primaryClass = {math.NT},
  url = {https://arxiv.org/abs/2210.03391},
  note = {Preprint, version 3, revised 29 January 2026}
}

@misc{calegari2024linear,
  author = {Calegari, Frank and Dimitrov, Vesselin and Tang, Yunqing},
  title = {The linear independence of {$1$, $\zeta(2)$, and $L(2,\chi_{-3})$}},
  year = {2024},
  eprint = {2408.15403},
  archivePrefix = {arXiv},
  primaryClass = {math.NT},
  url = {https://arxiv.org/abs/2408.15403},
  note = {Preprint, version 2, revised 16 September 2024}
}

@misc{calegariDimitrovTang2025Holonomy,
  author = {Calegari, Frank and Dimitrov, Vesselin and Tang, Yunqing},
  title = {Arithmetic holonomy bounds and effective {Diophantine} approximation},
  year = {2025},
  eprint = {2510.04156},
  archivePrefix = {arXiv},
  primaryClass = {math.NT},
  note = {Preprint},
  url = {https://arxiv.org/abs/2510.04156}
}

@misc{cohenZudilin2025variations,
  author = {Cohen, Henri and Zudilin, Wadim},
  title = {Variations on a theme of {Ap{\'e}ry}},
  year = {2025},
  eprint = {2501.10090},
  archivePrefix = {arXiv},
  primaryClass = {math.NT},
  url = {https://arxiv.org/abs/2501.10090},
  note = {Version 3, revised 3 November 2025}
}

@article{dauguetZudilin2014,
  author = {Dauguet, Simon and Zudilin, Wadim},
  title = {On simultaneous diophantine approximations to {$\zeta(2)$} and {$\zeta(3)$}},
  journal = {Journal of Number Theory},
  volume = {145},
  pages = {362--387},
  year = {2014},
  doi = {10.1016/j.jnt.2014.06.007},
  eprint = {1401.5322},
  archivePrefix = {arXiv},
  primaryClass = {math.NT}
}

@article{elimelech2024algorithm,
  author = {Elimelech, Rotem and David, Ofir and De la Cruz Mengual, Carlos and Kalisch, Rotem and Berndt, Wolfgang and Shalyt, Michael and Silberstein, Mark and Hadad, Yaron and Kaminer, Ido},
  title = {Algorithm-assisted discovery of an intrinsic order among mathematical constants},
  journal = {Proc. Natl. Acad. Sci. USA},
  volume = {121},
  number = {25},
  year = {2024},
  pages = {e2321440121},
  doi = {10.1073/pnas.2321440121}
}

@misc{eskandari2026catalan,
  author = {Eskandari, Payman},
  title = {Improved rational approximations to {Catalan}'s constant},
  year = {2026},
  eprint = {2609.26354},
  archivePrefix = {arXiv},
  primaryClass = {math.NT},
  url = {https://arxiv.org/abs/2609.26354},
  note = {Preprint, version 1, 22 September 2026}
}

@incollection{fischler2004bourbaki,
  author = {Fischler, St{\'e}phane},
  title = {Irrationalit{\'e} de valeurs de z{\^e}ta (d'apr{\`e}s {Ap{\'e}ry}, {Rivoal}, \ldots)},
  booktitle = {S{\'e}minaire Bourbaki, Vol. 2002/2003},
  series = {Ast{\'e}risque},
  volume = {294},
  pages = {27--62},
  note = {Exp. no. 910},
  publisher = {Soci{\'e}t{\'e} Math{\'e}matique de France},
  year = {2004},
  eprint = {math/0303066},
  archivePrefix = {arXiv},
  url = {https://www.numdam.org/item/SB_2002-2003__45__27_0/}
}

@article{gautschi1967computational,
  author = {Gautschi, Walter},
  title = {Computational aspects of three-term recurrence relations},
  journal = {SIAM Review},
  volume = {9},
  number = {1},
  pages = {24--82},
  year = {1967},
  doi = {10.1137/1009002}
}

@misc{giraudin2026zeta2,
  author = {Niedbala Giraudin, David},
  title = {A new upper bound for the irrationality exponent of {$\zeta(2)$}},
  year = {2026},
  eprint = {2610.02912},
  archivePrefix = {arXiv},
  primaryClass = {math.NT},
  url = {https://arxiv.org/abs/2610.02912},
  note = {Preprint, version 1, 2 October 2026}
}

@misc{giraudin2026zeta3,
  author = {Niedbala Giraudin, David},
  title = {A new upper bound for the irrationality exponent of {$\zeta(3)$}},
  year = {2026},
  eprint = {2609.26980},
  archivePrefix = {arXiv},
  primaryClass = {math.NT},
  url = {https://arxiv.org/abs/2609.26980},
  note = {Preprint, version 1, 22 September 2026}
}

@article{hata1993pi,
  author = {Hata, Masayoshi},
  title = {Rational approximations to {$\pi$} and some other numbers},
  journal = {Acta Arithmetica},
  volume = {63},
  number = {4},
  pages = {335--349},
  year = {1993},
  doi = {10.4064/aa-63-4-335-349}
}

@article{hata2000zeta3,
  author = {Hata, Masayoshi},
  title = {A new irrationality measure for {$\zeta(3)$}},
  journal = {Acta Arithmetica},
  volume = {92},
  number = {1},
  pages = {47--57},
  year = {2000},
  doi = {10.4064/aa-92-1-47-57}
}

@article{hendriksSinger1999,
  author = {Hendriks, Peter A. and Singer, Michael F.},
  title = {Solving difference equations in finite terms},
  journal = {Journal of Symbolic Computation},
  volume = {27},
  number = {3},
  pages = {239--259},
  year = {1999},
  doi = {10.1006/jsco.1998.0251},
  url = {https://singer.math.ncsu.edu/papers/finite.pdf}
}

@misc{lai2026zeta34,
  author = {Lai, Li},
  title = {The linear independence of {$1$, $\zeta(3)$, and $\zeta(4)$}},
  year = {2026},
  eprint = {2610.06838},
  archivePrefix = {arXiv},
  primaryClass = {math.NT},
  url = {https://arxiv.org/abs/2610.06838},
  note = {Preprint, version 1, 5 October 2026}
}

@article{marcovecchio2009log2,
  author = {Marcovecchio, Raffaele},
  title = {The {Rhin--Viola} method for {$\log 2$}},
  journal = {Acta Arithmetica},
  volume = {139},
  number = {2},
  pages = {147--184},
  year = {2009},
  doi = {10.4064/aa139-2-5}
}

@article{nesterenko2016catalan,
  author = {Nesterenko, Yuri V.},
  title = {On {Catalan}'s constant},
  journal = {Proceedings of the Steklov Institute of Mathematics},
  volume = {292},
  number = {1},
  pages = {153--170},
  year = {2016},
  doi = {10.1134/S0081543816010107}
}

@misc{ramanujantools,
  author = {{Ramanujan Machine Group}},
  title = {{ramanujantools}: Symbolic and numeric research tools},
  year = {2025},
  howpublished = {Software},
  url = {https://github.com/RamanujanMachine/ramanujantools}
}

@inproceedings{raz2025euler,
  author = {Raz, Tomer and Shalyt, Michael and Leibtag, Elyasheev and Kalisch, Rotem and Weinbaum, Shachar and Hadad, Yaron and Kaminer, Ido},
  title = {From {Euler} to {AI}: Unifying formulas for mathematical constants},
  booktitle = {Advances in Neural Information Processing Systems},
  volume = {38},
  year = {2025},
  pages = {137390--137471},
  doi = {10.52202/085713-4132},
  url = {https://proceedings.neurips.cc/paper_files/paper/2025/hash/b394adff21327520268aaf565b63414c-Abstract-Conference.html}
}

@article{rhin1996permutation,
  author = {Rhin, Georges and Viola, Carlo},
  title = {On a permutation group related to {$\zeta(2)$}},
  journal = {Acta Arithmetica},
  volume = {77},
  number = {1},
  pages = {23--56},
  year = {1996},
  doi = {10.4064/aa-77-1-23-56}
}

@incollection{dvornicichViola1990,
  author = {Dvornicich, Roberto and Viola, Carlo},
  title = {Some remarks on {Beukers}' integrals},
  booktitle = {Number Theory, Vol. II (Budapest, 1987)},
  series = {Colloquia Mathematica Societatis J\'anos Bolyai},
  volume = {51},
  pages = {637--657},
  publisher = {North-Holland},
  address = {Amsterdam},
  year = {1990}
}

@article{alladiRobinson1980,
  author = {Alladi, Krishnaswami and Robinson, M. L.},
  title = {Legendre polynomials and irrationality},
  journal = {Journal f\"ur die reine und angewandte Mathematik},
  volume = {318},
  pages = {137--155},
  year = {1980}
}

@article{rukhadze1987log2,
  author = {Rukhadze, E. A.},
  title = {A lower bound for the approximation of {$\ln 2$} by rational numbers},
  journal = {Vestnik Moskovskogo Universiteta. Seriya 1. Matematika. Mekhanika},
  number = {6},
  pages = {25--29},
  year = {1987},
  note = {In Russian}
}

@incollection{mignotte1974pi,
  author = {Mignotte, Maurice},
  title = {Approximations rationnelles de {$\pi$} et quelques autres nombres},
  booktitle = {Journ\'ees arithm\'etiques (Grenoble, 1973)},
  series = {Bulletin de la Soci\'et\'e math\'ematique de France. M\'emoire},
  number = {37},
  pages = {121--132},
  year = {1974},
  doi = {10.24033/msmf.139}
}

@article{rhinViola2001,
  author = {Rhin, Georges and Viola, Carlo},
  title = {The group structure for {$\zeta(3)$}},
  journal = {Acta Arithmetica},
  volume = {97},
  number = {3},
  pages = {269--293},
  year = {2001},
  doi = {10.4064/aa97-3-6}
}

@article{salikhov2010pi,
  author = {Salikhov, Vladislav Kh.},
  title = {On the measure of irrationality of the number {$\pi$}},
  journal = {Mathematical Notes},
  volume = {88},
  pages = {563--573},
  year = {2010},
  doi = {10.1134/S0001434610090294}
}

@article{van1979proof,
  author = {van der Poorten, Alfred},
  title = {A proof that {Euler} missed \ldots\ {Ap{\'e}ry}'s proof of the irrationality of {$\zeta(3)$}},
  journal = {The Mathematical Intelligencer},
  volume = {1},
  number = {4},
  pages = {195--203},
  year = {1979},
  note = {An informal report},
  doi = {10.1007/BF03028234}
}

@misc{weinbaum2025conservative,
  author = {Weinbaum, Shachar and Leibtag, Elyasheev and Kalisch, Rotem and Shalyt, Michael and Kaminer, Ido},
  title = {On conservative matrix fields: Continuous asymptotics and arithmetic},
  year = {2025},
  eprint = {2507.08138},
  archivePrefix = {arXiv},
  primaryClass = {math.NT},
  url = {https://arxiv.org/abs/2507.08138},
  note = {Preprint, version 2, revised 15 July 2026}
}

@incollection{zagier2009apery,
  author = {Zagier, Don},
  title = {Integral solutions of {Ap{\'e}ry}-like recurrence equations},
  booktitle = {Groups and Symmetries: From the Neolithic Scots to John McKay},
  series = {CRM Proceedings and Lecture Notes},
  volume = {47},
  pages = {349--366},
  publisher = {American Mathematical Society},
  address = {Providence, RI},
  year = {2009},
  url = {https://people.mpim-bonn.mpg.de/zagier/files/tex/AperylikeRecEqs/fulltext.pdf}
}

@article{zeilberger1990holonomic,
  author = {Zeilberger, Doron},
  title = {A holonomic systems approach to special functions identities},
  journal = {J. Comput. Appl. Math.},
  volume = {32},
  number = {3},
  year = {1990},
  pages = {321--368},
  doi = {10.1016/0377-0427(90)90042-X}
}

@article{zeilberger1991method,
  author = {Zeilberger, Doron},
  title = {The method of creative telescoping},
  journal = {J. Symbolic Comput.},
  volume = {11},
  number = {3},
  year = {1991},
  pages = {195--204},
  doi = {10.1016/S0747-7171(08)80044-2}
}

@article{zeilbergerZudilin2020,
  author = {Zeilberger, Doron and Zudilin, Wadim},
  title = {The irrationality measure of {$\pi$} is at most {$7.103205334137\ldots$}},
  journal = {Moscow Journal of Combinatorics and Number Theory},
  volume = {9},
  number = {4},
  pages = {407--419},
  year = {2020},
  doi = {10.2140/moscow.2020.9.407},
  eprint = {1912.06345},
  archivePrefix = {arXiv},
  primaryClass = {math.NT}
}

@article{zudilin2011arithmeticHypergeometric,
  author = {Zudilin, Wadim},
  title = {Arithmetic hypergeometric series},
  journal = {Russian Mathematical Surveys},
  volume = {66},
  number = {2},
  year = {2011},
  pages = {369--420},
  url = {https://www.math.ru.nl/~zudilin/PS/ahs-eng.pdf}
}

@article{zudilin2002arithmetic,
  author = {Zudilin, Wadim},
  title = {Arithmetic of linear forms involving odd zeta values},
  journal = {Journal de th{\'e}orie des nombres de Bordeaux},
  volume = {16},
  number = {1},
  pages = {251--291},
  year = {2004},
  doi = {10.5802/jtnb.447},
  eprint = {math/0206176},
  archivePrefix = {arXiv},
  primaryClass = {math.NT}
}

@article{zudilin2014zeta2,
  author = {Zudilin, Wadim},
  title = {Two hypergeometric tales and a new irrationality measure of {$\zeta(2)$}},
  journal = {Annales math{\'e}matiques du Qu{\'e}bec},
  volume = {38},
  pages = {101--117},
  year = {2014},
  doi = {10.1007/s40316-014-0016-0},
  eprint = {1310.1526},
  archivePrefix = {arXiv},
  primaryClass = {math.NT}
}

@misc{openai2026catalan,
  author = {{OpenAI}},
  title = {{Catalan's constant is irrational}},
  year = {2026},
  month = sep,
  note = {OpenAI Math Release preprint, September 24, 2026},
  url = {https://github.com/openai/math/blob/main/preprints/Catalans-constant-is-irrational-September-24-2026/paper.pdf}
}

@article{viola2022catalan,
  author = {Viola, Carlo},
  title = {Rational approximations to {Catalan}'s constant},
  journal = {Oberwolfach Reports},
  volume = {19},
  number = {2},
  year = {2022},
  pages = {1100--1101},
  doi = {10.4171/OWR/2022/21},
  note = {Report on joint work with Raffaele Marcovecchio},
  url = {https://ems.press/content/serial-article-files/46958}
}

@article{hata1995beukers,
  author = {Hata, Masayoshi},
  title = {A note on {Beukers}' integral},
  journal = {J. Austral. Math. Soc. Ser. A},
  volume = {58},
  number = {2},
  pages = {143--153},
  year = {1995},
  doi = {10.1017/S1446788700038192}
}

@article{hata1992hypergeometric,
  author  = {Hata, Masayoshi},
  title   = {Irrationality measures of the values of hypergeometric functions},
  journal = {Acta Arithmetica},
  volume  = {60},
  number  = {4},
  pages   = {335--347},
  year    = {1992},
  doi     = {10.4064/aa-60-4-335-347}
}

@article{stanley1980differentiably,
  author  = {Stanley, Richard P.},
  title   = {Differentiably Finite Power Series},
  journal = {European Journal of Combinatorics},
  volume  = {1},
  number  = {2},
  pages   = {175--188},
  year    = {1980},
  doi     = {10.1016/S0195-6698(80)80051-5}
}

@article{lipshitz1989d,
  author  = {Lipshitz, Leonard},
  title   = {D-Finite Power Series},
  journal = {Journal of Algebra},
  volume  = {122},
  number  = {2},
  pages   = {353--373},
  year    = {1989},
  doi     = {10.1016/0021-8693(89)90222-6}
}

@article{birkhoff1930formal,
  author  = {Birkhoff, George D.},
  title   = {Formal Theory of Irregular Linear Difference Equations},
  journal = {Acta Mathematica},
  volume  = {54},
  number  = {1},
  pages   = {205--246},
  year    = {1930}
}

@article{birkhoff1933analytic,
  author  = {Birkhoff, George D. and Trjitzinsky, W. J.},
  title   = {Analytic Theory of Singular Difference Equations},
  journal = {Acta Mathematica},
  volume  = {60},
  number  = {1},
  pages   = {1--89},
  year    = {1933}
}

@article{wimp1985resurrecting,
  author  = {Wimp, Jet and Zeilberger, Doron},
  title   = {Resurrecting the Asymptotics of Linear Recurrences},
  journal = {Journal of Mathematical Analysis and Applications},
  volume  = {111},
  number  = {1},
  pages   = {162--176},
  year    = {1985},
  doi     = {10.1016/0022-247X(85)90209-4}
}

@article{zahar1977mathematical,
  author  = {Zahar, R. V. M.},
  title   = {A Mathematical Analysis of Miller's Algorithm},
  journal = {Numerische Mathematik},
  volume  = {27},
  pages   = {427--447},
  year    = {1977},
  doi     = {10.1007/BF01399606}
}

@article{gelfandKapranovZelevinsky1990,
  author  = {Gel'fand, I. M. and Kapranov, M. M. and Zelevinsky, A. V.},
  title   = {Generalized {E}uler integrals and {$A$}-hypergeometric functions},
  journal = {Advances in Mathematics},
  volume  = {84},
  number  = {2},
  pages   = {255--271},
  year    = {1990}
}

@book{saitoSturmfelsTakayama2000,
  author    = {Saito, Mutsumi and Sturmfels, Bernd and Takayama, Nobuki},
  title     = {Gr{\"o}bner Deformations of Hypergeometric Differential Equations},
  series    = {Algorithms and Computation in Mathematics},
  volume    = {6},
  publisher = {Springer-Verlag},
  address   = {Berlin},
  year      = {2000}
}

@article{jainTzermias2005,
  author  = {Jain, Lalit and Tzermias, Pavlos},
  title   = {Beukers' Integrals and Ap{\'e}ry's Recurrences},
  journal = {Journal of Integer Sequences},
  volume  = {8},
  year    = {2005},
  number  = {1},
  pages   = {Article 05.1.1}
}

@incollection{rhin1987pade,
  author = {Rhin, Georges},
  title = {Approximants de {Pad\'e} et mesures effectives d'irrationalit\'e},
  booktitle = {S\'eminaire de {Th\'eorie} des {Nombres}, {Paris} 1985--86},
  series = {Progr. Math.},
  volume = {71},
  pages = {155--164},
  publisher = {Birkh\"auser},
  address = {Boston},
  year = {1987},
  doi = {10.1007/978-1-4757-4267-1_11}
}

@article{mahler1953pi,
  author = {Mahler, Kurt},
  title = {On the approximation of {$\pi$}},
  journal = {Indag. Math.},
  volume = {15},
  year = {1953},
  pages = {30--42}
}

@book{lewin1981polylogarithms,
  author = {Lewin, Leonard},
  title = {Polylogarithms and Associated Functions},
  publisher = {North-Holland},
  address = {New York},
  year = {1981}
}

@article{euler1776meditationes,
  author = {Euler, Leonhard},
  title = {Meditationes circa singulare serierum genus},
  journal = {Novi Comment. Acad. Sci. Petropol.},
  volume = {20},
  year = {1776},
  pages = {140--186}
}

@article{flajolet1998euler,
  author = {Flajolet, Philippe and Salvy, Bruno},
  title = {Euler sums and contour integral representations},
  journal = {Experiment. Math.},
  volume = {7},
  number = {1},
  year = {1998},
  pages = {15--35}
}

@article{fischler2018rational,
  author = {Fischler, St{\'e}phane and Rivoal, Tanguy},
  title = {Rational approximation to values of {$G$}-functions, and their expansions in integer bases},
  journal = {Manuscripta Math.},
  volume = {155},
  year = {2018},
  pages = {579--595},
  doi = {10.1007/s00229-017-0933-8}
}

@article{fischler2018erratum,
  author = {Fischler, St{\'e}phane and Rivoal, Tanguy},
  title = {Erratum to: {Rational} approximation to values of {$G$}-functions, and their expansions in integer bases},
  journal = {Manuscripta Math.},
  volume = {155},
  year = {2018},
  pages = {597--598},
  doi = {10.1007/s00229-017-0988-6}
}

@misc{kawashima2026linear,
  author = {Kawashima, Makoto},
  title = {Linear independence of periods related to polylogarithms},
  year = {2026},
  eprint = {2605.17756},
  archivePrefix = {arXiv},
  primaryClass = {math.NT},
  note = {Preprint},
  url = {https://arxiv.org/abs/2605.17756}
}

@book{nikishin1991rational,
  author = {Nikishin, E. M. and Sorokin, V. N.},
  title = {Rational Approximations and Orthogonality},
  series = {Translations of Mathematical Monographs},
  volume = {92},
  publisher = {American Mathematical Society},
  address = {Providence, RI},
  year = {1991}
}

@book{elaydi2005introduction,
  author = {Elaydi, Saber},
  title = {An Introduction to Difference Equations},
  edition = {3},
  publisher = {Springer},
  address = {New York},
  year = {2005}
}

@article{rosser1962approximate,
  author = {Rosser, J. Barkley and Schoenfeld, Lowell},
  title = {Approximate formulas for some functions of prime numbers},
  journal = {Illinois J. Math.},
  volume = {6},
  year = {1962},
  pages = {64--94},
  doi = {10.1215/ijm/1255631807}
}

@article{zudilin2002thirdOrder,
  author = {Zudilin, Wadim},
  title = {A third-order {Ap\'ery}-like recursion for {$\zeta(5)$}},
  journal = {Mathematical Notes},
  volume = {72},
  number = {5},
  pages = {733--737},
  year = {2002},
  doi = {10.1023/A:1021473409544},
  eprint = {math/0206178},
  archivePrefix = {arXiv}
}

@article{huttner2003abelian,
  author = {Huttner, Marc},
  title = {On linear independence measures of some {Abelian} integrals},
  journal = {Kyushu Journal of Mathematics},
  volume = {57},
  number = {1},
  pages = {129--157},
  year = {2003},
  doi = {10.2206/kyushujm.57.129}
}

@article{salikhov2007log3,
  author = {Salikhov, V. Kh.},
  title = {On the irrationality measure of {$\ln 3$}},
  journal = {Doklady Mathematics},
  volume = {76},
  number = {3},
  pages = {955--957},
  year = {2007},
  doi = {10.1134/S1064562407060361}
}

@article{heimonenMatalaAhoVaananen1994,
  author = {Heimonen, A. and Matala-aho, T. and V{\"a}{\"a}n{\"a}nen, K.},
  title = {An application of {Jacobi} type polynomials to irrationality measures},
  journal = {Bulletin of the Australian Mathematical Society},
  volume = {50},
  number = {2},
  pages = {225--243},
  year = {1994},
  doi = {10.1017/S0004972700013691}
}

@online{zeilbergerZudilin2019supplement,
  author = {Zeilberger, Doron and Zudilin, Wadim},
  title = {Automatic Discovery of Irrationality Proofs and Irrationality Measures: supplementary material},
  date = {2019-12-20},
  url = {https://www.math.ru.nl/~wzudilin/gat/},
  urldate = {2026-10-11},
  note = {The ALLADI package outputs oALLADI1.txt (linear denominators) and oALLADI4.txt (quadratic denominators)}
}

@article{borweinChoiPigulla2005,
  author = {Borwein, Jonathan Michael and Choi, Kwok-Kwong Stephen and Pigulla, Wilfried},
  title = {Continued fractions of tails of hypergeometric series},
  journal = {The American Mathematical Monthly},
  volume = {112},
  number = {6},
  pages = {493--501},
  year = {2005},
  url = {https://www.sfu.ca/~schoia/paper/contfrac.pdf}
}

@article{zudilin2003catalan,
  author = {Zudilin, Wadim},
  title = {An {Ap\'ery}-like difference equation for {Catalan}'s constant},
  journal = {The Electronic Journal of Combinatorics},
  volume = {10},
  number = {1},
  eid = {R14},
  year = {2003},
  eprint = {math/0201024},
  archivePrefix = {arXiv},
  note = {Preprint first posted in January 2002}
}

@misc{zudilin2002catalanRemarks,
  author = {Zudilin, Wadim},
  title = {A few remarks on linear forms involving {Catalan}'s constant},
  year = {2002},
  eprint = {math/0210423},
  archivePrefix = {arXiv}
}

@online{ramanujanMachine2020catalan,
  author = {{Ramanujan Machine}},
  title = {{Catalan}'s constant {$G$} results},
  year = {2020},
  url = {https://www.ramanujanmachine.com/wp-content/uploads/2020/06/catalan.pdf},
  urldate = {2026-10-11},
  note = {Catalogue of polynomial continued fractions; 22 entries marked ``new and unproven''}
}

@article{hermite1873exponentielle,
  author = {Hermite, Charles},
  title = {Sur la fonction exponentielle},
  journal = {Comptes rendus hebdomadaires des s{\'e}ances de l'Acad{\'e}mie des sciences},
  volume = {77},
  pages = {18--24, 74--79, 226--233, 285--293},
  year = {1873},
  url = {https://www.e-rara.ch/zut/content/titleinfo/27940137},
  note = {Digitized original offprint, ETH-Bibliothek Z{\"u}rich, Rar 22000}
}

@article{rivoal2009gamma,
  author = {Rivoal, Tanguy},
  title = {Rational approximations for values of derivatives of the {Gamma} function},
  journal = {Transactions of the American Mathematical Society},
  volume = {361},
  number = {11},
  pages = {6115--6149},
  year = {2009},
  doi = {10.1090/S0002-9947-09-04905-8},
  url = {https://rivoal.perso.math.cnrs.fr/articles/eulerconstant.pdf}
}

@article{hessamiPilehrood2013continued,
  author = {Hessami Pilehrood, Khodabakhsh and Hessami Pilehrood, Tatiana},
  title = {On a continued fraction expansion for {Euler}'s constant},
  journal = {Journal of Number Theory},
  volume = {133},
  pages = {769--786},
  year = {2013},
  eprint = {1010.1420},
  archivePrefix = {arXiv},
  url = {https://arxiv.org/abs/1010.1420}
}

@article{borweinBradley1997,
  author = {Borwein, Jonathan M. and Bradley, David M.},
  title = {Empirically determined {Ap\'ery}-like formulae for {$\zeta(4n+3)$}},
  journal = {Experimental Mathematics},
  volume = {6},
  number = {3},
  pages = {181--194},
  year = {1997},
  eprint = {math/0505124},
  archivePrefix = {arXiv},
  url = {https://arxiv.org/abs/math/0505124}
}

@article{reicheltSchulzeSevenheckWalther2021,
  author = {Reichelt, Thomas and Schulze, Mathias and Sevenheck, Christian and Walther, Uli},
  title = {Algebraic aspects of hypergeometric differential equations},
  journal = {Beitr{\"a}ge zur Algebra und Geometrie},
  year = {2021},
  doi = {10.1007/s13366-020-00560-1}
}

@incollection{Gosper1990,
  author = {R. W. Gosper},
  title = {Strip Mining in the Abandoned Orefields of Nineteenth Century Mathematics},
  booktitle = {Computers in Mathematics},
  publisher = {Dekker},
  year = {1990},
}

@misc{David2023,
  title = {The conservative matrix field}, 
  author = {O. David},
  year = {2023},
  eprint = {2303.09318},
  archivePrefix = {arXiv},
}

@misc{OharaTakayama2015,
  author        = {Katsuyoshi Ohara and Nobuki Takayama},
  title         = {Pfaffian Systems of {A}-Hypergeometric Systems {II}---Holonomic Gradient Method},
  year          = {2015},
  eprint        = {1505.02947},
  archivePrefix = {arXiv},
  primaryClass  = {cs.SC},
  url           = {https://arxiv.org/abs/1505.02947},
}

@misc{David2023CMF,
  author        = {David, Ofir},
  title         = {The conservative matrix field},
  year          = {2023},
  eprint        = {2303.09318},
  archivePrefix = {arXiv},
  primaryClass  = {math.GM},
  note          = {Version 3, December 2023}
}

@article{Chestnov2022Macaulay,
  author = {Chestnov, Vsevolod and Gasparotto, Federico and Mandal, Manoj K. and Mastrolia, Pierpaolo and Matsubara-Heo, Saiei J. and Munch, Henrik J. and Takayama, Nobuki},
  title = {Macaulay Matrix for Feynman Integrals: Linear Relations and Intersection Numbers},
  journal = {Journal of High Energy Physics},
  year = {2022},
  volume = {09},
  pages = {187},
  doi = {10.1007/JHEP09(2022)187},
  eprint = {2204.12983},
  archivePrefix = {arXiv}
}

@article{lambert1768,
  author  = {Lambert, Johann Heinrich},
  title   = {M{\'e}moire sur quelques propri{\'e}t{\'e}s remarquables des quantit{\'e}s transcendentes circulaires et logarithmiques},
  journal = {M{\'e}moires de l'Acad{\'e}mie Royale des Sciences et Belles-Lettres de Berlin},
  volume  = {17},
  year    = {1768},
  pages   = {265--322},
  note    = {Volume for the year 1761}
}

@article{Nesterenko85,
  author  = {Nesterenko, Yu. V.},
  title   = {On the linear independence of numbers},
  journal = {Vestnik Moskov. Univ. Ser. I Mat. Mekh.},
  year    = {1985}, number = {1}, pages = {46--49}}

@article{Hata90,
  author  = {Hata, Masayoshi},
  title   = {Legendre type polynomials and irrationality measures},
  journal = {J. Reine Angew. Math.},
  volume  = {407}, year = {1990}, pages = {99--125}}

@incollection{zagier2006integral,
  author    = {Zagier, Don},
  title     = {Integral Solutions of Ap{\'e}ry-like Recurrence Equations},
  booktitle = {Groups and Symmetries: From Neolithic Flats to {K}3 Surfaces},
  series    = {AMS/IP Studies in Advanced Mathematics},
  volume    = {38},
  pages     = {349--366},
  year      = {2006},
  publisher = {American Mathematical Society / International Press}
}

@article{brown2026mellin,
  title={Mellin transforms, transfinite diameter and rational approximations of integrals},
  author={Brown, Francis},
  journal={Experimental Mathematics},
  pages={1--36},
  year={2026},
  publisher={Taylor \& Francis}
}

@misc{fauzan2026,
  author       = {Fauzan, Aabir},
  title        = {\(\zeta(5)\) is irrational},
  month        = sep,
  year         = 2026,
  publisher    = {Zenodo},
  version      = {v1},
  doi          = {10.5281/zenodo.22826419},
  url          = {https://doi.org/10.5281/zenodo.22826419},
}

@article{brown2014irrationality,
  title={Irrationality proofs for zeta values, moduli spaces and dinner parties},
  author={Brown, Francis},
  journal={arXiv preprint arXiv:1412.6508},
  year={2014}
}

@article{zudilin2017determinantal,
  title={A determinantal approach to irrationality},
  author={Zudilin, Wadim},
  journal={Constructive Approximation},
  volume={45},
  number={2},
  pages={301--310},
  year={2017},
  publisher={Springer}
}

@article{ball2001irrationalite,
  title={Irrationalit{\'e} d’une infinit{\'e} de valeurs de la fonction z{\^e}ta aux entiers impairs},
  author={Ball, Keith and Rivoal, Tanguy},
  journal={Inventiones mathematicae},
  volume={146},
  number={1},
  pages={193--207},
  year={2001},
  publisher={Springer Berlin Heidelberg Berlin/Heidelberg}
}

@techreport{openai2026planar,
  author      = {{OpenAI}},
  title       = {Planar Point Sets with Many Unit Distances},
  institution = {OpenAI},
  year        = {2026},
  month       = 5,
  url         = {https://cdn.openai.com/pdf/74c24085-19b0-4534-9c90-465b8e29ad73/unit-distance-proof.pdf}
}

@book{turnbull1939gregory,
  editor    = {Turnbull, H. W.},
  title     = {James Gregory Tercentenary Memorial Volume},
  publisher = {G. Bell and Sons for the Royal Society of Edinburgh},
  address   = {London},
  year      = {1939},
  note      = {Contains Gregory's letter to J.~Collins of 15 February 1671 with the series for $\arctan x$}
}

@inproceedings{demoura2021lean4,
  author    = {de Moura, Leonardo and Ullrich, Sebastian},
  title     = {The {L}ean 4 Theorem Prover and Programming Language},
  booktitle = {Automated Deduction -- CADE 28},
  series    = {Lecture Notes in Computer Science},
  volume    = {12699},
  publisher = {Springer},
  year      = {2021},
  pages     = {625--635},
  doi       = {10.1007/978-3-030-79876-5_37}
}

\newpage
\appendix

\section{Additional CMF Constructions, Embeddings, and Numerical Data}
\label{sec:appendix-history}

\Needspace{8\baselineskip}
\subsection{Additional CMF recoveries}
\label{app:full-results}
This appendix supplements the reconstruction data in the main text.
We first collect the additional recoveries,
then specify the special CMFs, the numerator and denominator initial values, and finally the integral CMF used for $\pi$.

We use the initial-point, direction, walk and initial value conventions of the
main text.

Only fixed trajectory directions are included in the supplementary table;
index-dependent paths are omitted. Here $\varphi=(1+\sqrt5)/2$.

The supplementary table contains the entries omitted from the main
reconstruction table. It is obtained from the full table in Section~3 by
removing those constructions and all index-dependent paths, retaining
the remaining rows in their original order. For a published
irrationality-measure bound $\mu_0$, the comparison quantity
\[
  \delta_{\mathrm{eq}}=\frac{1}{\mu_0-1}
\]
need not be the approximation exponent of the underlying sequence.
We distinguish this quantity from sequence bounds, numerical estimates
of $\delta$, and the finite-depth values $\delta_N$ reported in the table.

\begingroup
\scriptsize
\setlength{\tabcolsep}{2pt}
\renewcommand{\arraystretch}{1.20}
\setlength{\LTcapwidth}{\linewidth}
\begin{longtable}{@{}>{\raggedright\arraybackslash}p{.13\linewidth}>{\raggedright\arraybackslash}p{.19\linewidth}c>{\raggedright\arraybackslash}p{.06\linewidth}>{\raggedright\arraybackslash}p{.105\linewidth}>{\raggedright\arraybackslash}p{.16\linewidth}>{\raggedright\arraybackslash}p{.21\linewidth}@{}}
\caption{The dictionary of CMF embeddings capturing the classical Ap\'ery-type Diophantine approximations from papers spanning several decades. 
This table expands on the selected results in Section~4. 
Subscripts $\mathrm F$ and $\mathrm D$ denote forward and dual walks. Notation such as $0^6$ denotes six repeated entries. Ramanujan Machine Catalan limits remain conjectural.}
\label{tab:appendix-reconstructions}\\
\toprule
Constant & Reference & Year & $\delta_N$ & CMF & Initial point & Direction \\
\midrule\endfirsthead
\multicolumn{7}{@{}l}{\emph{Additional CMF recoveries (continued)}}\\
\toprule
Constant & Reference & Year & $\delta_N$ & CMF & Initial point & Direction \\
\midrule\endhead
\midrule\multicolumn{7}{r@{}}{\emph{Continued on next page}}\\\endfoot
\bottomrule\endlastfoot
$\zeta(2)$ & Ap\'ery \cite{apery1979irrationalite} & 1979 & $0.094$
& ${}_3F_2(1)$ & $(-2,-2,3;1,1)$ & $(-1,-1,1;0,0)_{\mathrm F}$ \\
 & Beukers \cite{beukers1979note} & 1979 & $0.094$
& ${}_3F_2(1)$ & $(3,3,3;6,6)$ & $(1,1,1;2,2)_{\mathrm D}$ \\
 & Hata \cite{hata1995beukers} & 1995 & $0.217$
& ${}_3F_2(1)$ & $(3,3,3;6,6)$ & $(12,15,14;30,28)_{\mathrm D}$ \\
 & Rhin--Viola \cite{rhin1996permutation} & 1996 & $0.225$
& ${}_3F_2(1)$ & $(3,3,3;6,6)$ & $(12,13,14;24,28)_{\mathrm D}$ \\
 & Zudilin \cite{zudilin2002thirdOrder} (simultaneous) & 2002 & $-0.043$
& ${}_4F_3(1)$ & $(3,3,3,3;6,6,6)$ & $(1,1,1,1;2,2,2)_{\mathrm F}$ \\
 & Zudilin \cite{zudilin2014zeta2} & 2014 & $0.244$
& ${}_4F_3(1)$ & $(1,1,1,0;1,1,1)$ & $(8,7,6,-6;1,2,3)_{\mathrm F}$ \\
 & Dauguet--Zudilin \cite{dauguetZudilin2014} (simultaneous) & 2014 & $0.062$
& ${}_4F_3(1)$ & $(1,1,1,0;1,1,1)$ & $(10,9,8,-5;1,2,3)_{\mathrm F}$ \\
 & Niedbala Giraudin \cite{giraudin2026zeta2} & 2026 & $0.252$
& ${}_4F_3(1)$ & $(1,1,1,0;1,1,1)$ & $(169,146,123,-123;\allowbreak 23,46,69)_{\mathrm F}$ \\
\addlinespace
$\zeta(3)$ & Ap\'ery \cite{apery1979irrationalite} & 1979 & $0.082$
& ${}_4F_3(1)$ & $(-2,-2,3,3;\allowbreak 1,1,1)$ & $(-1,-1,1,1;0,0,0)_{\mathrm F}$ \\
 & Beukers \cite{beukers1979note} & 1979 & $0.081$
& ${}_4F_3(1)$ & $(-2,-2,3,3;\allowbreak 1,1,1)$ & $(-1,-1,1,1;0,0,0)_{\mathrm F}$ \\
 & Hata \cite{hata2000zeta3} & 2000 & $0.157$
& $G^{4,2}_{4,4}(1)$ & $(1,1,2,2;\allowbreak 1,1,1,1)$ & $(0,1,15,16;\allowbreak 8,7,8,9)_{\mathrm F}$ \\
 & Rhin--Viola \cite{rhinViola2001} & 2001 & $0.221$
& $G^{4,2}_{4,4}(1)$ & $(1,1,2,2;\allowbreak 1,1,1,1)$ & $(0,7,31,32;\allowbreak 18,17,16,19)_{\mathrm F}$ \\
 & Zudilin \cite{zudilin2002thirdOrder} (simultaneous) & 2002 & $-0.042$
& ${}_4F_3(1)$ & $(3,3,3,3;6,6,6)$ & $(1,1,1,1;2,2,2)_{\mathrm F}$ \\
 & Dauguet--Zudilin \cite{dauguetZudilin2014} (simultaneous) & 2014 & $0.038$
& ${}_4F_3(1)$ & $(1,1,1,0;1,1,1)$ & $(10,9,8,-5;1,2,3)_{\mathrm F}$ \\
 & Niedbala Giraudin \cite{giraudin2026zeta3} & 2026 & $0.222$
& $G^{4,2}_{4,4}(1)$ & $(1,1,2,2;\allowbreak 1,1,1,1)$ & $(0,363979,1612007,\allowbreak 1664004;\allowbreak 935996,883999,\allowbreak 832002,987993)_{\mathrm F}$ \\
\addlinespace
$\zeta(5)$ & Zudilin \cite{zudilin2002thirdOrder} & 2002 & $-0.221$
& ${}_8F_7(1)$ & $(8,3,3,3,3,3,3,3;\allowbreak 6,6,6,6,6,6,6)$ & $(3,1,1,1,1,1,1,1;\allowbreak 2,2,2,2,2,2,2)_{\mathrm F}$ \\
 & Brown--Zudilin \cite{brownZudilin2022zeta5} & 2022 & $-0.134$
& ${}_8F_7(1)$ & $(5,2,2,2,2,2,2,2;\allowbreak 4,4,4,4,4,4,4)$ & $(41,17,16,15,\allowbreak 14,13,12,11;\allowbreak 24,25,26,27,28,29,30)_{\mathrm D}$ \\
\addlinespace
$\gamma$ & Rivoal \cite{rivoal2009gamma} (third order) & 2009 & $-0.940$
& $G^{3, 2}_{2,3}(1)$ & $(0, 0; 0, 0, 0)$ & $(0, 0; 1, 1, 1)_{\mathrm D}$ \\
 & Rivoal \cite{rivoal2009gamma} (simultaneous) & 2009 & $-0.928$
& $G^{4, 3}_{3, 4}(1)$ & $(0, 0, 0; 0, 0, 0, 0)$ & $(0, 0, 0; 1, 1, 1, 1)_{\mathrm D}$ \\
\addlinespace
Catalan's $G$ & Zudilin \cite{zudilin2003catalan} (original) & 2003 & $-0.476$
& ${}_3F_2(1)$ & $(2,2,\tfrac32;\tfrac72,\tfrac72)$ & $(1,1,1;2,2)_{\mathrm F}$ \\
 & Zudilin \cite{zudilin2002catalanRemarks} (modified) & 2002 & $-0.476$
& ${}_3F_2(1)$ & $(1,1,\tfrac12;\tfrac52,\tfrac52)$ & $(1,1,1;2,2)_{\mathrm F}$ \\
 & Nesterenko \cite{nesterenko2016catalan} & 2016 & $-0.476$
& ${}_3F_2(1)$ & $(1,1,\tfrac32;\tfrac52,\tfrac52)$ & $(1,1,1;2,2)_{\mathrm F}$ \\
$\zeta(7)$ & Borwein--Bradley \cite{borweinBradley1997} & 1997 & $-0.824$ & $\mathcal B_{15}$ & $(0,\allowbreak 0)$ & $(1,\allowbreak 0)_{\mathrm F}$ \\
\addlinespace
$\log 2$ & Alladi--Robinson \cite{alladiRobinson1980} & 1980 & $0.282$ & ${}_2F_1(-1)$ & $(1,\allowbreak 1;\allowbreak 2)$ & $(1,\allowbreak 1;\allowbreak 2)_{\mathrm F}$ \\
 & Rhin (Chudnovsky benchmark) \cite{rhin1987pade} & 1987 & $0.324$ & $T(e_1,\allowbreak e_2,\allowbreak b)$ & $(0,\allowbreak 0,\allowbreak 1)$ & $(9,\allowbreak 1,\allowbreak 10)_{\mathrm F}$ \\
 & Rukhadze \cite{rukhadze1987log2} & 1987 & $0.346$ & ${}_2F_1(-1)$ & $(1,\allowbreak 1;\allowbreak 2)$ & $(7,\allowbreak 6;\allowbreak 14)_{\mathrm D}$ \\
 & Huttner \cite{huttner2003abelian} (simultaneous) & 2003 & $0.225$ & $\mathcal H$ & $(2,\allowbreak 2,\allowbreak 2,\allowbreak 3)$ & $(2,\allowbreak 2,\allowbreak 2,\allowbreak 3)_{\mathrm F}$ \\
 & Marcovecchio \cite{marcovecchio2009log2} & 2009 & $0.388$ & ${}_3F_2(\tfrac12)$ & $(1,\allowbreak 1,\allowbreak 1;\allowbreak 2,\allowbreak 1)$ & $(7,\allowbreak 5,\allowbreak 6;\allowbreak 2,\allowbreak 1)_{\mathrm F}$ \\
\addlinespace
$\log 3$ & Huttner \cite{huttner2003abelian} (simultaneous) & 2003 & $0.100$ & $\mathcal H$ & $(2,\allowbreak 2,\allowbreak 2,\allowbreak 3)$ & $(2,\allowbreak 2,\allowbreak 2,\allowbreak 3)_{\mathrm F}$ \\
 & Salikhov \cite{salikhov2007log3} & 2007 & $0.245$ & $\mathcal S$ & $(-\tfrac12,\allowbreak 0,\allowbreak 0,\allowbreak 1)$ & $(1,\allowbreak 1,\allowbreak 1,\allowbreak 2)_{\mathrm F}$ \\
\addlinespace
$\log(3/2)$ & Alladi--Robinson \cite{alladiRobinson1980} & 1980 & $0.408$ & ${}_2F_1(-\tfrac12)$ & $(2,\allowbreak 2;\allowbreak 4)$ & $(1,\allowbreak 1;\allowbreak 2)_{\mathrm F}$ \\
 & Huttner \cite{huttner2003abelian} (simultaneous) & 2003 & $0.099$ & $\mathcal H$ & $(2,\allowbreak 2,\allowbreak 2,\allowbreak 3)$ & $(2,\allowbreak 2,\allowbreak 2,\allowbreak 3)_{\mathrm F}$ \\
\addlinespace
$\log(4/3)$ & Alladi--Robinson \cite{alladiRobinson1980} & 1980 & $0.461$ & ${}_2F_1(-\tfrac13)$ & $(2,\allowbreak 2;\allowbreak 4)$ & $(1,\allowbreak 1;\allowbreak 2)_{\mathrm F}$ \\
 & Huttner \cite{huttner2003abelian} (simultaneous) & 2003 & $0.209$ & $\mathcal H$ & $(2,\allowbreak 2,\allowbreak 2,\allowbreak 3)$ & $(2,\allowbreak 2,\allowbreak 2,\allowbreak 3)_{\mathrm F}$ \\
\addlinespace
$\log 5$ & Alladi--Robinson \cite{alladiRobinson1980} & 1980 & $-0.420$ & ${}_2F_1(-4)$ & $(1,\allowbreak 1;\allowbreak 2)$ & $(1,\allowbreak 1;\allowbreak 2)_{\mathrm F}$ \\
\addlinespace
$\log 7$ & Alladi--Robinson \cite{alladiRobinson1980} & 1980 & $-0.553$ & ${}_2F_1(-6)$ & $(1,\allowbreak 1;\allowbreak 2)$ & $(1,\allowbreak 1;\allowbreak 2)_{\mathrm F}$ \\
\addlinespace
$\log(5/4)$ & Alladi--Robinson \cite{alladiRobinson1980} & 1980 & $0.497$ & ${}_2F_1(-\tfrac14)$ & $(1,\allowbreak 1;\allowbreak 2)$ & $(1,\allowbreak 1;\allowbreak 2)_{\mathrm F}$ \\
\addlinespace
$\log(6/5)$ & Alladi--Robinson \cite{alladiRobinson1980} & 1980 & $0.518$ & ${}_2F_1(-\tfrac15)$ & $(1,\allowbreak 1;\allowbreak 2)$ & $(1,\allowbreak 1;\allowbreak 2)_{\mathrm F}$ \\
\addlinespace
$\log(5/3)$ & Alladi--Robinson \cite{alladiRobinson1980} & 1980 & $0.106$ & ${}_2F_1(-\tfrac23)$ & $(1,\allowbreak 1;\allowbreak 2)$ & $(1,\allowbreak 1;\allowbreak 2)_{\mathrm F}$ \\
\addlinespace
$\log(7/5)$ & Alladi--Robinson \cite{alladiRobinson1980} & 1980 & $0.194$ & ${}_2F_1(-\tfrac25)$ & $(1,\allowbreak 1;\allowbreak 2)$ & $(1,\allowbreak 1;\allowbreak 2)_{\mathrm F}$ \\
\addlinespace
$\pi/\sqrt3$ & Zeilberger--Zudilin \cite{zeilbergerZudilin2019supplement} & 2019 & $0.145$ & ${}_2F_1(-\tfrac1{3})$ & $(\tfrac12,\allowbreak 1;\allowbreak \tfrac32)$ & $(1,\allowbreak 1;\allowbreak 2)_{\mathrm F}$ \\
 & Hata \cite{hata1993pi} & 1993 & $0.277$ & ${}_3F_2(-\tfrac13)$ & $(\tfrac12,\allowbreak 1,\allowbreak \tfrac12;\allowbreak \tfrac12,\allowbreak \tfrac32)$ & $(3,\allowbreak 3,\allowbreak 2;\allowbreak 0,\allowbreak 6)_{\mathrm F}$ \\
 & Heimonen--Matala-aho--V\"a\"an\"anen \cite{heimonenMatalaAhoVaananen1994} & 1994 & $0.255$ & ${}_2F_1(-\tfrac13)$ & $(3,\allowbreak \tfrac52;\allowbreak \tfrac{13}2)$ & $(2,\allowbreak 2;\allowbreak 5)_{\mathrm F}$ \\
\addlinespace
\shortstack[l]{$\arctan(1/\sqrt{7})$\\$/\sqrt{7}$} & Heimonen--Matala-aho--V\"a\"an\"anen \cite{heimonenMatalaAhoVaananen1994} & 1994 & $0.344$ & ${}_2F_1(-\tfrac17)$ & $(2,\allowbreak \tfrac32;\allowbreak \tfrac72)$ & $(3,\allowbreak 3;\allowbreak 7)_{\mathrm F}$ \\
 & Zeilberger--Zudilin \cite{zeilbergerZudilin2019supplement} & 2019 & $0.268$ & ${}_2F_1(-\tfrac1{7})$ & $(\tfrac12,\allowbreak 1;\allowbreak \tfrac32)$ & $(1,\allowbreak 1;\allowbreak 2)_{\mathrm F}$ \\
\addlinespace
\shortstack[l]{$\arctan(1/\sqrt{11})$\\$/\sqrt{11}$} & Hata \cite{hata1992hypergeometric} & 1992 & $0.371$ & ${}_2F_1(-\tfrac1{11})$ & $(1,\allowbreak \tfrac12;\allowbreak \tfrac32)$ & $(4,\allowbreak 4;\allowbreak 9)_{\mathrm F}$ \\
 & Zeilberger--Zudilin \cite{zeilbergerZudilin2019supplement} & 2019 & $0.324$ & ${}_2F_1(-\tfrac1{11})$ & $(\tfrac12,\allowbreak 1;\allowbreak \tfrac32)$ & $(1,\allowbreak 1;\allowbreak 2)_{\mathrm F}$ \\
\addlinespace
\shortstack[l]{$\arctan(1/\sqrt{15})$\\$/\sqrt{15}$} & Hata \cite{hata1992hypergeometric} & 1992 & $0.399$ & ${}_2F_1(-\tfrac1{15})$ & $(1,\allowbreak \tfrac12;\allowbreak \tfrac32)$ & $(4,\allowbreak 4;\allowbreak 9)_{\mathrm F}$ \\
 & Zeilberger--Zudilin \cite{zeilbergerZudilin2019supplement} & 2019 & $0.356$ & ${}_2F_1(-\tfrac1{15})$ & $(\tfrac12,\allowbreak 1;\allowbreak \tfrac32)$ & $(1,\allowbreak 1;\allowbreak 2)_{\mathrm F}$ \\
\addlinespace
\shortstack[l]{$\pi/\sqrt3$\\$-\log3$} & Hata \cite{hata1992hypergeometric} & 1992 & $0.286$ & ${}_2F_1(-\tfrac18)$ & $(2,\allowbreak \tfrac53;\allowbreak \tfrac{11}3)$ & $(3,\allowbreak 3;\allowbreak 7)_{\mathrm F}$ \\
\addlinespace
\shortstack[l]{$\pi/\sqrt3$\\$+\log3$} & Hata \cite{hata1992hypergeometric} & 1992 & $0.286$ & ${}_2F_1(-\tfrac18)$ & $(2,\allowbreak \tfrac43;\allowbreak \tfrac{10}3)$ & $(3,\allowbreak 3;\allowbreak 7)_{\mathrm F}$ \\
\addlinespace
\shortstack[l]{$\pi/\sqrt3-$\\$\sqrt3\log(2+\sqrt3)$} & Hata \cite{hata1992hypergeometric} & 1992 & $0.200$ & ${}_2F_1(\tfrac19)$ & $(2,\allowbreak \tfrac74;\allowbreak \tfrac{15}4)$ & $(3,\allowbreak 3;\allowbreak 7)_{\mathrm F}$ \\
\addlinespace
\shortstack[l]{$\pi/\sqrt3+$\\$\sqrt3\log(2+\sqrt3)$} & Hata \cite{hata1992hypergeometric} & 1992 & $0.200$ & ${}_2F_1(\tfrac19)$ & $(2,\allowbreak \tfrac54;\allowbreak \tfrac{13}4)$ & $(3,\allowbreak 3;\allowbreak 7)_{\mathrm F}$ \\
\addlinespace
$\sqrt5\log\varphi$ & Hata \cite{hata1992hypergeometric} & 1992 & $0.297$ & ${}_2F_1(\tfrac15)$ & $(1,\allowbreak \tfrac12;\allowbreak \tfrac32)$ & $(2,\allowbreak 2;\allowbreak 5)_{\mathrm F}$ \\
\addlinespace
$\pi$ & Gauss arctangent CF (DLMF) \cite[Eq.~4.25.4]{NIST:DLMF} & 2026 & $-0.203$ & ${}_2F_1(-1)$ & $(\tfrac12,\allowbreak 1;\allowbreak \tfrac32)$ & $(1,\allowbreak 1;\allowbreak 2)_{\mathrm F}$ \\
 & Mahler \cite{mahler1953pi} & 1953 & $-0.538$ & $G^{11,\allowbreak 0}_{11,\allowbreak 11}(-i)$ & $(1^{11};\allowbreak 0^{11})$ & $(1^{11};\allowbreak 0^{11})_{\mathrm D}$ \\
 & Mignotte \cite{mignotte1974pi} & 1974 & $-0.413$ & $G^{6,\allowbreak 0}_{6,\allowbreak 6}(-i)$ & $(1^{6};\allowbreak 0^{6})$ & $(1^{6};\allowbreak 0^{6})_{\mathrm D}$ \\
 & Hata \cite{hata1993pi} & 1993 & $0.144$ & $J(u,\allowbreak v,\allowbreak w,\allowbreak m)$ & $(2,\allowbreak 2,\allowbreak 2,\allowbreak 4)$ & $(2,\allowbreak 2,\allowbreak 2,\allowbreak 3)_{\mathrm F}$ \\
 & Borwein--Choi--Pigulla \cite{borweinChoiPigulla2005} (Taylor index $2$) & 2005 & $-0.198$ & ${}_2F_1(-1)$ & $(\tfrac52,\allowbreak 1;\allowbreak \tfrac72)$ & $(1,\allowbreak 1;\allowbreak 2)_{\mathrm F}$ \\
 & Salikhov \cite{salikhov2010pi} & 2010 & $0.157$ & $K(a,\allowbreak b,\allowbreak g)$ & $(6,\allowbreak 11,\allowbreak 3)$ & $(6,\allowbreak 10,\allowbreak 3)_{\mathrm F}$ \\
 & Zeilberger--Zudilin \cite{zeilbergerZudilin2020} & 2020 & $0.163$ & $K(a,\allowbreak b,\allowbreak g)$ & $(2,\allowbreak 4,\allowbreak 1)$ & $(4,\allowbreak 6,\allowbreak 2)_{\mathrm F}$ \\
 & Bai (preprint) \cite{bai2026pi} & 2026 & $0.165$ & $K(a,\allowbreak b,\allowbreak g)$ & $(0,\allowbreak 1,\allowbreak 0)$ & $(3714,\allowbreak 5570,\allowbreak 1857)_{\mathrm F}$ \\
\addlinespace
$e$ & Hermite A \cite{hermite1873exponentielle} & 1873 & $1.001$ & ${}_1F_1(1)$ & $(1;\allowbreak 2)$ & $(1;\allowbreak 2)_{\mathrm F}$ \\
 & Hermite B \cite{hermite1873exponentielle} & 1873 & $1.000$ & ${}_1F_1(1)$ & $(2;\allowbreak 3)$ & $(1;\allowbreak 2)_{\mathrm F}$ \\
 & Hermite C \cite{hermite1873exponentielle} & 1873 & $1.000$ & ${}_1F_1(1)$ & $(1;\allowbreak 3)$ & $(1;\allowbreak 2)_{\mathrm F}$ \\
\addlinespace
$\sqrt e$ & Hermite A \cite{hermite1873exponentielle} & 1873 & $1.001$ & ${}_1F_1(\tfrac12)$ & $(1;\allowbreak 2)$ & $(1;\allowbreak 2)_{\mathrm F}$ \\
 & Hermite B \cite{hermite1873exponentielle} & 1873 & $1.000$ & ${}_1F_1(\tfrac12)$ & $(2;\allowbreak 3)$ & $(1;\allowbreak 2)_{\mathrm F}$ \\
 & Hermite C \cite{hermite1873exponentielle} & 1873 & $1.000$ & ${}_1F_1(\tfrac12)$ & $(1;\allowbreak 3)$ & $(1;\allowbreak 2)_{\mathrm F}$ \\
\addlinespace
$\sqrt[3]{e}$ & Hermite A \cite{hermite1873exponentielle} & 1873 & $1.001$ & ${}_1F_1(\tfrac13)$ & $(1;\allowbreak 2)$ & $(1;\allowbreak 2)_{\mathrm F}$ \\
 & Hermite B \cite{hermite1873exponentielle} & 1873 & $1.000$ & ${}_1F_1(\tfrac13)$ & $(2;\allowbreak 3)$ & $(1;\allowbreak 2)_{\mathrm F}$ \\
 & Hermite C \cite{hermite1873exponentielle} & 1873 & $1.000$ & ${}_1F_1(\tfrac13)$ & $(1;\allowbreak 3)$ & $(1;\allowbreak 2)_{\mathrm F}$ \\
\addlinespace
$e^2$ & Hermite A \cite{hermite1873exponentielle} & 1873 & $1.001$ & ${}_1F_1(2)$ & $(1;\allowbreak 2)$ & $(1;\allowbreak 2)_{\mathrm F}$ \\
 & Hermite B \cite{hermite1873exponentielle} & 1873 & $1.000$ & ${}_1F_1(2)$ & $(2;\allowbreak 3)$ & $(1;\allowbreak 2)_{\mathrm F}$ \\
 & Hermite C \cite{hermite1873exponentielle} & 1873 & $1.000$ & ${}_1F_1(2)$ & $(1;\allowbreak 3)$ & $(1;\allowbreak 2)_{\mathrm F}$ \\
\addlinespace
$\gamma+\log(z)$ & Rivoal \cite{rivoal2009gamma} (third order) & 2009 & $\approx -1$ & $G^{3,\allowbreak  2}_{2,\allowbreak 3}(z)$ & $(0,\allowbreak  0;\allowbreak  0,\allowbreak  0,\allowbreak  0)$ & $(0,\allowbreak  0;\allowbreak  1,\allowbreak  1,\allowbreak  1)_{\mathrm D}$ \\
\addlinespace
$\zeta(2)-\gamma^2$ & Rivoal \cite{rivoal2009gamma} (simultaneous) & 2009 & $-0.933$ & $G^{4,\allowbreak  3}_{3,\allowbreak  4}(1)$ & $(0,\allowbreak  0,\allowbreak  0;\allowbreak  0,\allowbreak  0,\allowbreak  0,\allowbreak  0)$ & $(0,\allowbreak  0,\allowbreak  0;\allowbreak  1,\allowbreak  1,\allowbreak  1,\allowbreak  1)_{\mathrm D}$ \\
\addlinespace
Euler--Gompertz $\mathcal G$ & Hessami Pilehrood--Hessami Pilehrood \cite{hessamiPilehrood2013continued} (Stieltjes CF) & 2013 & $-0.979$ & $G^{0,\allowbreak  2}_{2,\allowbreak  1}(1)$ & $(1,\allowbreak 1;\allowbreak  0)$ & $(1,\allowbreak 1;\allowbreak  0)_{\mathrm D}$ \\
\addlinespace
Catalan's $G$ & RM catalogue \#1 \cite{ramanujanMachine2020catalan} & 2020 & $-0.825$ & $\mathcal{R}^g_{0,\allowbreak 0}$ & $(0,\allowbreak 0)$ & $(1,\allowbreak 0)_{\mathrm F}$ \\
 & RM catalogue \#2 \cite{ramanujanMachine2020catalan} & 2020 & $-0.837$ & $\mathcal{R}^g_{0,\allowbreak 0}$ & $(0,\allowbreak 1)$ & $(1,\allowbreak 0)_{\mathrm F}$ \\
 & RM catalogue \#3 \cite{ramanujanMachine2020catalan} & 2020 & $-0.848$ & $\mathcal{R}^g_{0,\allowbreak 0}$ & $(0,\allowbreak 2)$ & $(1,\allowbreak 0)_{\mathrm F}$ \\
 & RM catalogue \#4 \cite{ramanujanMachine2020catalan} & 2020 & $-0.857$ & $\mathcal{R}^g_{0,\allowbreak 0}$ & $(0,\allowbreak 3)$ & $(1,\allowbreak 0)_{\mathrm F}$ \\
 & RM catalogue \#5 \cite{ramanujanMachine2020catalan} & 2020 & $-0.815$ & $\mathcal{R}^g_{0,\allowbreak 2}$ & $(0,\allowbreak 0)$ & $(1,\allowbreak 0)_{\mathrm F}$ \\
 & RM catalogue \#6 \cite{ramanujanMachine2020catalan} & 2020 & $-0.828$ & $\mathcal{R}^g_{0,\allowbreak 2}$ & $(0,\allowbreak 1)$ & $(1,\allowbreak 0)_{\mathrm F}$ \\
 & RM catalogue \#7 \cite{ramanujanMachine2020catalan} & 2020 & $-0.840$ & $\mathcal{R}^g_{0,\allowbreak 2}$ & $(0,\allowbreak 2)$ & $(1,\allowbreak 0)_{\mathrm F}$ \\
 & RM catalogue \#8 \cite{ramanujanMachine2020catalan} & 2020 & $-0.850$ & $\mathcal{R}^g_{0,\allowbreak 2}$ & $(0,\allowbreak 3)$ & $(1,\allowbreak 0)_{\mathrm F}$ \\
 & RM catalogue \#9 \cite{ramanujanMachine2020catalan} & 2020 & $-0.812$ & $\mathcal{R}^f_{1,\allowbreak 1}$ & $(0,\allowbreak 1)$ & $(1,\allowbreak 0)_{\mathrm F}$ \\
 & RM catalogue \#10 \cite{ramanujanMachine2020catalan} & 2020 & $-0.808$ & $\mathcal{R}^g_{0,\allowbreak 4}$ & $(0,\allowbreak 0)$ & $(1,\allowbreak 0)_{\mathrm F}$ \\
 & RM catalogue \#11 \cite{ramanujanMachine2020catalan} & 2020 & $-0.820$ & $\mathcal{R}^g_{0,\allowbreak 4}$ & $(0,\allowbreak 1)$ & $(1,\allowbreak 0)_{\mathrm F}$ \\
 & RM catalogue \#12 \cite{ramanujanMachine2020catalan} & 2020 & $-0.830$ & $\mathcal{R}^g_{0,\allowbreak 4}$ & $(0,\allowbreak 2)$ & $(1,\allowbreak 0)_{\mathrm F}$ \\
 & RM catalogue \#13 \cite{ramanujanMachine2020catalan} & 2020 & $-0.808$ & $\mathcal{R}^g_{2,\allowbreak 2}$ & $(0,\allowbreak 0)$ & $(1,\allowbreak 0)_{\mathrm F}$ \\
 & RM catalogue \#14 \cite{ramanujanMachine2020catalan} & 2020 & $-0.817$ & $\mathcal{R}^g_{2,\allowbreak 2}$ & $(0,\allowbreak 1)$ & $(1,\allowbreak 0)_{\mathrm F}$ \\
 & RM catalogue \#15 \cite{ramanujanMachine2020catalan} & 2020 & $-0.833$ & $\mathcal{R}^g_{2,\allowbreak 2}$ & $(0,\allowbreak 2)$ & $(1,\allowbreak 0)_{\mathrm F}$ \\
 & RM catalogue \#16 \cite{ramanujanMachine2020catalan} & 2020 & $-0.844$ & $\mathcal{R}^g_{2,\allowbreak 2}$ & $(0,\allowbreak 3)$ & $(1,\allowbreak 0)_{\mathrm F}$ \\
 & RM catalogue \#17 \cite{ramanujanMachine2020catalan} & 2020 & $-0.805$ & $\mathcal{R}^f_{1,\allowbreak 1}$ & $(0,\allowbreak 3)$ & $(1,\allowbreak 0)_{\mathrm F}$ \\
 & RM catalogue \#18 \cite{ramanujanMachine2020catalan} & 2020 & $-0.800$ & $\mathcal{R}^g_{2,\allowbreak 4}$ & $(0,\allowbreak 0)$ & $(1,\allowbreak 0)_{\mathrm F}$ \\
 & RM catalogue \#19 \cite{ramanujanMachine2020catalan} & 2020 & $-0.811$ & $\mathcal{R}^g_{2,\allowbreak 4}$ & $(0,\allowbreak 1)$ & $(1,\allowbreak 0)_{\mathrm F}$ \\
 & RM catalogue \#20 \cite{ramanujanMachine2020catalan} & 2020 & $-0.823$ & $\mathcal{R}^g_{2,\allowbreak 4}$ & $(0,\allowbreak 2)$ & $(1,\allowbreak 0)_{\mathrm F}$ \\
 & RM catalogue \#21 \cite{ramanujanMachine2020catalan} & 2020 & $-0.801$ & $\mathcal{R}^f_{2,\allowbreak 2}$ & $(0,\allowbreak 2)$ & $(1,\allowbreak 0)_{\mathrm F}$ \\
 & RM catalogue \#22 \cite{ramanujanMachine2020catalan} & 2020 & $-0.798$ & $\mathcal{R}^f_{1,\allowbreak 3}$ & $(0,\allowbreak 3)$ & $(1,\allowbreak 0)_{\mathrm F}$ \\
 & Eskandari (preprint) \cite{eskandari2026catalan} & 2026 & $-0.373$ & $E(k,\allowbreak m)$ & $(0,\allowbreak 0)$ & $(3,\allowbreak 10)_{\mathrm F}$ \\
\addlinespace
$\zeta(2)+\zeta(3)$ & Zudilin \cite{zudilin2002thirdOrder} & 2002 & $-0.158$ & ${}_4F_3(1)$ & $(3,\allowbreak 3,\allowbreak 3,\allowbreak 3;\allowbreak 6,\allowbreak 6,\allowbreak 6)$ & $(1,\allowbreak 1,\allowbreak 1,\allowbreak 1;\allowbreak 2,\allowbreak 2,\allowbreak 2)_{\mathrm F}$ \\
 & Dauguet--Zudilin \cite{dauguetZudilin2014} & 2014 & $-0.101$ & ${}_4F_3(1)$ & $(1,\allowbreak 1,\allowbreak 1,\allowbreak 0;\allowbreak 1,\allowbreak 1,\allowbreak 1)$ & $(10,\allowbreak 9,\allowbreak 8,\allowbreak -5;\allowbreak 3,\allowbreak 2,\allowbreak 1)_{\mathrm F}$ \\
\addlinespace
$\zeta(3)+\zeta(5)$ & Zudilin \cite{zudilin2002thirdOrder} & 2002 & $-0.281$ & ${}_8F_7(1)$ & $(8,\allowbreak 3^{7};\allowbreak 6^{7})$ & $(3,\allowbreak 1^{7};\allowbreak 2^{7})_{\mathrm F}$ \\
\addlinespace
$L(2,\chi_{-3})$ & Zagier \cite{zagier2009apery} & 2009 & $-0.475$ & ${}_3F_2(4)$ & $(0,\allowbreak 0,\allowbreak \tfrac12;\allowbreak 1,\allowbreak 1)$ & $(-1,\allowbreak -1,\allowbreak 0;\allowbreak 0,\allowbreak 0)_{\mathrm F}$ \\
\end{longtable}
\endgroup

\Needspace{8\baselineskip}
\subsection{Constructing the special CMFs}
\label{app:special-cmf-constructions}
\label{app:special-cmfs}

This subsection specifies the special CMFs beyond the generalized
hypergeometric and Meijer--$G$ families. Constructions represented by
${}_pF_q$ or Meijer--$G$ CMFs in the reconstruction tables are treated
using those families and are not repeated here.

We record the family-specific weights, recurrence data and shift matrices.
The integral CMF for $\pi$ is developed separately in
Section~\ref{app:pi-build-recovery}.
Write $C(q_0,\ldots,q_r)$ for the row-convention companion matrix,
with subdiagonal ones and last column
$-(q_0,\ldots,q_{r-1})^{\mathsf T}/q_r$.

\Needspace{7\baselineskip}
\subsubsection{Integral CMFs: construction of the shift matrices}
For each integral family below, the weight and contour determine the
moments. The following reduction procedure specifies the matrices;
listing the weight alone is not the construction.

\smallskip
\paragraph{\textbf{Input and moment basis.}}
Write the weight in exponent coordinates as
\[
 w(t;\lambda)=\prod_{\nu=1}^d f_\nu(t)^{\lambda_\nu},
 \qquad D(t)=\prod_{\nu=1}^d f_\nu(t),\qquad r=\deg D-1.
\]
Use the moment basis
\[
 I_j(\lambda)=\int_\gamma t^jw(t;\lambda)\,dt,
 \qquad 0\leq j<r.
\]
The construction is taken on a parameter domain where these moment
classes form a basis. The integer $r$ specifies the size of the matrices.

\smallskip
\paragraph{\textbf{Reduction coefficients.}}
Integration by parts, initially with vanishing boundary terms, gives
\[
 \int_\gamma R_k(t;\lambda)w(t;\lambda)\,dt=0,
 \qquad
 R_k=kt^{k-1}D+t^k\left(D'+D\frac{w'}w\right).
\]
Since $w'/w=\sum_\nu\lambda_\nu f_\nu'/f_\nu$, each $R_k$ is a
polynomial with coefficients rational in the parameters. Write
\[
 R_k=\sum_{j=0}^{r+k}\rho_{kj}(\lambda)t^j.
\]
Let $u_0,\ldots,u_{r-1}$ be the standard columns of $\mathbb Q^r$.
For $\ell\geq r$, define further columns successively by
\[
 u_\ell=-\frac{1}{\rho_{\ell-r,\ell}}
       \sum_{j=0}^{\ell-1}\rho_{\ell-r,j}u_j.
\]
Where the displayed denominator is nonzero, this is exactly the reduction
\[
 I_\ell=(I_0,\ldots,I_{r-1})u_\ell.
\]
Thus every required higher moment is reduced by a finite sequence of
rational operations. No numerical fitting is involved.

\smallskip
\paragraph{\textbf{Positive exponent shifts.}}
Expand $f_\nu(t)=\sum_{a=0}^{d_\nu}f_{\nu,a}t^a$. Increasing
$\lambda_\nu$ by one multiplies the integrand by $f_\nu$, so
\[
 I_j(\lambda+e_\nu)
 =\sum_{a=0}^{d_\nu}f_{\nu,a}I_{j+a}(\lambda).
\]
Consequently column $j$ of the positive-shift matrix is
\[
\operatorname{col}_j\!\bigl(M_{\lambda_\nu}(\lambda)\bigr)
=\sum_{a=0}^{d_\nu} f_{\nu,a}\,u_{j+a}(\lambda),
\qquad 0\leq j<r.
\]
Only the columns through $u_{r-1+d_\nu}$ are needed. This formula,
combined with the reduction recursion, specifies every matrix entry.

\smallskip
\paragraph{\textbf{Denominator parameters and compatibility.}}
A negative exponent shift is represented by
$M_{\lambda_\nu}(\lambda-e_\nu)^{-1}$. Thus increasing a denominator
parameter, such as $b$ in $(1+t)^{-b}$, uses a negative exponent shift;
the inverse is evaluated at the shifted parameters.
Commuting parameter shifts induce compatible matrices when reduced in
this common basis. The matrices therefore give a CMF on the domain
where the basis, reduction denominators and shift inverses are valid.
These conditions must also hold along the chosen trajectory.

\Needspace{7\baselineskip}
\subsubsection{Integral families}
For each family, substitute the factors of the displayed weight into the
construction above. The exponents may be affine functions of the named
parameters; for example, $g-\tfrac12$ increases by one when $g$ does.
The stated contours determine the moments and their boundary conditions.

\smallskip
\paragraph{\textbf{The rank-four CMF $T(e_1,e_2,b)$.}}
The weight and contour are
\[
 w_T(t)=[t(1-t)]^{e_1}(6t^2-5t+1)^{e_2}(1+t)^{-b},
 \qquad \gamma=[0,1].
\]

\smallskip
\paragraph{\textbf{The rank-fourteen CMF $\widetilde T(\mathbf e,b)$.}}
Here $\mathbf e=(e_1,\ldots,e_6)$, and
\[
 w_{\widetilde T}(t)=\prod_{\nu=1}^6p_\nu(t)^{e_\nu}(1+t)^{-b},
 \qquad \gamma=[0,1],
\]
with polynomial factors
\[
\begin{aligned}
 p_1&=t(1-t),&p_2&=t^2+2t-1,\\
 p_3&=6t^2-5t+1,&p_4&=7t^2-6t+1,\\
 p_5&=13t^2-11t+2,&p_6&=2t^4+3t^3+3t^2-5t+1.
\end{aligned}
\]

\smallskip
\paragraph{\textbf{The rank-three CMF $J(u,v,w,m)$.}}
The weight and contour are
\[
 w_J(t)=(t-1)^u(t-2)^v(t-1-i)^w t^{-m},
 \qquad \gamma:1\longrightarrow1+i.
\]
The rank is three over $\mathbb Q(i)$, or six after taking real and
imaginary parts.

\smallskip
\paragraph{\textbf{The rank-three CMF $K(a,b,g)$.}}
In the moment coordinate $t$, the weight is
\[
 w_K(t)=t^{g-1/2}(t^2+6t+25)^a(t-25)^{-b}.
\]
The contour is the image of the segment $-1-2i\to-1+2i$ under $t=x^2$,
with the square root continued along that path.
Section~\ref{sec:pi-integral-data} gives the original integral coordinates
and the normalization used for the $\pi$ embedding.

\smallskip
\paragraph{\textbf{Logarithm CMFs.}}
The same procedure gives the logarithm CMFs $\mathcal H$, $\mathcal S$ and $\mathcal R$ from the weights
\[
\begin{aligned}
\mathcal H:&\quad (t-1)^u(t-\tfrac23)^v(t-\tfrac12)^w t^{-m},\\
\mathcal S:&\quad t^u(t-25)^v(t-49)^w(1225-t)^{-m},\\
\mathcal R:&\quad t^{-b}\prod_{j=1}^6 f_j(t)^{e_j}.
\end{aligned}
\]
Their ranks are $3$, $3$ and $8$, respectively; the coordinates are
$(u,v,w,m)$, $(u,v,w,m)$ and $(e_1,\ldots,e_6,b)$. Here
\[
\begin{aligned}
(f_1,\ldots,f_6)={}&(t-1,t-\tfrac23,t-\tfrac43,5t-4,\\
&17t^2-34t+16,19t^2-36t+16),\\
\boldsymbol\alpha={}&(0.704324,0.552418,0.447582,\\
&0.109072,0.038934,0.054368).
\end{aligned}
\]

\Needspace{7\baselineskip}
\subsubsection{Recurrence-based moment extensions: the common construction}
\label{app:auxiliary-moment-fields}
\smallskip
\paragraph{\textbf{Source sequence and auxiliary parameter.}}
The input is a scalar sequence $J_n$ with a known recurrence.
Define the linear functional $L_J$ on polynomials by $L_J(t^n)=J_n$.
Introducing a second parameter gives
\[
 S_J(u,v)=L_J(t^u(1-t)^v)
 =\sum_{k=0}^v(-1)^k\binom vk J_{u+k}.
\]
\smallskip
\paragraph{\textbf{State and shifts.}}
For nonnegative integer $v$, the displayed sum is finite. At $v=0$,
$S_J(u,0)=J_u$, so the original sequence is recovered on the horizontal
axis. Use the row state
\[
 \mathcal S(u,v)=\bigl(S_J(u,v),S_J(u+1,v),\ldots,S_J(u+r-1,v)\bigr).
\]
The $v$-axis is our binomial-transform extension, not an additional
parameter asserted in the source paper.

\smallskip
\paragraph{\textbf{Reduction from the source recurrence.}}
If
$\sum_{i=0}^h a_i(n)J_{n+i}=0$ and $d=\max_i\deg a_i$, put
\[
 \Omega=\sum_{i=0}^h t^{i-1}a_i(\theta-1),\qquad
 \theta=t\frac{d}{dt},\qquad
 \sum_{j=0}^r q_jt^j
 =\frac{\Omega[t^{u+1}(1-t)^{v+d}]}{t^u(1-t)^v}.
\]
The resulting shift matrices are
\[
 M_u=C(q_0,\ldots,q_r),\qquad M_v=I_r-M_u.
\]

\smallskip
\paragraph{\textbf{Embedding of the source sequence.}}
We use nonnegative forward walks from
$(0,0)$ in direction $(1,0)$, so the walk depth is the source index.
The notebook helpers evaluate the higher-rank matrices at integer
points without expanding all symbolic entries.

\Needspace{7\baselineskip}
\subsubsection{Recurrence-based families: source data and initial values}
\smallskip
\paragraph{\textbf{The $\zeta(7)$ partial sums: the rank-fifteen CMF $\mathcal B_{15}$.}}
Let
\[
 J_N=\frac52\sum_{k=1}^N\frac{(-1)^{k+1}}{k^7\binom{2k}{k}}
 +\frac{25}{2}\sum_{k=1}^N
 \frac{(-1)^{k+1}H_{k-1}^{(4)}}{k^3\binom{2k}{k}},\qquad
 H_m^{(4)}=\sum_{j=1}^m j^{-4}.
\]
This is the partial-sum sequence of \cite[Eq.~(1.4)]{borweinBradley1997}.
Its scalar recurrence is obtained by eliminating the two summand
coordinates from
\[
 \begin{pmatrix}J_{n+1}\\a_{n+1}\\b_{n+1}\end{pmatrix}
 =\begin{pmatrix}
 1&5/[2(n+1)^4]&25/2\\
 0&\rho_n&0\\
 0&\rho_n/(n+1)^4&\rho_n
 \end{pmatrix}
 \begin{pmatrix}J_n\\a_n\\b_n\end{pmatrix},\quad
 \rho_n=-\frac{(n+1)^3}{2(n+2)^2(2n+3)},
\]
where $a_n=(-1)^{n+2}/[(n+1)^3\binom{2n+2}{n+1}]$ and
$b_n=a_nH_n^{(4)}$. After clearing denominators, the recurrence
has order three and coefficient degree $12$, giving rank $15$.
Both $J_n$ and the constant sequence $1$ satisfy it. Thus the seed
rows $(J_0,\ldots,J_{14})$ and $(1,\ldots,1)$ give the approximation $J_N$.
Here $M_v$ is singular at $v=0$ because it annihilates the constant
solution; no inverse $v$-step is used at this boundary.
The notebooks record exact source-sequence and off-axis checks,
and overlapping two-precision Arb enclosures at source index $1000$.
For the rank-$15$ CMF, conservation is supported by
the moment construction and exact sampled squares, not a fully
expanded symbolic matrix check.

\smallskip
\paragraph{\textbf{Eskandari: the rank-eight CMF $E(k,m)$.}}
Let $u_k$ satisfy Eskandari's second-order recurrence
\cite{eskandari2026catalan}, and set
$U(k,m)=\sum_{\ell=0}^m\binom m\ell(-16)^\ell u_{k+\ell}$.
The $k$-shift advances the source index. The $m$-shift satisfies
$U(k,m+1)=U(k,m)-16U(k+1,m)$ by the binomial identity.
The rank-eight state is $(U(k+j,m))_{j=0}^7$.
If $\sum_{j=0}^8c_j(k,m)U(k+j,m)=0$ is the induced recurrence,
then $M_k=C(c)$ and $M_m=I_8-16C(c)$.

\Needspace{7\baselineskip}
\subsubsection{Dvornicich--Viola moment CMFs}
\smallskip
\paragraph{\textbf{Parameters and source.}}
These two CMFs use the Beukers moment functional associated with
$\zeta(s)$, for $s=2,3$. Their three parameters change the powers of
$t$, $1-A_st$ and $1-B_st$. Set $(A_2,B_2)=(11,12)$, $(A_3,B_3)=(34,35)$,
$h=t^{e_0}(1-A_st)^{e_1}(1-B_st)^{e_2}$ and $\theta=t\,d/dt$. Put
\[
\begin{aligned}
\Omega_2&=t(\theta+1)^2+11\theta^2+11\theta+3-t^{-1}\theta^2,\\
\Omega_3&=t(\theta+1)^3-34\theta^3-51\theta^2-27\theta-5+t^{-1}\theta^3,\\
\sum_{j=0}^{2s+2}q_jt^j&=h^{-1}\Omega_s[th(1-A_st)^s(1-B_st)^s].
\end{aligned}
\]
\smallskip
\paragraph{\textbf{State and shift matrices.}}
Then $D_s$ is the rank-$(2s+2)$ CMF with matrices
$(M_{e_0},M_{e_1},M_{e_2})=(C(q),I-A_sC(q),I-B_sC(q))$.
The state consists of consecutive $e_0$-shifts of the Beukers moment
functional applied to $h$ \cite{dvornicichViola1990}.

\Needspace{7\baselineskip}
\subsubsection{Ramanujan Machine: explicit rank-two shift matrices}
The catalogue families $\mathcal R^g_{i,j}$ and $\mathcal R^f_{i,j}$
are specified directly by their shift matrices. The indices $i,j$ are
fixed family labels; $(n,k)$ are the two CMF coordinates. The horizontal
matrix advances $n$ by one. The vertical matrix advances $k$ by one
for the $g$ family and by two for the $f$ family.

\smallskip
\paragraph{\textbf{Horizontal shifts.}}
Define the horizontal coefficients: \[
\begin{aligned}
 b&=-2n(n+i)(n+j)(n+k),\\
 a_g&=3n^2+(2i+2j+3)n+(i+1)(j+1),\\
 a_f&=3n^2+[2(i+j+k)+3]n+(i+1)(j+1)(k+1)-ijk.
\end{aligned}
\]
The horizontal matrices are
\[
 X_g=\begin{pmatrix}0&b\\1&a_g\end{pmatrix},\qquad
 X_f=\begin{pmatrix}0&b\\1&a_f\end{pmatrix}.
\]
\smallskip
\paragraph{\textbf{Vertical shifts.}}
The vertical matrices $Y_g$ and $Y_f$, for steps $k\mapsto k+1$ and $k\mapsto k+2$, respectively, are
\[
\begin{aligned}
Y_g&=\begin{pmatrix}u&b\\1&v\end{pmatrix},\qquad
Y_f=\tfrac12\begin{pmatrix}k^2+k-n^2+n&b\\1&w\end{pmatrix},\\
u&=4k^2+2kn+2(1-i-j)(k+n)-2n^2,\\
v&=4k^2+2kn+(4-2i-2j)k+n^2+n+(i-1)(j-1),\\
w&=k^2+2kn+(i+j+2)k+2n^2+2(i+j+1)n+(i+1)(j+1).
\end{aligned}
\]

\Needspace{8\baselineskip}
\subsection{Embedding initial values and numerical evaluation}
\label{app:cmf-recoveries}
Table~\ref{tab:appendix-reconstructions} records the supplementary
recoveries. This subsection distinguishes their numerical scores from
published bounds and specifies the initial values used for the
experimental constructions.

\paragraph{Historical qualifications.}
The supplementary table reports numerical sequence exponents. These
should be distinguished from historical irrationality-measure bounds.
For a published irrationality-measure bound $\mu_0$, the comparison quantity
$\delta_{\mathrm{eq}}=1/(\mu_0-1)$ is not necessarily the exponent of a
particular approximation sequence. Brown--Zudilin's $\zeta(5)$ construction
\cite{brownZudilin2022zeta5}
gives a sequence bound $\delta\geq-0.14$, not an irrationality-measure bound.
The Zagier sequence for $L(2,\chi_{-3})$ has $\delta=-0.47651$;
Calegari--Dimitrov--Tang's quantitative result reports
$\mu(L(2,\chi_{-3}))\leq24781$ \cite{calegariDimitrovTang2025Holonomy}.
The Bai and OpenAI $\pi$ entries are preprints; the latter does not specify
a CMF sequence with exponent $\delta=1$. Mignotte's $\pi$ construction
uses degree-five Hermite--Pad\'e forms, not a scalar rational approximant.
The data marked as numerical in the embedding table are not claimed as proved
arithmetic bounds.

For the recovered Dauguet--Zudilin construction, the initial point and
direction in Table~\ref{tab:appendix-reconstructions} encode
\cite[Eq.~(27)]{dauguetZudilin2014}. The approximation is
\[
 r_N=\frac{\Phi_{x_0,v}(N)_{10}}{\Phi_{x_0,v}(N)_{00}}.
\]

\Needspace{7\baselineskip}
\subsubsection{Experimental approximations}
\label{sec:appendix-extractions}

We give the data for the two experimental rows in
the main results table. For each forward or dual walk $W_N$, wherever $(W_N)_{00}\ne0$, set
\[
 c_j=\frac{(W_N)_{j0}}{(W_N)_{00}}.
\]

\paragraph{\texorpdfstring{The sum $\zeta(2)+\zeta(3)$}{The sum zeta(2)+zeta(3)}}
Use the ${}_4F_3(1)$ CMF with
\[
 x_0=(7,8,6,9;2,3,16),\qquad v=(11,13,15,16;2,4,30),
 \qquad W_N=\Phi_{x_0,v}(N).
\]
The rational approximation is
\[
 E_{2,3}=\frac{N_{2,3}}{D_{2,3}},
\]
where
\begin{align*}
N_{2,3}&=320242182383092497c_1-5523372274576928103c_2\\
          &\quad-18266336406752995,\\
D_{2,3}&=112484439099648000c_1-1940073689296512000c_2\\
          &\quad-6416014872960000.
\end{align*}
This is the experimental trajectory with exponent near $-0.00018$;
it differs from the recovered Dauguet--Zudilin trajectory in
Table~\ref{tab:appendix-reconstructions}.

\paragraph{\texorpdfstring{The value $L(2,\chi_{-3})$}{The value L(2,chi(-3))}}
Use the ${}_4F_3(\tfrac14)$ CMF with
\[
 x_0=(\tfrac12,\tfrac12,\tfrac12,\tfrac12;
       \tfrac32,\tfrac32,\tfrac32),\qquad
 v=(5,7,6,8;17,18,8),\qquad W_N=\Phi_{x_0,v}(N)^{-\mathsf T}.
\]
The numerator and denominator initial values used in the experiment then define the approximation
\[
 E_L=\frac{-16c_1-8}{-54c_1-108c_2-72c_3-9}.
\]
The corresponding special-value expression reduces to the stated
Dirichlet value. Writing $\psi_1$ for the trigamma function,
its duplication and series formulas \cite[5.15.7 and 5.15.1]{NIST:DLMF} give
\[
 \psi_1(\tfrac13)+\psi_1(\tfrac56)=4\psi_1(\tfrac23),
\]
and hence
\[
 \frac{3\psi_1(1/3)-\psi_1(5/6)}{36}
 =\frac{\psi_1(1/3)-\psi_1(2/3)}9
 =\sum_{k\ge0}\left(\frac1{(3k+1)^2}-\frac1{(3k+2)^2}\right)
 =L(2,\chi_{-3}).
\]
The numerical approximation exponents in the main results table
refer to the reduced rational outputs of these approximations.

\FloatBarrier
\Needspace{8\baselineskip}
\subsection{The integral CMF for \texorpdfstring{$\pi$}{pi} and its recoveries}
\label{app:pi-build-recovery}
\Needspace{7\baselineskip}
\subsubsection{Integral coordinates for \texorpdfstring{$\pi$}{pi}}
\label{sec:pi-integral-data}
Put
\[
 H(x)=\bigl((x-4)^2+4\bigr)\bigl((x-6)^2+4\bigr),
 \qquad
 K(a,b,g)=-i\int_{4-2i}^{4+2i}
       \frac{H(x)^a(x-5)^{2g}}{x^b(x-10)^b}\,dx,
\]
where the contour is the straight segment.  For the nonnegative integral
parameters used below, with $b>0$, the integrand is rational and has no pole
on the contour.  For the CMF calculation, use the coordinates
\begin{equation}\label{eq:pi-I-K-coordinates}
 I(A,B,G)=K(2A,2B+10,G),\qquad
 \mathcal S_I(A,B,G)=
 \bigl(I(A,B,G),I(A,B,G+1),I(A,B+1,G)\bigr).
\end{equation}
In particular, half-integral values of $B$ allow odd values of $b$.
The integral decomposition belongs to $\mathbb Q+\mathbb Q\pi$:
partial fractions at $0$ and $10$ have opposite residues, and the logarithmic
contribution uses $\arctan(1/2)+\arctan(1/3)=\pi/4$.

\Needspace{7\baselineskip}
\subsubsection{Shift matrices and their normalization}
The three shift matrices are specified by polynomial identities. Write
\begin{gather*}
 y=(x-5)^2,\qquad h(y)=y^2+6y+25,\qquad V=y-25,\\
 \boldsymbol\beta=(1,y,V^{-2}),\qquad
 (F_A,F_B,F_G)=(h^2,V^{-2},y).
\end{gather*}
For a polynomial $P$ of degree at most six, define
\begin{align*}
 \mathcal T(P)={}&(1+2G)hPV
 +2y\bigl((h'P+hP')V-3hP\bigr)\\
 &+4Ayh'PV-(4B+20)yhP.
\end{align*}
The $j$th column of $M_i$, together with a polynomial $P_{ij}$ of degree at
most six, is specified by the polynomial identity
\begin{equation}\label{eq:pi-axis-coefficient-specification}
 V^4\left(F_i\beta_j-\sum_{k=0}^2(M_i)_{kj}\beta_k\right)
       =\mathcal T(P_{ij}),\qquad 0\leq j\leq2.
\end{equation}
Equating coefficients of $1,y,\ldots,y^9$ gives ten linear equations for
these ten unknowns over $\mathbb Q(A,B,G)$.  The coefficient determinant at
$(A,B,G)=(1,2,1)$ is $-216423567360000000000$, so the solution is generically
unique.  The supplied exact matrices satisfy all nine polynomial identities.
Indeed, if $f(x)=H(x)^{2A}(x-5)^{2G}/[x(x-10)]^{2B+10}$, then the
right-hand side divided by $V^4$ is
\[
 \frac1{f(x)}\frac{d}{dx}
 \left(f(x)\frac{(x-5)h(y)P_{ij}(y)}{V^3}\right).
\]
The primitive vanishes at both endpoints when $2A\geq0$.  Integration
therefore gives the exact row-state convention
$\mathcal S_I(z+e_i)=\mathcal S_I(z)M_i(z)$ wherever the entries are regular.

\Needspace{7\baselineskip}
\subsubsection{Initial state and approximation}
At the initial point $x_0=(1,2,1)$ one has
\begin{equation}\label{eq:pi-I-initial-rows}
 \mathcal S_I(x_0)=\frac{\pi\mathbf q-\mathbf p}
 {600600000000000000000000000000},
\end{equation}
where
\[
\begin{aligned}
 \mathbf p&=(-5236454140000,\ 249191746500000,\ 3609168431),\\
 \mathbf q&=(-4799995200000,\ 83002920000000,\ -3414290880).
\end{aligned}
\]
Thus the rational approximation for direction $v$ is
$r_N=(\mathbf p\Phi_{x_0,v}(N)e_0)/(\mathbf q\Phi_{x_0,v}(N)e_0)$.
Equation~\eqref{eq:pi-I-initial-rows} follows by rational partial fractions;
it does not rely on numerical recognition of $\pi$.

\Needspace{7\baselineskip}
\subsubsection{Embedding of published integral families}
The relation to the published integrals requires keeping the initial point
as well as the direction.  The Salikhov integral displayed in
\cite[Part~I]{zeilbergerZudilin2020} is $K(6n,10n+1,3n)$.
Consequently the initial point $(6,11,3)$ and direction $(6,10,3)$ in $K$
coordinates select its index shift $n\mapsto n+1$.
Bai's integral \cite[\S1, Eqs.~(1.1), (1.6)]{bai2026pi}, at indices $5570n$,
is exactly $K(3714n,5570n+1,1857n)$; this explains the initial point
$(0,1,0)$ and direction $(3714,5570,1857)$.
These are identities between integral families.  They do not transfer a
published arithmetic normalization to a different initial point.

For comparison, the improved Zeilberger--Zudilin integral
\cite[Part~II, Eq.~(1)]{zeilbergerZudilin2020} is
$-5K(2m,3m+1,m)$.
Its odd subsequence $m=2n+1$ is therefore
$-5I(1+2n,-3+3n,1+2n)$.
The initial point used here $(1,2,1)$ has the same direction $(2,3,2)$ but
adds ten to the denominator parameter $b$.
It is a fixed contiguous variant, with the initial values in \eqref{eq:pi-I-initial-rows}; its finite-depth scores
are not the published Zeilberger--Zudilin irrationality-measure bound.

The shift matrices and exact checks of the polynomial identities
and initial-state formula are provided in the supplementary
Git repository notebooks.

\subsubsection{Comparison at equal accuracy}
\label{app:equal-accuracy-plots}
The following figures compare the finite-depth approximation exponents
$\delta_N$ against the achieved accuracy $D_N$ (in decimal digits).
The legends identify the individual constructions. These plots show
finite numerical data, not proved asymptotic bounds. The data and plotting
notebook are provided in the supplementary Git repository.

\begin{figure}[!htbp]
\centering
\includegraphics[width=\linewidth,height=.37\textheight,keepaspectratio]{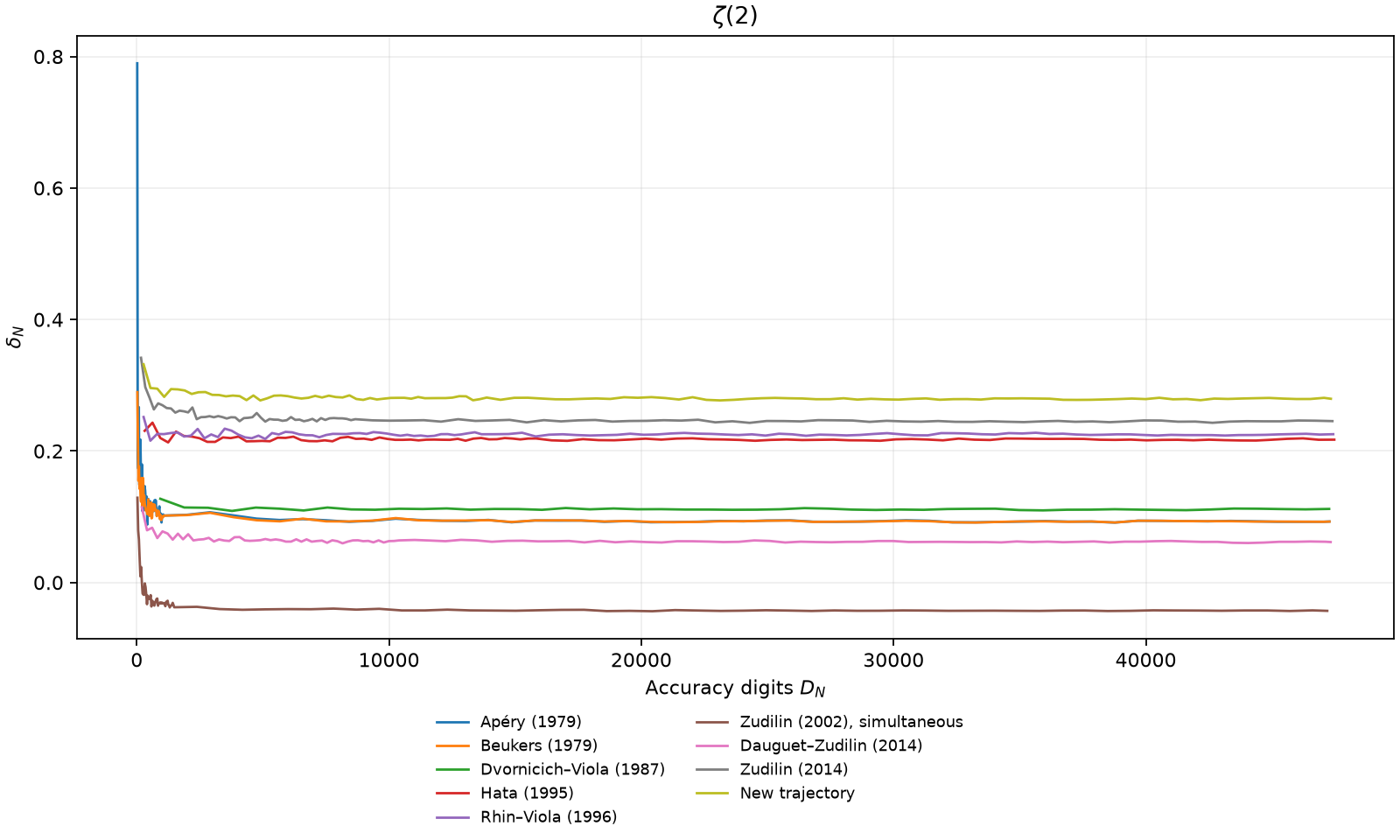}
\caption{Comparison for $\zeta(2)$: finite-depth exponent $\delta_N$ versus accuracy $D_N$ in decimal digits.}
\label{fig:equal-accuracy-zeta2}
\end{figure}

\begin{figure}[!htbp]
\centering
\includegraphics[width=\linewidth,height=.37\textheight,keepaspectratio]{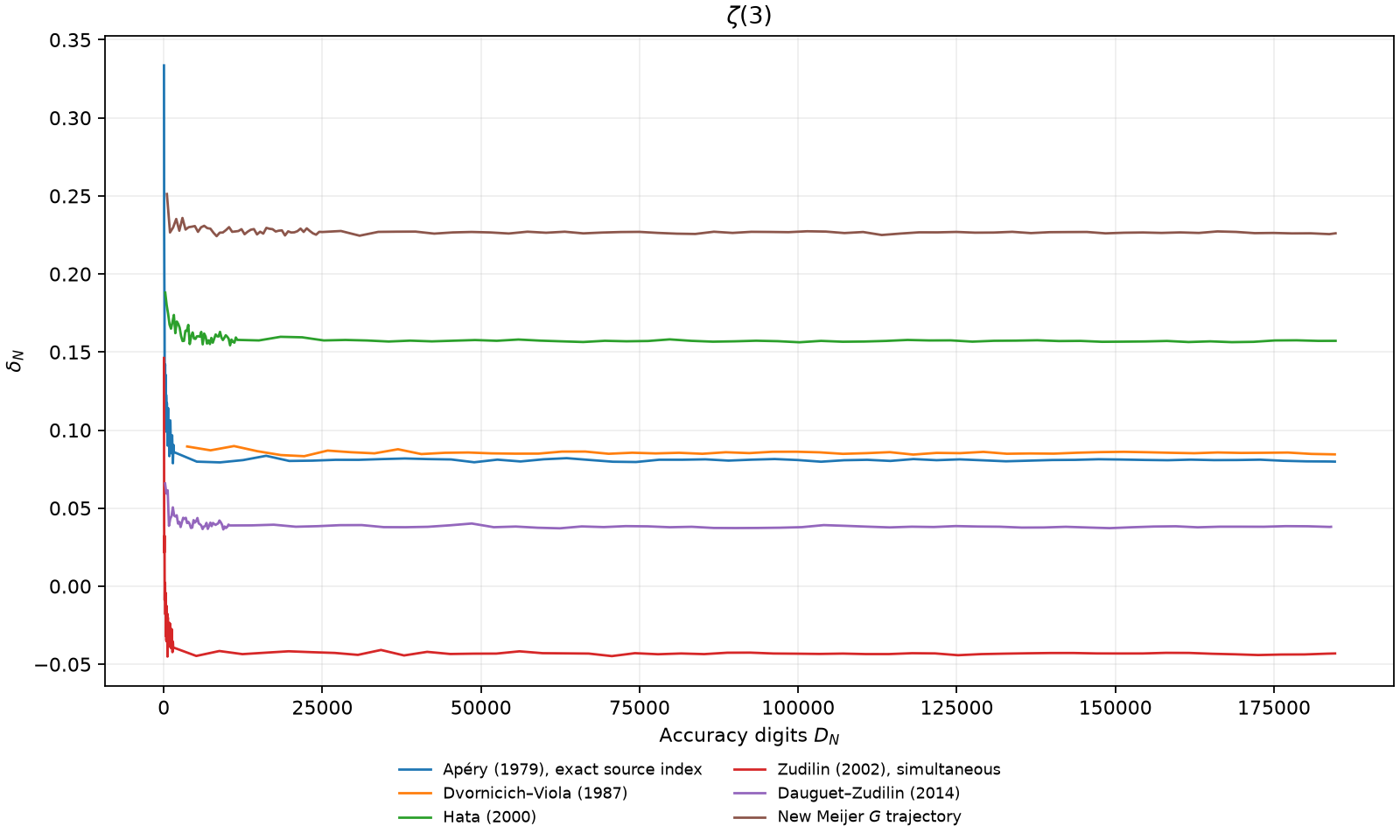}
\caption{Comparison for $\zeta(3)$: finite-depth exponent $\delta_N$ versus accuracy $D_N$ in decimal digits.}
\label{fig:equal-accuracy-zeta3}
\end{figure}

\clearpage

\begin{figure}[!htbp]
\centering
\includegraphics[width=\linewidth,height=.37\textheight,keepaspectratio]{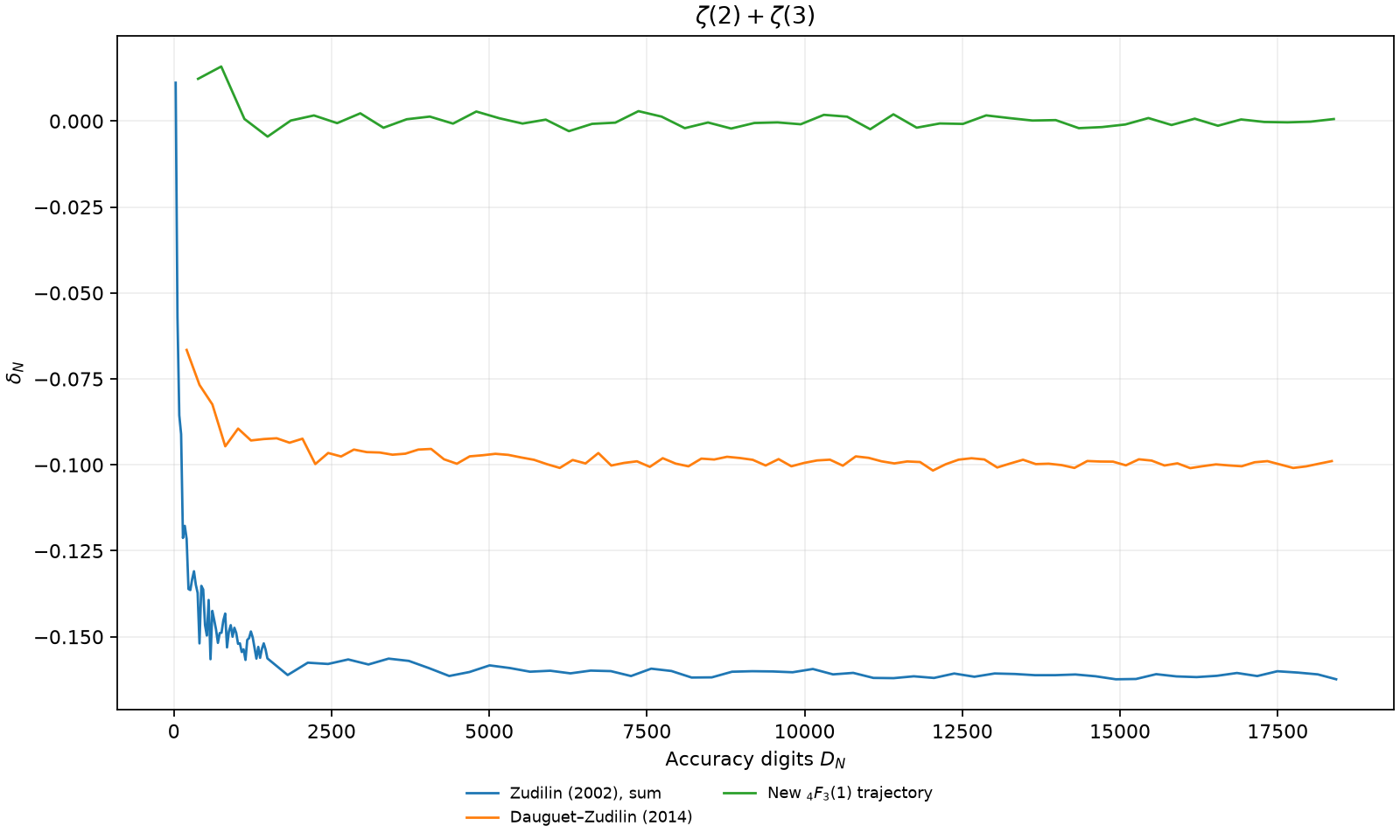}
\caption{Comparison for $\zeta(2)+\zeta(3)$: finite-depth exponent $\delta_N$ versus accuracy $D_N$ in decimal digits.}
\label{fig:equal-accuracy-zeta2-zeta3}
\end{figure}

\begin{figure}[!htbp]
\centering
\includegraphics[width=\linewidth,height=.37\textheight,keepaspectratio]{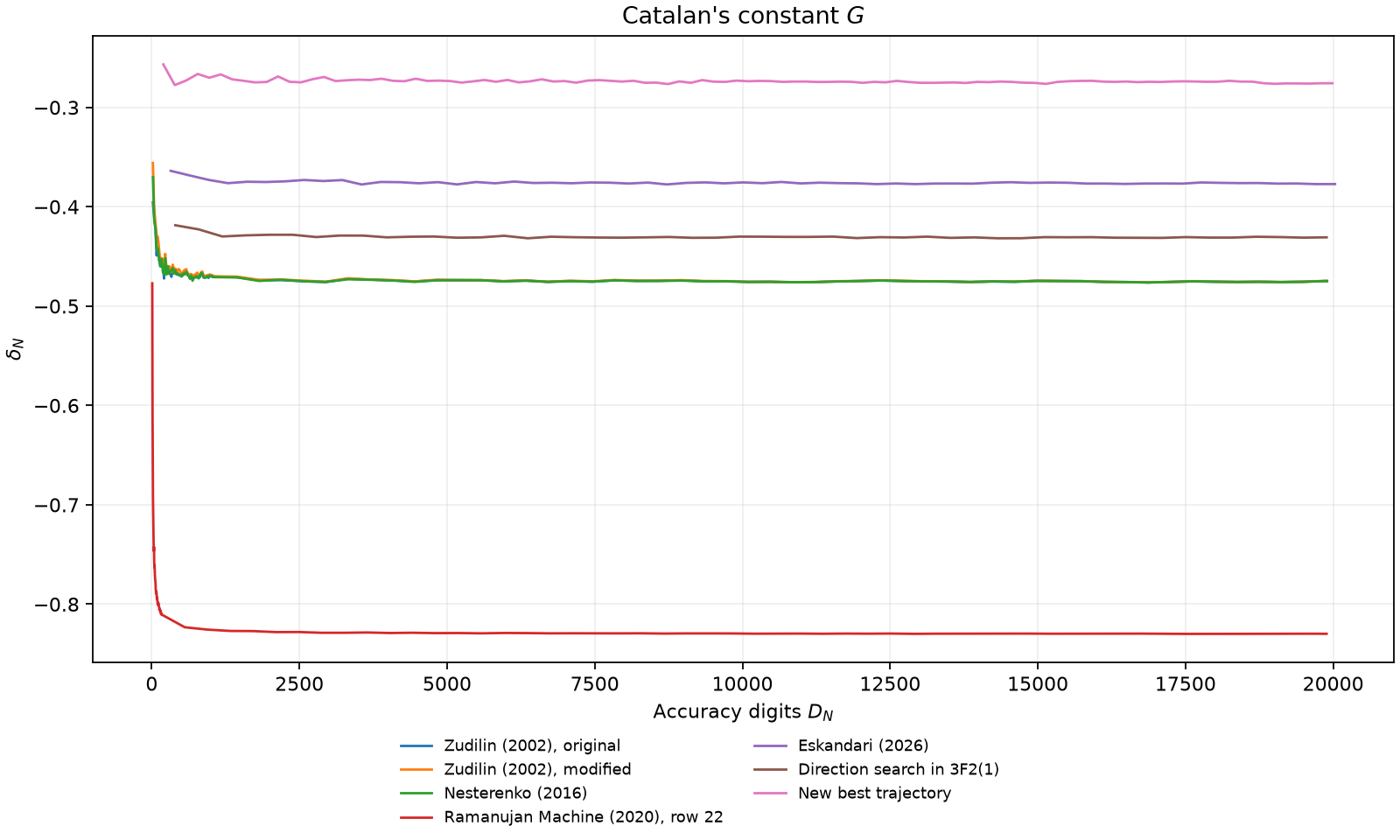}
\caption{Comparison for Catalan's constant $G$: finite-depth exponent $\delta_N$ versus accuracy $D_N$ in decimal digits.}
\label{fig:equal-accuracy-catalan}
\end{figure}

\clearpage

\begin{figure}[!htbp]
\centering
\includegraphics[width=\linewidth,height=.37\textheight,keepaspectratio]{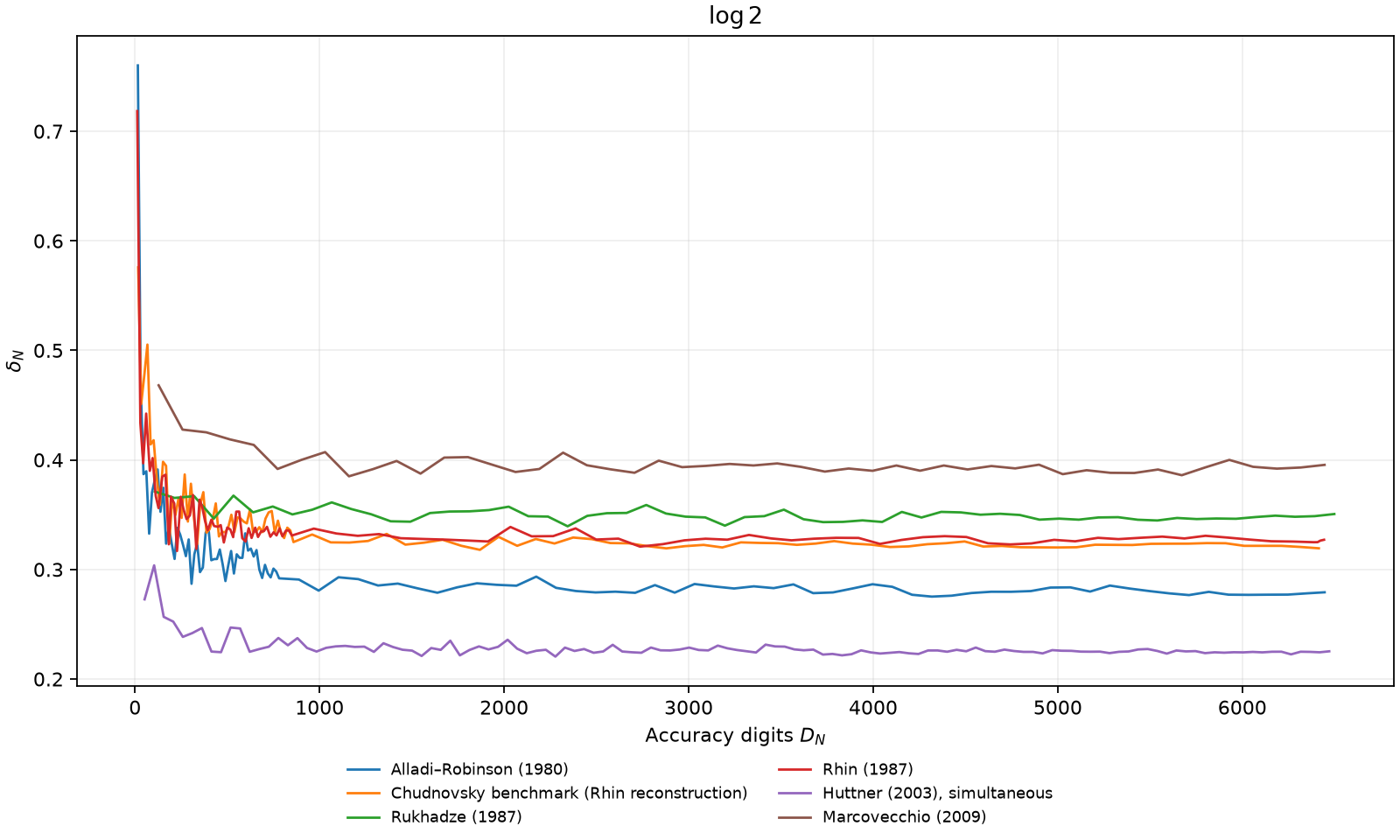}
\caption{Comparison for $\log 2$: finite-depth exponent $\delta_N$ versus accuracy $D_N$ in decimal digits.}
\label{fig:equal-accuracy-log2}
\end{figure}

\begin{figure}[!htbp]
\centering
\includegraphics[width=\linewidth,height=.37\textheight,keepaspectratio]{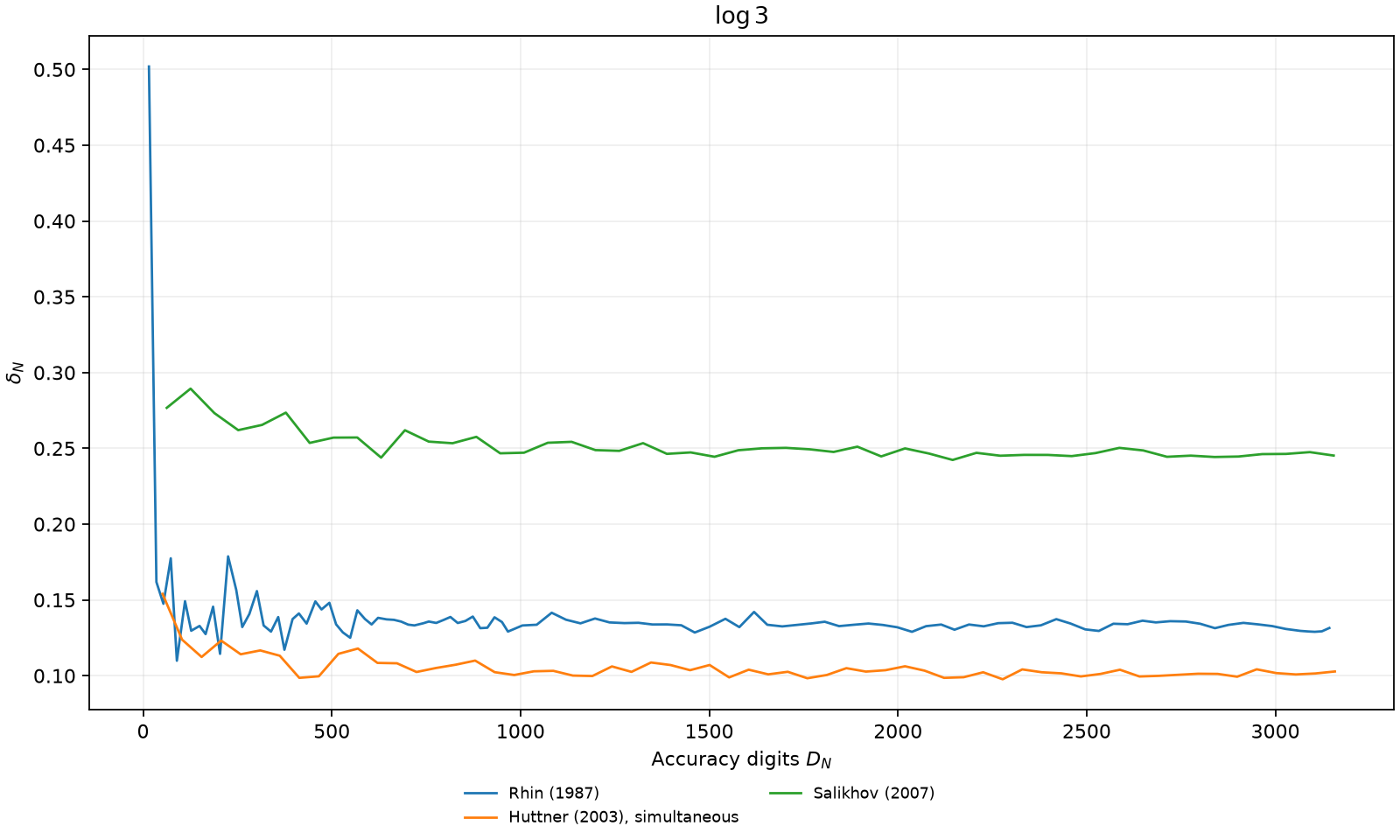}
\caption{Comparison for $\log 3$: finite-depth exponent $\delta_N$ versus accuracy $D_N$ in decimal digits.}
\label{fig:equal-accuracy-log3}
\end{figure}

\clearpage

\begin{figure}[!htbp]
\centering
\includegraphics[width=\linewidth,height=.37\textheight,keepaspectratio]{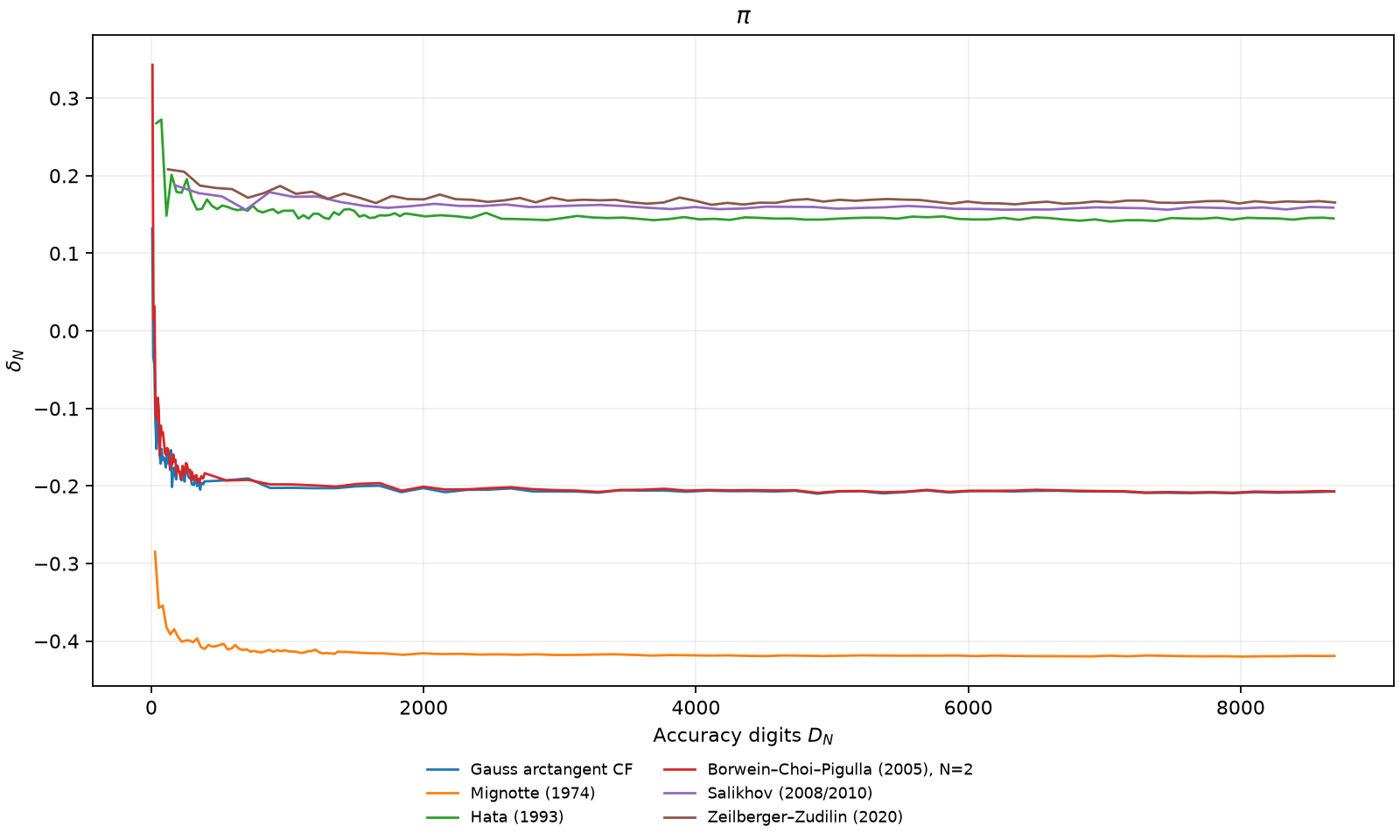}
\caption{Comparison for $\pi$: finite-depth exponent $\delta_N$ versus accuracy $D_N$ in decimal digits.}
\label{fig:equal-accuracy-pi}
\end{figure}

\begin{figure}[!htbp]
\centering
\includegraphics[width=\linewidth,height=.37\textheight,keepaspectratio]{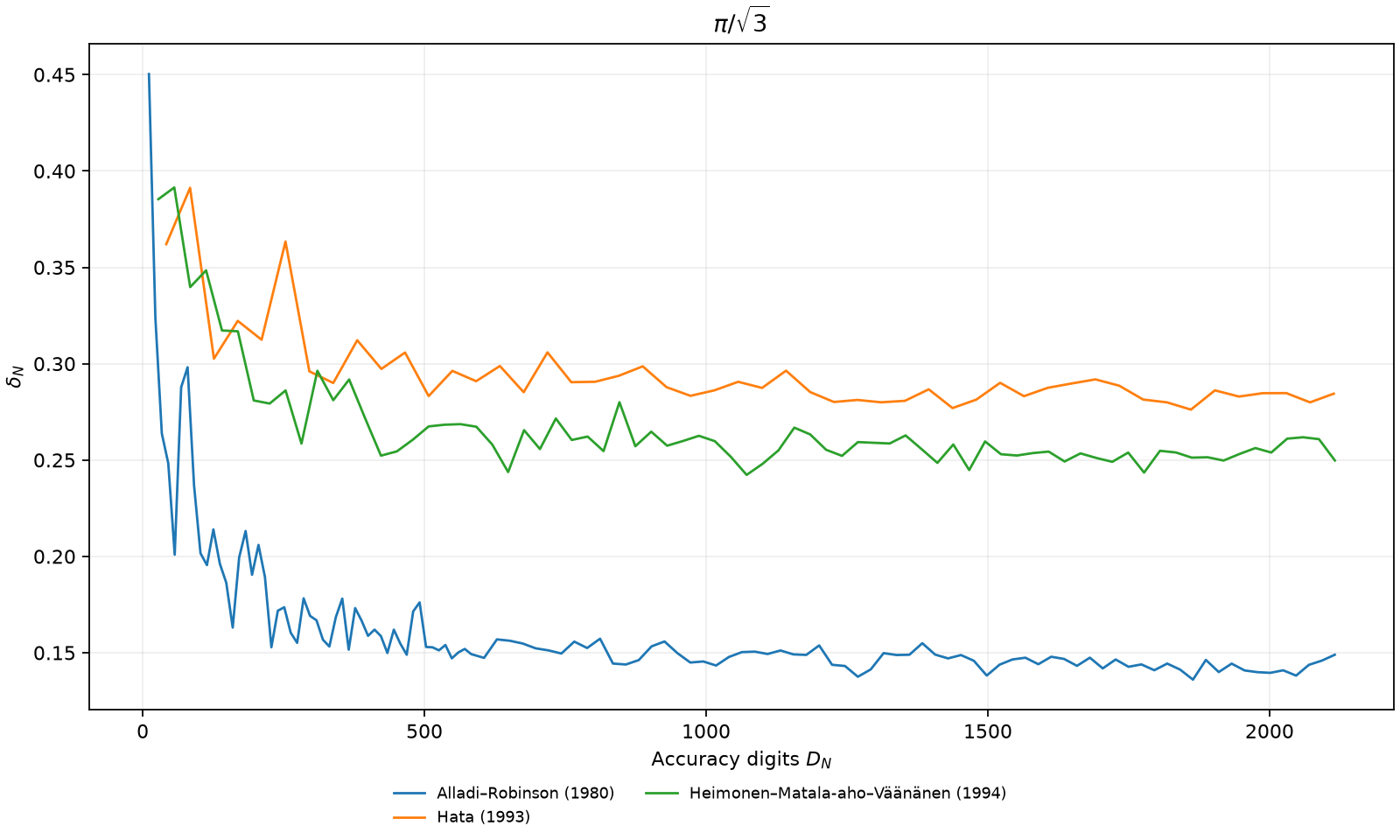}
\caption{Comparison for $\pi/\sqrt3$: finite-depth exponent $\delta_N$ versus accuracy $D_N$ in decimal digits.}
\label{fig:equal-accuracy-pi-sqrt3}
\end{figure}

\clearpage

\begin{figure}[!htbp]
\centering
\includegraphics[width=\linewidth,height=.37\textheight,keepaspectratio]{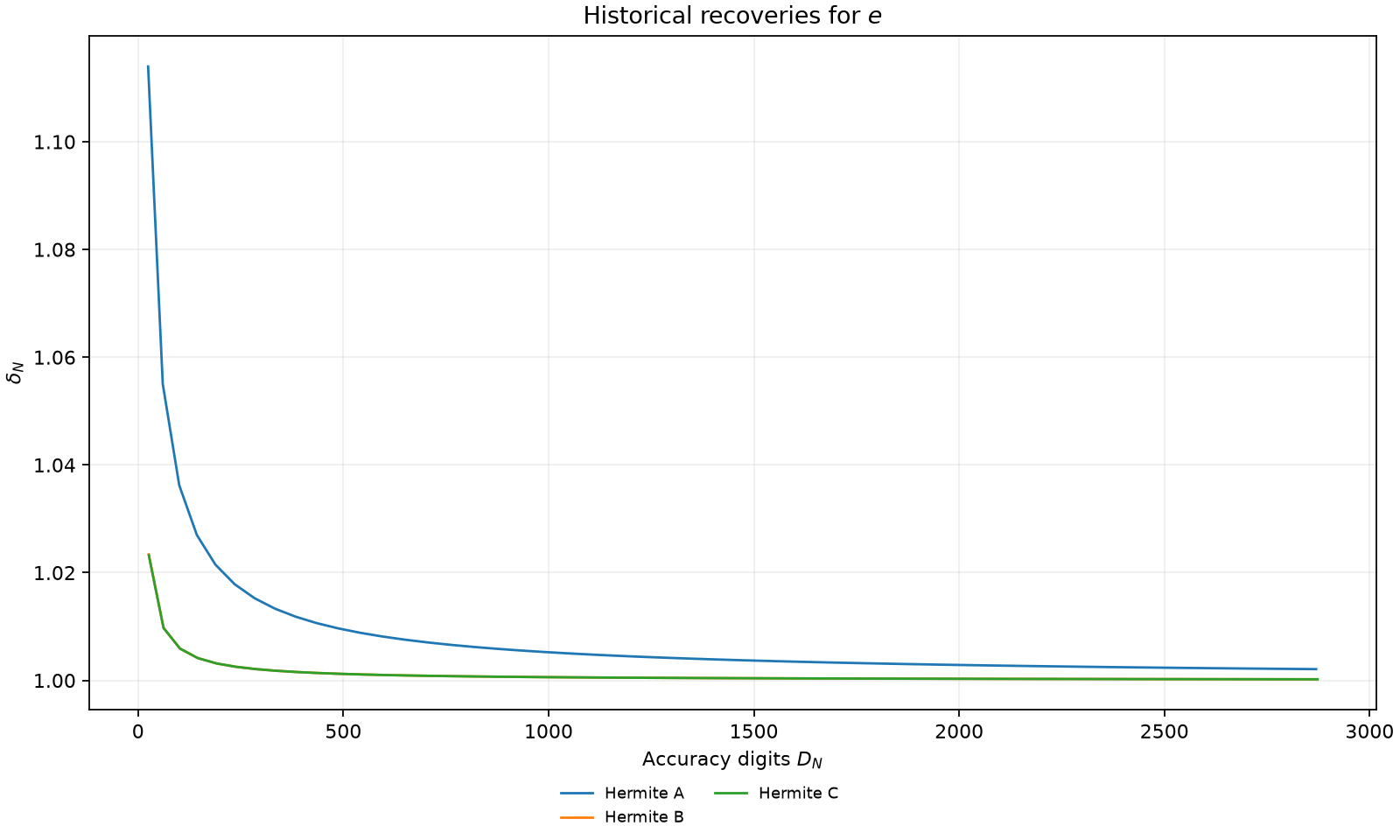}
\caption{Comparison for $e$: finite-depth exponent $\delta_N$ versus accuracy $D_N$ in decimal digits.}
\label{fig:equal-accuracy-e}
\end{figure}

\begin{figure}[!htbp]
\centering
\includegraphics[width=\linewidth,height=.37\textheight,keepaspectratio]{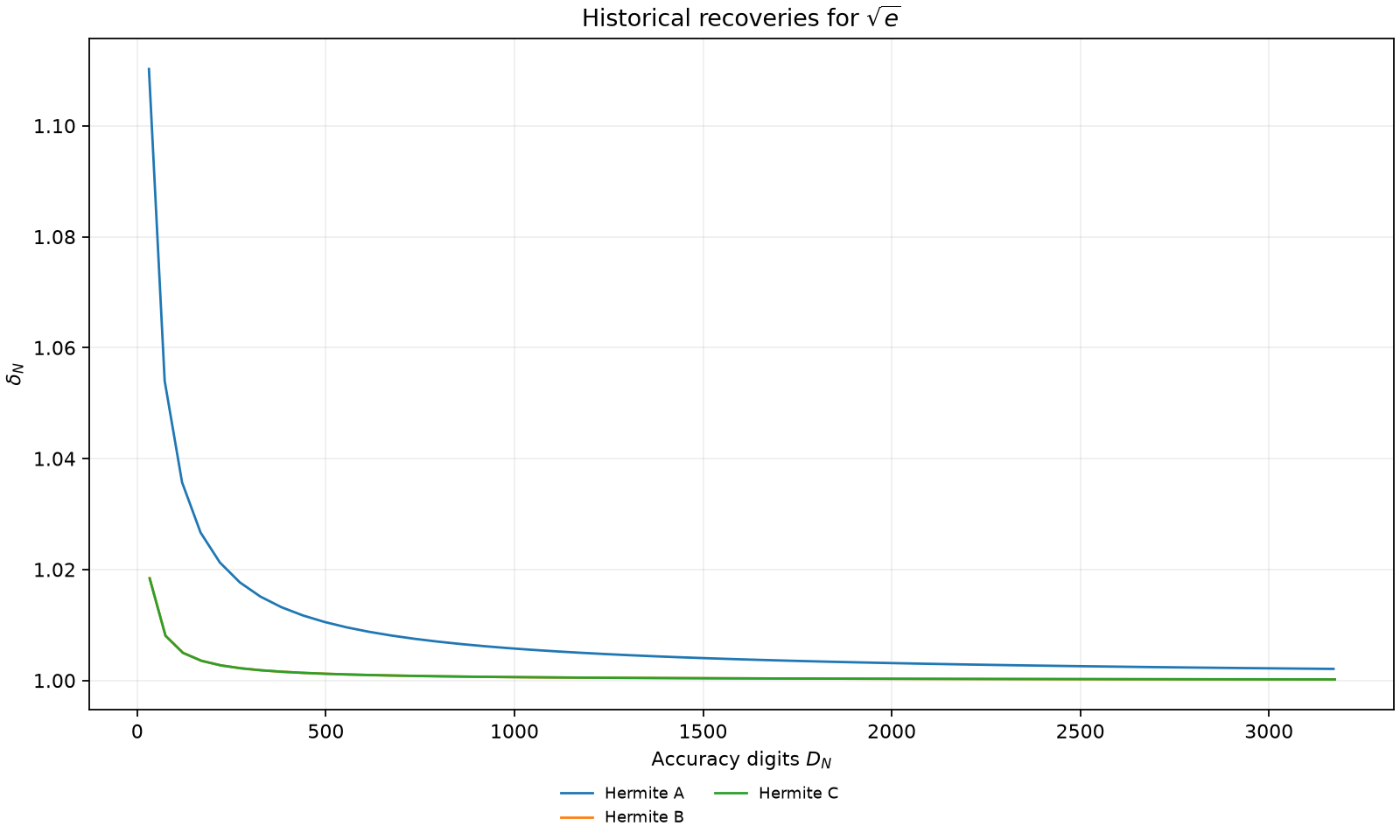}
\caption{Comparison for $\sqrt e$: finite-depth exponent $\delta_N$ versus accuracy $D_N$ in decimal digits.}
\label{fig:equal-accuracy-sqrt-e}
\end{figure}

\clearpage

\begin{figure}[!htbp]
\centering
\includegraphics[width=\linewidth,height=.37\textheight,keepaspectratio]{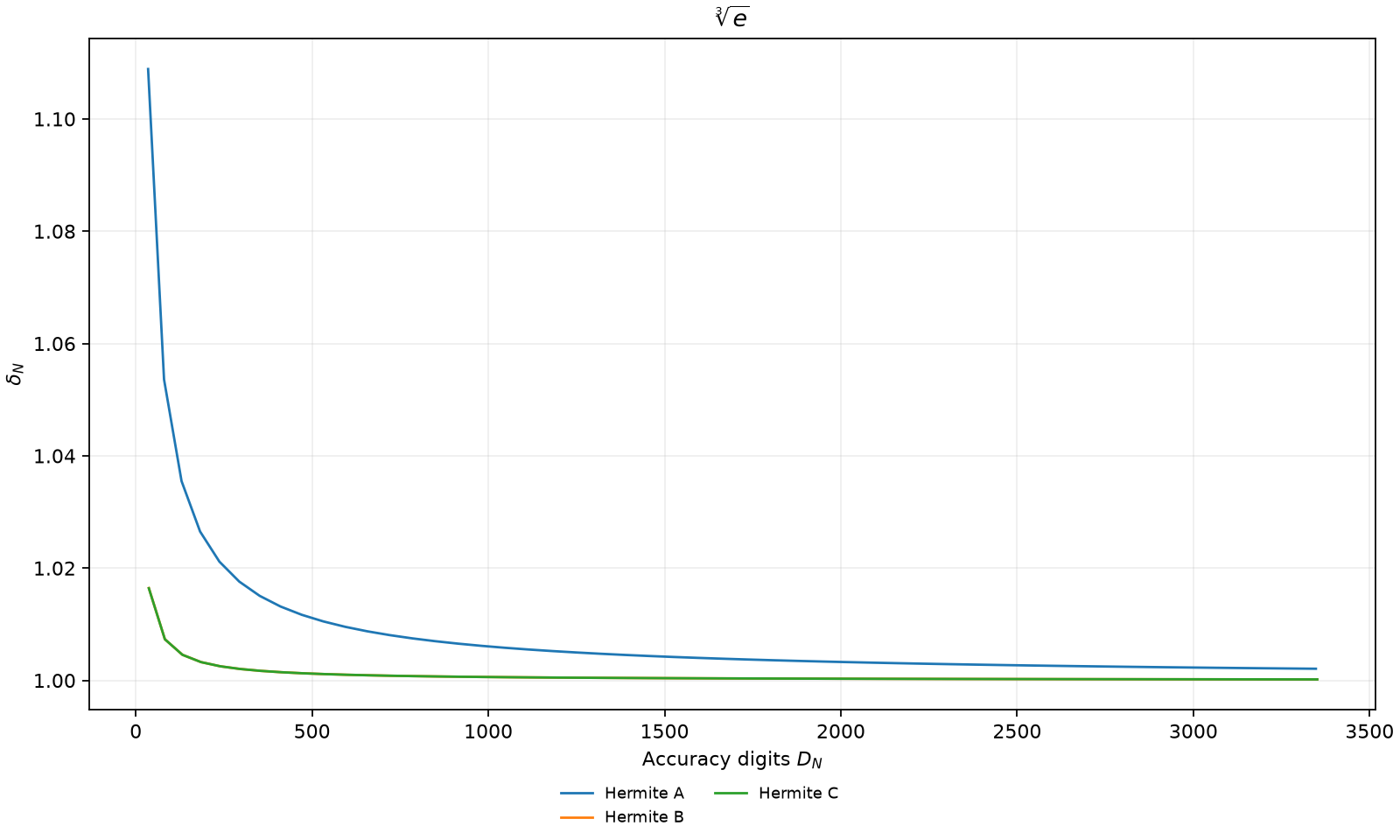}
\caption{Comparison for $\sqrt[3]{e}$: finite-depth exponent $\delta_N$ versus accuracy $D_N$ in decimal digits.}
\label{fig:equal-accuracy-cbrt-e}
\end{figure}

\begin{figure}[!htbp]
\centering
\includegraphics[width=\linewidth,height=.37\textheight,keepaspectratio]{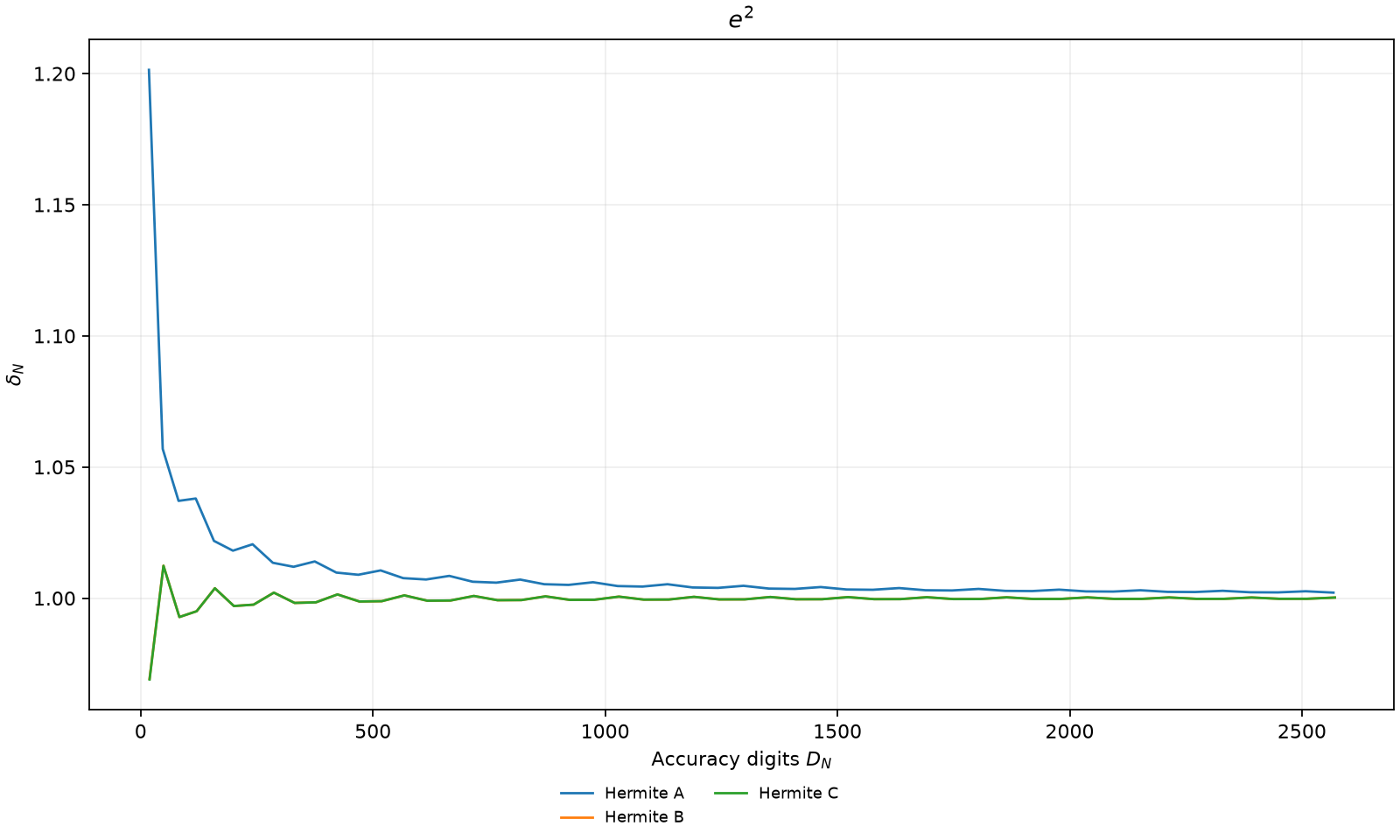}
\caption{Comparison for $e^2$: finite-depth exponent $\delta_N$ versus accuracy $D_N$ in decimal digits.}
\label{fig:equal-accuracy-e-squared}
\end{figure}

\clearpage

\begin{figure}[!htbp]
\centering
\includegraphics[width=\linewidth,height=.37\textheight,keepaspectratio]{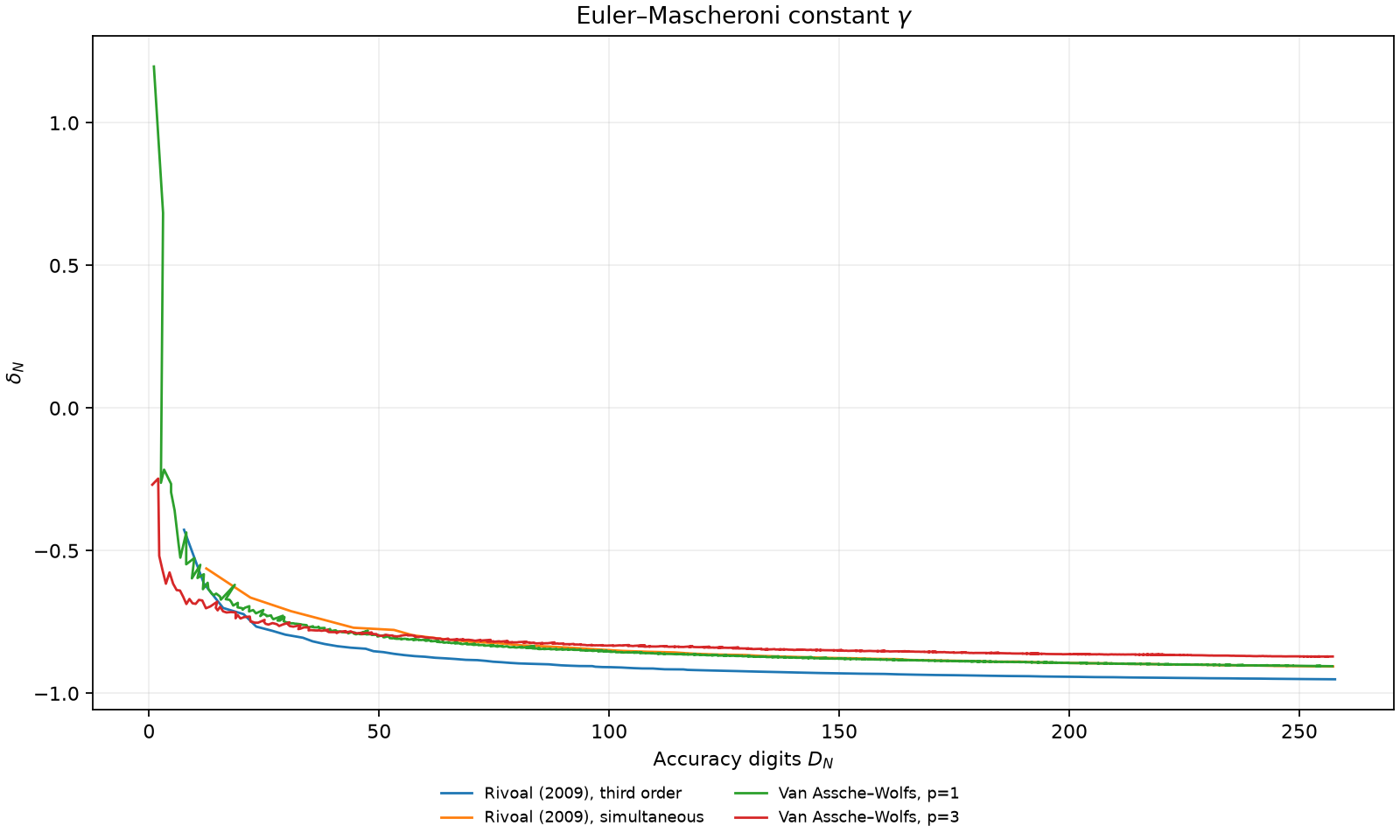}
\caption{Comparison for Euler\textquotesingle s constant $\gamma$: finite-depth exponent $\delta_N$ versus accuracy $D_N$ in decimal digits.}
\label{fig:equal-accuracy-euler-gamma}
\end{figure}

\clearpage

\section{Exact Calculations from Section~\ref{sec:proof_strategy}}
\label{Appendex: proof calculations}

\subsection{Subsection~\ref{subsec:z34} Proofs}
\label{app:zeta34-proof}

\begin{lemma}\label{lem:z34-recessive}\label{app:z34-recessive}
The state $\mathcal G_n$ is the recessive solution of $Q(n)$.
\end{lemma}

\begin{proof}
Leaving the argument free, write
\[
G_n(z)=G_{5,5}^{5,2}\!\left(
z\;\middle|\;
\begin{matrix}
-16n,-17n;\ 1+11n,1+12n,2+10n\\
-2n,-n,0,n,2n
\end{matrix}\right),
\qquad \theta=z\frac{d}{dz}.
\]
The Mellin--Barnes representation \cite[Eq.~16.17.1]{NIST:DLMF} is
\[
G_n(z)=\frac{1}{2\pi i}
\int_{2n+\frac12-i\infty}^{2n+\frac12+i\infty}
\frac{\Gamma(16n+1-s)\Gamma(17n+1-s)
      \prod_{j=-2}^{2}\Gamma(s+jn)}
     {\Gamma(s+11n+1)\Gamma(s+12n+1)\Gamma(s+10n+2)}
z^{-s}\,ds.
\]
The contour is traversed upwards. It separates the left poles,
$s\leq2n$, from the right poles, $s\geq16n+1$.
The substitution $s=2n+w$ puts the contour on the fixed line
$\mathcal C=\tfrac12+i\mathbb R$ without crossing any poles.
For $z>0$ the integral and its first two $\theta$ derivatives
converge absolutely.

Apply Euler's beta integral \cite[Eq.~5.12.1]{NIST:DLMF} to
\[
\frac{\Gamma(w+2n)}{\Gamma(w+13n+1)},\qquad
\frac{\Gamma(w+3n)}{\Gamma(w+14n+1)},\qquad
\frac{\Gamma(w+4n)}{\Gamma(w+12n+2)}.
\]
This gives
\[
\begin{aligned}
G_n(z)
&=\frac{z^{-2n}}{2\pi i(11n)!^2(8n+1)!}
\int_{\mathcal C}
\Gamma(w)\Gamma(w+n)\Gamma(14n+1-w)\Gamma(15n+1-w)z^{-w}\\
&\quad\times\left[\int_{[0,1]^3}
x^{w+2n-1}y^{w+3n-1}t^{w+4n-1}
(1-x)^{11n}(1-y)^{11n}(1-t)^{8n+1}
\,dx\,dy\,dt\right]dw.
\end{aligned}
\]
Apply the beta integral once more in the form
\[
\Gamma(w+n)\Gamma(15n+1-w)
=(16n)!\int_0^1r^{w+n-1}(1-r)^{15n-w}\,dr.
\]
All beta parameters have positive real parts on $\mathcal C$.
Absolute convergence permits interchanging the integrals.
With $u=xyt$, Mellin inversion \cite[Eq.~5.13.1]{NIST:DLMF} gives
\[
\frac{1}{2\pi i}\int_{\mathcal C}
\Gamma(w)\Gamma(14n+1-w)
\left(\frac{z(1-r)}{ur}\right)^{-w}\,dw
=(14n)!\left(\frac{ur}{z(1-r)+ur}\right)^{14n+1}.
\]
The argument is positive, and the contour condition
$0<\tfrac12<14n+1$ holds for every $n$. Hence
\[
\begin{aligned}
G_n(z)&=\frac{z^{-2n}(14n)!(16n)!}{(11n)!^2(8n+1)!}\\
&\quad\times\int_{[0,1]^4}
\frac{x^{16n}y^{17n}t^{18n}r^{15n}
 (1-x)^{11n}(1-y)^{11n}(1-t)^{8n+1}(1-r)^{15n}}
 {[z(1-r)+xytr]^{14n+1}}
\,dx\,dy\,dt\,dr.
\end{aligned}
\]
Since $1-r+xytr\geq xyt$, for $n\geq1$ another application of
the beta integral yields
\[
\begin{aligned}
0<G_n(1)
&\leq\frac{(14n)!(16n)!}{(11n)!^2(8n+1)!}
\int_{[0,1]^4}
x^{2n-1}y^{3n-1}t^{4n-1}r^{15n}\\
&\qquad\times
(1-x)^{11n}(1-y)^{11n}(1-t)^{8n+1}(1-r)^{15n}
\,dx\,dy\,dt\,dr\\
&=\frac{(16n)!(15n)!^2(2n-1)!(3n-1)!(4n-1)!}
        {(13n)!(12n+1)!(30n+1)!}
=:\beta_n.
\end{aligned}
\]
The factorial ratio satisfies
\[
\lim_{n\to\infty}\frac{\beta_{n+1}}{\beta_n}
=\frac{16^{16}15^{30}2^2 3^3 4^4}
       {13^{13}12^{12}30^{30}}
=\frac{1048576}{5961490723055315799}<1.
\]
Thus $\beta_n=O(\eta^n)$ for some $\eta<1$.

Applying $\theta^k$ inserts $(-s)^k=(-1)^k(w+2n)^k$ in
the Mellin--Barnes integrand. For $n\geq1$ and $k\leq2$,
the polynomial $(w+2n)^k$ has nonnegative coefficients in the
rising-factorial basis $(w)_j$:
\[
w+2n=(w)_1+2n,\qquad
(w+2n)^2=(w)_2+(4n-1)(w)_1+4n^2.
\]
Each factor $(w)_j$ replaces $\Gamma(w)$ by $\Gamma(w+j)$.
The same Mellin inversion therefore multiplies the positive
integrand at $z=1$ by
\[
(14n+1)_j\left(\frac{1-r}{1-r+ur}\right)^j
\leq(14n+1)_j.
\]
Evaluating the polynomial expansion at $w=14n+1$ gives
\[
|(\theta^kG_n)(1)|
\leq(16n+1)^k\beta_n
=O(n^k\eta^n),\qquad k=0,1,2.
\]

Put
\[
T(n)=\operatorname{diag}(1,n+1,(n+1)^2),\qquad
Q_\infty=\lim_{n\to\infty}T(n)Q(n)T(n+1)^{-1}.
\]
The exact matrix computations in
the supplementary material,
show that $Q_\infty$ has three simple, nonzero characteristic roots:
\[
\lambda_{1,2}\approx(2.803\pm6.487i)\cdot10^{25},
\qquad \lambda_3\approx7.203\cdot10^{-24}.
\]
The rational change of coordinates contributes only polynomial
factors. The asymptotic theory used in the CMF proof strategy,
together with the exponential decay just proved, identifies the
nonzero state $\mathcal G_n$ as the recessive solution corresponding
to $\lambda_3$.
\end{proof}

\begin{proof}[Proof of Lemma~\ref{lem:z34-selected-error}]
\phantomsection\label{app:z34-error}
The characteristic roots of the reduced dual system are
\[
\lambda_3^{-1},\qquad \lambda_1^{-1},\qquad \lambda_2^{-1},
\qquad |\lambda_3|^{-1}>|\lambda_1|^{-1}=|\lambda_2|^{-1},
\]
where $\lambda_1,\lambda_2,\lambda_3$ are the characteristic roots
of $Q_\infty$, with $\lambda_3$ recessive. The additional root $-1$
of the full Meijer G system does not occur in $Q(n)$.

Fix $(A,B)\in\mathbb Q^2\setminus\{(0,0)\}$ and put
$\alpha=A\zeta(3)+B\zeta(4)$.
The limit identification in the main text fixes the dominant dual
contribution. The source pairings show that its coefficient in the
denominator is nonzero, whereas
\[
\bigl(\alpha\rho_q-A\rho_{p,3}-B\rho_{p,4}\bigr)
K(0)\mathcal G_0^{\mathsf T}=0.
\]
Thus this contribution cancels in the error numerator.
With $\kappa_n=\prod_{j=0}^{n-1}c(j)$, the asymptotic error-rate
argument from the CMF proof method gives
\[
\begin{aligned}
|\alpha q_n-Ap_{3,n}-Bp_{4,n}|
&\leq|\kappa_n|\,|\lambda_1|^{-n}e^{o(n)},\\
|q_n|&=|\kappa_n|\,|\lambda_3|^{-n}e^{o(n)}.
\end{aligned}
\]
Dividing cancels the common scalar and yields
\[
\left|\alpha-\frac{Ap_{3,n}+Bp_{4,n}}{q_n}\right|
\leq\left(\frac{|\lambda_3|}{|\lambda_1|}\right)^n e^{o(n)}.
\]
Equivalently, if $\tau=|\lambda_3|/|\lambda_1|$, then for every
$\varepsilon>0$ the error is
$O_{A,B,\varepsilon}((\tau+\varepsilon)^n)$.

Put
\[
g=\log\left|\frac{\lambda_1}{\lambda_3}\right|.
\]
Exact root isolation and interval arithmetic are calculated in the supplementary material, it is shown that 
\[
112.80753<g<112.80754.
\]
Therefore $g>112.8075$, and the preceding estimate is precisely
\[
\left|A\zeta(3)+B\zeta(4)
       -\frac{Ap_{3,n}+Bp_{4,n}}{q_n}\right|
\leq e^{-(g-o(1))n}.
\]
\end{proof}

\subsection{Subsection~\ref{subsec:z23} Proofs}
\label{app:zeta2-proof}
\begin{lemma}\label{lem:z2-recessive}\label{app:z2-recessive}
The state $\mathcal G_n$ is the recessive solution of $Q(n)$.
\end{lemma}

\begin{proof}
We use the common factorial normalization $(n!)^{13}$.
Leaving the argument free, write
\[
G_n(z)=G_{4,4}^{4,1}\!\left(
z\;\middle|\;
\begin{matrix}
-11n;\ 1+10n,1+9n,2+11n\\
0,2n,n,3n
\end{matrix}\right),
\qquad \theta=z\frac{d}{dz}.
\]
Put $L=11n+1$ and $\beta_n=(11n)!/[(8n)!^2(8n+1)!]$.
The Mellin--Barnes representation \cite[Eq.~16.17.1]{NIST:DLMF} is
\[
G_n(z)=\frac{1}{2\pi i}\int_{\mathcal C}
\underbrace{
\frac{\Gamma(s)\Gamma(s+n)\Gamma(s+2n)\Gamma(s+3n)\Gamma(L-s)}
     {\Gamma(s+10n+1)\Gamma(s+9n+1)\Gamma(s+11n+2)}
}_{\displaystyle\Phi_n(s)}z^{-s}\,ds.
\]
We use the principal branch, with $\log(-1)=i\pi$.
The fixed contour $\mathcal C=\tfrac12+i\mathbb R$ is traversed
upwards; it separates the left poles, $s\leq0$, from the right
poles, $s\geq L$. The gamma estimate
\cite[Eq.~5.11.9]{NIST:DLMF} gives, for $k\leq2$,
\[
\left|\left(-\tfrac12-iv\right)^k
\Phi_n(\tfrac12+iv)z^{-\frac12-iv}\right|
=O_n\!\left((1+|v|)^{-13n-4+k}\right),
\]
uniformly for $\tfrac12\leq|z|\leq2$ and $|\arg z|\leq\pi$.
Thus the integral and its first two $\theta$ derivatives converge
absolutely at $z=-1$, and limits in $z$ may be taken under the
integral with the contour fixed.

For $|\arg z|<\pi$, apply Euler's beta integral
\cite[Eq.~5.12.1]{NIST:DLMF} to
\[
\frac{\Gamma(s+2n)}{\Gamma(s+10n+1)},\qquad
\frac{\Gamma(s+n)}{\Gamma(s+9n+1)},\qquad
\frac{\Gamma(s+3n)}{\Gamma(s+11n+2)}.
\]
This gives
\[
\begin{aligned}
G_n(z)
&=\frac{1}{2\pi i(8n)!^2(8n+1)!}
\int_{\mathcal C}\Gamma(s)\Gamma(L-s)z^{-s}\\
&\quad\times\left[\int_{[0,1]^3}
x^{s+2n-1}y^{s+n-1}t^{s+3n-1}
(1-x)^{8n}(1-y)^{8n}(1-t)^{8n+1}
\,dx\,dy\,dt\right]ds.
\end{aligned}
\]
All beta parameters have positive real parts on $\mathcal C$.
Absolute convergence permits interchanging the integrals.
With $u=xyt$, Mellin inversion \cite[Eq.~5.13.1]{NIST:DLMF} gives
\[
\frac{1}{2\pi i}\int_{\mathcal C}
\Gamma(s)\Gamma(L-s)(z/u)^{-s}\,ds
=\Gamma(L)\left(\frac{u}{z+u}\right)^L.
\]
The argument condition is unchanged because $u>0$, and the contour
condition $0<\tfrac12<L$ holds for every $n$. Consequently,
\[
G_n(z)=\beta_n\int_{[0,1]^3}
\frac{x^{13n}y^{12n}t^{14n}
      (1-x)^{8n}(1-y)^{8n}(1-t)^{8n+1}}
     {(z+xyt)^{11n+1}}\,dx\,dy\,dt.
\]
To pass to $z=-1$, take $\operatorname{Re}z\leq-1$,
$\operatorname{Im}z>0$, and put $d=1-xyt$. Since
\[
0\leq1-x,1-y,1-t\leq d\leq1,\qquad |z+xyt|\geq d,
\]
the absolute integrand is bounded by
$d^{24n+1}/d^{11n+1}=d^{13n}\leq1$.
Dominated convergence therefore gives, since the integrand has
constant sign at $z=-1$,
\[
0<|G_n(-1)|\leq\beta_n.
\]

Set $a_n=(n!)^{13}\beta_n$ and $\eta=11^{11}/4^{24}<1$.
For $n\geq1$, bounding each numerator factor by $11(n+1)$
and each denominator factor by $4(n+1)$ gives
\[
\frac{a_{n+1}}{a_n}
=\frac{(n+1)^{13}\prod_{j=1}^{11}(11n+j)}
       {\left(\prod_{j=1}^{8}(8n+j)\right)^2
        \prod_{j=2}^{9}(8n+j)}
\leq\eta.
\]
Hence $a_n=O(\eta^n)$.

Applying $\theta^k$ inserts $(-s)^k$ in the Mellin--Barnes
integrand. The identities $s=(s)_1$ and $s^2=(s)_2-(s)_1$
reduce these derivatives to the same Mellin inversion.
Each factor $(s)_j$ multiplies the cube integrand by
$(L)_j[z/(z+u)]^j$. At $z=-1$ its absolute value is
$(L)_jd^{-j}$. For $n\geq1$ and $j\leq2$, the resulting bound
$d^{13n-j}\leq1$ also justifies the boundary limit. Thus
\[
|(\theta G_n)(-1)|\leq L\beta_n,\qquad
|(\theta^2G_n)(-1)|\leq L(L+2)\beta_n,
\]
and consequently
\[
(n!)^{13}|(\theta^kG_n)(-1)|
=O(n^k\eta^n),\qquad k=0,1,2.
\]

Put
\[
T(n)=\operatorname{diag}(1,n+1,(n+1)^2),\qquad
Q_\infty=\lim_{n\to\infty}(n+1)^{13}T(n)Q(n)T(n+1)^{-1}.
\]
The exact matrix computations in the supplementary material show that $Q_\infty$ has three simple, nonzero characteristic roots:
\[
\lambda_{1,2}\approx-1.287\pm2.207i,\qquad
\lambda_3\approx-1.602\cdot10^{-27}.
\]
The normalized row state $(n!)^{13}\mathcal G_nT(n)^{-1}$
has one-step matrix $(n+1)^{13}T(n)Q(n)T(n+1)^{-1}$.
Since $|\lambda_{1,2}|>1$, the exponential decay just proved
and the asymptotic theory used in the CMF proof strategy identify
the nonzero normalized state as the recessive solution corresponding
to $\lambda_3$.
\end{proof}

\begin{proof}[Proof of Lemma~\ref{lem:z2-selected-error}]
\phantomsection\label{app:z2-error}
The normalized dual system has characteristic roots
\[
\lambda_3^{-1},\qquad\lambda_1^{-1},\qquad\lambda_2^{-1},
\qquad |\lambda_3|^{-1}>|\lambda_1|^{-1}=|\lambda_2|^{-1},
\]
where $\lambda_1,\lambda_2,\lambda_3$ are the characteristic roots
of $Q_\infty$, with $\lambda_3$ recessive. The unnormalized dual
solutions have the common factorial factor $(n!)^{13}$.

The limit identification in the main text fixes the dominant dual
contribution. The source pairings show that its coefficient in the
denominator is nonzero, whereas
\[
\bigl(\zeta(2)\rho_q-\rho_p\bigr)K(0)\mathcal G_0^{\mathsf T}=0.
\]
Thus this contribution cancels in the error numerator.
With $\kappa_n=\prod_{j=0}^{n-1}c(j)$, the asymptotic error-rate
argument from the CMF proof method gives
\[
\begin{aligned}
|\zeta(2)q_n-p_n|
&\leq|\kappa_n|(n!)^{13}|\lambda_1|^{-n}e^{o(n)},\\
|q_n|&=|\kappa_n|(n!)^{13}|\lambda_3|^{-n}e^{o(n)}.
\end{aligned}
\]
Dividing cancels both the common scalar and the common factorial
factor, giving
\[
\left|\zeta(2)-\frac{p_n}{q_n}\right|
\leq\left(\frac{|\lambda_3|}{|\lambda_1|}\right)^n e^{o(n)}.
\]
Equivalently, if $\tau=|\lambda_3|/|\lambda_1|$, then for every
$\varepsilon>0$ the error is $O_\varepsilon((\tau+\varepsilon)^n)$.

Put
\[
g=\log\left|\frac{\lambda_1}{\lambda_3}\right|.
\]
Exact root isolation and interval arithmetic in
the supplementary material give
\[
62.636575<g<62.636576.
\]
Therefore $g>62.63657$, and the preceding estimate is precise
\[
\left|\zeta(2)-\frac{p_n}{q_n}\right|
\leq e^{-(g-o(1))n}.
\]
Since $p_n/q_n$ is rational and $\zeta(2)$ is irrational, the error
is strictly positive whenever the approximant is defined.
\end{proof}

\subsection{The block decomposition for \texorpdfstring{$\zeta(2)$}{zeta(2)}}
\label{app:z2-blocks}

Conservation permits the four-block path of
Section~\ref{subsec:z23}. Its last block takes a finite
difference of order $11n$, leaving a polynomial part of degree at most
$13n-1$. The resulting binomial sums and their harmonic terms are given
in the supplementary notebook.

For $C_n=\prod_{u=9,10,11}\binom{un}{8n}$, put
\[
\begin{aligned}
 \rho(x,y)={}&\lfloor11x\rfloor-\lfloor y\rfloor-\lfloor11x-y\rfloor
             -3\lfloor8x\rfloor\\
 &+\sum_{u=9,10,11}
   \bigl(\lfloor ux+y\rfloor-\lfloor(u-8)x+y\rfloor\bigr),
 \qquad \phi(x)=\min_{0\le y<1}\rho(x,y).
\end{aligned}
\]
Legendre's formula and the divided-difference bounds give
\[
 \frac{C_nd_{12n}d_{13n}}{\mathcal D_n^{\mathrm{car}}}
 (p_n,q_n)\in\mathbb Z^2,\qquad
 \mathcal D_n^{\mathrm{car}}
 =\prod_{\sqrt{50n}<p\le13n}p^{\phi(n/p)}.
\]
Let $\mathcal D_n^{\mathrm{sh}}$ be the product of the primes
$p>\sqrt{50n}$ in the six open intervals
\[
 \left(\frac{13n}{2},8n\right),\quad
 \left(\frac{13n}{3},\frac{9n}{2}\right),\quad
 \left(\frac{13n}{4},\frac{10n}{3}\right),\quad
 \left(\frac{13n}{5},\frac{8n}{3}\right),\quad
 \left(\frac{13n}{6},\frac{9n}{4}\right),\quad
 \left(\frac{13n}{12},\frac{9n}{8}\right).
\]
For each such prime, the normalized sum reduces to
$\sum_{x\in\mathbb F_p}P(x)$ with $\deg P<p-1$, hence vanishes.
Thus, with $\mathcal D_n=\mathcal D_n^{\mathrm{car}}
\mathcal D_n^{\mathrm{sh}}$, the same integrality statement holds
after division by $\mathcal D_n$.

The prime number theorem gives
\[
 \lim_{n\to\infty}\frac1n
 \log\frac{d_{12n}d_{13n}}{\mathcal D_n}
 =25-\gamma-\frac{233}{120}<10.929906,\qquad
 \gamma=\int_{1/13}^{\infty}\frac{\phi(x)}{x^2}\,dx.
\]
The one-periodic profile $\phi$ has $94$ cells on $[0,1)$;
exact summation over $2000$ periods with a bounded tail gives
$\gamma>12.1284279747910$.
The shell calculations and integral bounds are reproduced in
the supplementary material
combines them with the spectral bounds to obtain Lemma~\ref{lem:z2-growth}.
The supplied formalization and full written proof are in the supplementary material.

\subsection{The arithmetic of the pure
\texorpdfstring{$\zeta(3)$}{zeta(3)} approximations}
\label{app:zeta3-measure}

Write
\[
 F_n(t)=\frac{\prod_{u=0}^{2}(t-(10+u)n)_{10n}}
 {(t+n)_{16n+1}(t+2n)_{14n+1}}
 =\sum_{j=n}^{17n}\left(\frac{a_j}{(t+j)^2}+\frac{b_j}{t+j}\right),
\]
and set
\[
 Q_n=\sum_ja_j,\qquad
 P_{3,n}=\sum_j\left(a_jH^{(3)}_{j+2n}
                       +\frac12b_jH^{(2)}_{j+2n}\right),\qquad
 C'_n=\frac{(15n)!^2}{(10n)!^3}.
\]
Gamma reflection and comparison of Laurent coefficients give
$C'_na_{2n+k}=(-1)^nB_{n,k}$ and
$C'_nb_{2n+k}=-(-1)^nB'_{n,k}$, including the simple-pole tails.
Hence $C'_n(Q_n,P_{3,n})=(-1)^n(\mathbf q_n,\mathbf p_{3,n})$,
so $P_{3,n}/Q_n=p_{3,n}/q_n$ for every $n\ge0$.
The error estimate of Lemma~\ref{lem:z34-selected-error} therefore applies.

The coefficient bounds $C'_na_j\in\mathbb Z$ and
$d_{16n}C'_nb_j\in\mathbb Z$ give
$6C'_nd_{19n}^3(Q_n,P_{3,n})\in\mathbb Z^2$.
For $p>29$, $p^2>30n$ and $p\nmid n$, a stronger common divisor
$\mathcal D'_n$ combines the factorial carry count with fourteen
first-order shells and the two second-order windows
\[
 \{n/p\}\in
 \left(\frac1{10},\frac2{19}\right)
 \cup\left(\frac{11}{12},\frac{15}{16}\right).
\]
The carry and shell counts are evaluated on the $270$ cells of the
Farey partition of order $29$, with periodic lifts.
The first-order terms cancel over $\mathbb F_p$; a trace identity
gives the additional factor in the two windows.
The resulting divisor satisfies, for $n\ge1$,
\[
 \frac{6C'_nd_{19n}^3}{\mathcal D'_n}(Q_n,P_{3,n})\in\mathbb Z^2,
 \qquad
 \liminf_{n\to\infty}\frac{\log\mathcal D'_n}{n}
 >30.458547243805.
\]
Together with
$\limsup n^{-1}\log|C'_nQ_n|<65.451536565866$, this gives
\[
 h<65.451536565866+57-30.458547243805<91.992990,
\]
as used in Theorem~\ref{thm:z23-measures}.

The sum identity and spectral calculations are in
the supplementary material.

\section{Exact Calculations for Harmonic-Weighted Gregory Series Irrationality Proof}
\label{app:gregory-series}
\subsection{The trajectory matrix}\label{app:hh-matrix}
For the trajectory of Theorem~\ref{thm:hh-main}, written with $z=1/M$, the trajectory
matrix (Definition~\ref{def:trajectory}) is
\begin{equation}\label{eq:hh-T}
 T(n)=\frac{9(2n+1)^2(6n+5)^2(6n+7)^2}{16z^3(n+1)^4(4n+1)^3(4n+3)^3}\,\bigl(P_{ij}(n,z)\bigr)_{0\le i,j\le2},
\end{equation}
where each $P_{ij}$ is a polynomial of degree $\le1$ in $z$:
\begin{align*}
 P_{00}&=z(4n+1)^3(16n+13),\qquad P_{01}=-\tfrac32z(4n+1)^3(28n^2+56n+27),\\
 P_{02}&=\tfrac14z(4n+1)^3(6n+7)(72n^2+146n+71),\\
 P_{10}&=12z(96n^3+132n^2+50n+5)-2(576n^3+876n^2+404n+59),\\
 P_{11}&=-2z(1504n^4+3856n^3+3246n^2+1018n+97)\\
       &\quad+3024n^4+8208n^3+7632n^2+2876n+383,\\
 P_{12}&=z(7680n^5+28864n^4+40096n^3+25020n^2+6742n+607)\\
       &\quad-\tfrac12(6n+7)(2592n^4+7128n^3+6708n^2+2570n+353),\\
 P_{20}&=4(z-1)(84n^2+104n+29),\\
 P_{21}&=-2(z-1)(432n^3+1044n^2+776n+173),\\
 P_{22}&=z(2176n^4+7840n^3+10128n^2+5472n+1017)\\
       &\quad-(6n+7)(360n^3+876n^2+650n+143).
\end{align*}
In particular $z^3T(n)$ is a polynomial of degree $\le1$ in $z$, and
\begin{equation}\label{eq:hh-detT}
 \det T(n)=\frac{729(2n+1)^6(6n+5)^6(6n+7)^6}{64z^6(n+1)^6(4n+1)^6(4n+3)^6}\ne0\qquad(n\ge0).
\end{equation}
With the diagonal coboundary (Definition~\ref{def:coboundary}) $U(n)=\operatorname{diag}(1,(n+1)^{-1},(n+1)^{-2})$, the matrix $\tilde T(n)=U(n)^{-1}T(n)U(n+1)$ has rational entries with finite limits as $n\to\infty$. Its limit is
\begin{equation}\label{eq:hh-Tinf}
 \tilde T_\infty=\begin{pmatrix}
 \frac{729}{z^2}&-\frac{15309}{8z^2}&\frac{19683}{4z^2}\\[2pt]
 \frac{6561(z-1)}{8z^3}&-\frac{729(188z-189)}{64z^3}&\frac{2187(80z-81)}{32z^3}\\[2pt]
 \frac{15309(z-1)}{64z^3}&-\frac{19683(z-1)}{32z^3}&\frac{729(136z-135)}{64z^3}
 \end{pmatrix}.
\end{equation}
Expanding the determinant gives $\det(\tilde T_\infty-\frac{729}{16z^3}\rho I)=-\frac{387420489}{4096z^9}\chi_z(\rho)$, where $\chi_z(\rho)=\rho^3-(3z+\frac{27}2)\rho^2+(3z^2-\frac{81}2z+\frac{729}{16})\rho-z^3$. Hence the eigenvalues of $\tilde T_\infty$ are $\frac{729}{16z^3}\rho_i$, where $\rho_1<\rho_2<\rho_3$ are the roots of $\chi_z$, and $g(M)=\log(\rho_2/\rho_1)$.

\subsection{Proofs for the Explicit Approximants}\label{app:hh-approx}

\begin{proof}[Proof of Lemma~\ref{lem:hh-mobius}]
Write $F=\rho_pS_0^{\mathsf T}=F_0$, $E=(1,2,0)S_0^{\mathsf T}$ and $A=\rho_qS_0^{\mathsf T}$. Since $S_0=(F,\theta F,\theta^2F)$, we have $E=(1+2\theta)F$ and $A=(1+2\theta)E$.
The coefficient of $z^n$ in $F$ is $(\frac12)_n/(n!(2n+1)^2)$, and $1+2\theta$ multiplies it by $2n+1$. Hence
$E={}_2F_1(\frac12,\frac12;\frac32;z)$ and $A=\sum_n(\frac12)_nz^n/n!=(1-z)^{-1/2}$, so the two ratios in the lemma are $\sqrt{1-z}\,F$ and $\sqrt{1-z}\,E$.

The series $f_1,f_2$ are the unique power-series solutions of $(2x\frac d{dx}+1)f_1=\frac1{1-x}$ and
$(2x\frac d{dx}+1)f_2=\frac{f_1}{1-x}$. Since $dz/dx=-(1-z)^2$, we have $2x\frac d{dx}=2z(1-z)\frac d{dz}=2(1-z)\theta$ and $\frac1{1-x}=1-z$. Then
\begin{align*}
 \bigl(2(1-z)\theta+1\bigr)\bigl(\sqrt{1-z}\,E\bigr)&=\sqrt{1-z}\bigl(2(1-z)\theta E-zE+E\bigr)\\
 &=(1-z)^{3/2}(E+2\theta E)=(1-z)^{3/2}A=1-z,
\end{align*}
and in the same way $\bigl(2(1-z)\theta+1\bigr)\bigl(\sqrt{1-z}\,F\bigr)=(1-z)^{3/2}(F+2\theta F)=(1-z)\sqrt{1-z}\,E$.
Hence $\sqrt{1-z}\,E=f_1$ and $\sqrt{1-z}\,F=f_2$, initially near $x=0$. At $z=1/M$, $x=-1/k$ and $\sqrt{1-z}\,F=B(M)$ by \eqref{eq:hh-forms}.

For the moments, put $I_n=\int_0^1y^{2n}\ell(y^2)\,dy$. Direct integration gives $I_0=1$.
For $n\ge1$, integration by parts using $y^{2n+1}-y^{2n-1}$ (whose product with $\ell(y^2)$ vanishes at both endpoints) gives
\[
 (2n+1)I_n-(2n-1)I_{n-1}=\int_0^1y^{2n}\,dy=\frac1{2n+1}.
\]
Thus $I_n=O_n/(2n+1)$. Substituting $u=y^2$ gives $\int_0^1u^n\ell(u)\,du/(2\sqrt u)=O_n/(2n+1)$, and directly $\int_0^1u^n\,du/(2\sqrt u)=1/(2n+1)$.
Summing the geometric series proves the integral representations for $|x|<1$, and analytic continuation along $-1<x\le0$ gives the stated identities there.
At $x=-1$ both series converge by the alternating-series test (their coefficient magnitudes decrease to zero); Abel's theorem and continuity of the integrals give the endpoint identities.
\end{proof}

\begin{lemma}\label{lem:hh-unique}
\eqref{eq:hh-orth} has, up to a scalar, a unique nonzero solution $Q$ with $\deg Q\le2n$.
\end{lemma}
\begin{proof}
The case $n=0$ is immediate. For $n\ge1$, let $h=p_1+p_2\ell$ with $\deg p_i<n$, not both zero. If $p_2\ne0$, then $\frac{d^n}{du^n}h=R/(1-u)^n$ with $0\ne R$ and $\deg R\le n-1$.
To see that $R\ne0$, expand $p_2$ in $(1-u)^j$: for $0\le j<n$,
$(1-u)^n\frac{d^n}{du^n}[(1-u)^j\ell(u)]=\frac{(-1)^j}{2}j!(n-j-1)!(1-u)^j$.
By Rolle's theorem, $h$ has at most $2n-1$ zeros in $(0,1)$ counted with multiplicity; this also holds trivially if
$p_2=0$. Thus $(u^{1/2},u^{1/2}\ell)$ is an AT system, which gives uniqueness up to scale of the type II orthogonal polynomial \cite{nikishin1991rational}.
\end{proof}

\begin{proof}[Proof of Proposition~\ref{prop:hh-rodrigues}]
Leibniz's rule with $\frac{d^n}{du^n}u^\gamma/n!=\binom\gamma nu^{\gamma-n}$ gives
\begin{equation}\label{eq:hh-QN-expansion}
 Q_n(u)=\sum_{i,j=0}^n(-1)^{i+j}\binom ni\binom nj\binom{2n+j+\frac12}n\binom{n+i+j+\frac12}nu^{i+j}.
\end{equation} Put $\Psi_1=u^{2n+1/2}(1-u)^n$ and $\Psi_2=(1-u)^n\frac{d^n}{du^n}\Psi_1$;
both vanish to order $\ge n$ at $0$ and $1$, and $u^{1/2}Q_n=\frac{d^n}{du^n}\Psi_2/n!^2$. Integrating by parts $n$ times gives
$\int u^{j+1/2}Q_n=\pm n!^{-2}\int\Psi_2\frac{d^n}{du^n}u^j=0$ and $\int u^{j+1/2}Q_n\ell=\pm n!^{-2}\int\Psi_2\frac{d^n}{du^n}(u^j\ell)$. The boundary terms
vanish because $\frac{d^i}{du^i}\Psi_2=O((1-u)^{n-i})$. Since $\frac{d^n}{du^n}(u^j\ell)=R_j/(1-u)^n$ with $\deg R_j<n$, the last integral is a constant multiple of
$\int R_j\frac{d^n}{du^n}\Psi_1=\pm\int\Psi_1\frac{d^n}{du^n}R_j=0$. Lemma~\ref{lem:hh-unique} gives $Q$ as a nonzero rational multiple of $Q_n$. Rescaling $\hat x_n$ gives $Q=Q_n$ without changing the CMF ratios, and $\tilde f$ is the stated truncation.
\end{proof}

\begin{proof}[Proof of Lemma~\ref{lem:hh-integers}]
At $z=1/M$ we have $x=-1/k$, so $r_n=\tilde f(-1/k)/\tilde a(-1/k)=A_n/D_n$ by Proposition~\ref{prop:hh-rodrigues}.
The expansion \eqref{eq:hh-QN-expansion} gives $(-1)^lq_l>0$ for $0\le l\le2n$, so $D_n>0$.
For an integer $j\ge n$, $4^n\binom{j+1/2}n=2^n\prod_{t<n}(2j+1-2t)/n!$ is an integer. Indeed, a product of $n$
consecutive odd integers is divisible by the odd part of $n!$, and the $2$-adic valuation satisfies $v_2(n!)<n$ for $n\ge1$ (the case $n=0$ is immediate). So $16^nq_l\in\mathbb Z$. Moreover
$k^{2n}\tilde f(-1/k)=\sum_lq_l\sum_{t\le l}(-1)^{l-t}k^{l-t}\frac{O_t}{2t+1}$, and $L_{4n+1}^2\frac{O_t}{2t+1}\in\mathbb Z$ for $t\le2n$. Since
$\binom yn$ increases for $y\ge n$, $|q_l|\le\binom{3n+1/2}n^2\binom{2n}l$, hence $|D_n|\le16^nL_{4n+1}^2\binom{3n+1/2}n^2M^{2n}$. Finally
$\binom{3n+1/2}n^{1/n}\to\frac{27}4$ by Stirling's formula, and $L_{4n+1}\le\lcm(1,\dots,4n+1)=e^{(4+o(1))n}$ by the prime number theorem.
\end{proof}

\subsection{Proof of Proposition~\ref{prop:hh-rate}}\label{app:hh-rate}
By \eqref{eq:hh-Tinf}, $\tilde T(n)\to\tilde T_\infty$, whose eigenvalues are $\frac{729}{16z^3}\rho_i$. The discriminant of
$\chi_z$ is $\frac{19683}{256}z(1-z)(32z-81)^2>0$ for $0<z<1$, so its roots are real and distinct. Moreover $\chi_z<0$ on $(-\infty,0]$,
because its coefficients alternate in sign; hence $0<\rho_1<\rho_2<\rho_3$. The moduli are therefore distinct and $\tilde T_\infty$
is invertible. By Remark~\ref{rem:separated-spectrum} and Perron's theorem \cite{elaydi2005introduction}, the recessive
solutions grow at rate $\log\frac{729\rho_1}{16z^3}$ and all others at rate $\ge\log\frac{729\rho_2}{16z^3}$.

Here $0<z\le1/2$. Since
$\chi_z(2)=(\frac12-z)(z^2-\frac{11}2z+\frac{361}4)\ge0$ and $\rho_3>\frac13(3z+\frac{27}2)>2$,
we have $\rho_2\ge2$. The nonzero row $\tilde Y_n=S_nU(n)$ solves $\tilde Y_{n+1}=\tilde Y_n\tilde T(n)$.
Termwise $(\frac12+2n)_j^3/(\frac32+3n)_j^2\le(\frac12+2n)_j$, so
$|\theta^iF_n|\le2(2n+2)^2(1-z)^{-2n-3}$ for $i=0,1,2$.
Its exponential growth rate is at most $-2\log(1-z)$, whereas
$729\rho_2/(16z^3)\ge729>4\ge(1-z)^{-2}$. Hence $\tilde Y$ is recessive.

We now check the terminal column required by Theorem~\ref{thm:dual-convergence}. From \eqref{eq:hh-Tinf},
\begin{equation}\label{eq:hh-cyclic}
 \det\bigl[e_0,\;\tilde T_\infty^{-\mathsf T}e_0,\;(\tilde T_\infty^{-\mathsf T})^2e_0\bigr]
 =-\frac{8192\,z^5(32z-81)}{14348907}\ne0\qquad(0<z<1),
\end{equation}
so $e_0$ is a cyclic vector of $\tilde T_\infty^{-\mathsf T}$, equivalently of $\tilde T_\infty^{\mathsf T}$.

We show that this makes $e_0$ nonexceptional. Let $H$ be the complementary space of Definition~\ref{def:weak-pp} for $\tilde T$, and let $h\ne0$ be a column with $Hh=0$; then $\tilde Y_0h\ne0$. The rows $\ell_n$ in the proof of Theorem~\ref{thm:dual-convergence}, formed for $\tilde T$, are
\[
 \ell_n^{\mathsf T}=\frac{\Phi_{\tilde T}(n)^{-1}h}{\tilde Y_0h},
\]
so $\zeta_n=\ell_n^{\mathsf T}$ solves $\zeta_{n+1}=\tilde T(n)^{-1}\zeta_n$, and $\|\zeta_n\|\ge1/\|\tilde Y_n\|$ because $\ell_n\tilde Y_n^{\mathsf T}=1$.
The eigenvalues of $\tilde T_\infty^{-1}$ are $\mu_i=16z^3/(729\rho_i)$, with $\mu_1>\mu_2>\mu_3>0$, and since $\tilde Y_n$ grows at rate $\log(729\rho_1/(16z^3))$, we have $\|\zeta_n\|\ge\mu_1^ne^{o(n)}$.
In an eigenbasis of $\tilde T_\infty$, the system reads $\zeta_{n+1}=(\operatorname{diag}(\mu_1,\mu_2,\mu_3)+E_n)\zeta_n$ with $\|E_n\|_\infty\to0$. Let $a_n$ be the modulus of the first coordinate of $\zeta_n$, $b_n$ the larger modulus of the other two, and $q_n=b_n/a_n$.
If $q_n>1$ for all large $n$, then $\|\zeta_{n+1}\|_\infty=b_{n+1}\le(\mu_2+\|E_n\|_\infty)\|\zeta_n\|_\infty$, which contradicts $\|\zeta_n\|\ge\mu_1^ne^{o(n)}$.
If $q_n\le1$, then $a_{n+1}\ge(\mu_1-\|E_n\|_\infty)a_n$ and $b_{n+1}\le\mu_2b_n+\|E_n\|_\infty a_n$, so $q_{n+1}\le(\mu_2q_n+\|E_n\|_\infty)/(\mu_1-\|E_n\|_\infty)$; for large $n$ this keeps $q_n\le1$ and forces $q_n\to0$.
Hence $[\ell_n]\to[v^{\mathsf T}]$, where $v$ is a right eigenvector of $\tilde T_\infty$ for its smallest eigenvalue $729\rho_1/(16z^3)$.
If $e_0^{\mathsf T}v=0$, then $\bigl((\tilde T_\infty^{-\mathsf T})^je_0\bigr)^{\mathsf T}v=e_0^{\mathsf T}\tilde T_\infty^{-j}v=(16z^3/(729\rho_1))^j\,e_0^{\mathsf T}v=0$ for $j=0,1,2$, so $v$ is orthogonal to the three columns in \eqref{eq:hh-cyclic}. These columns span $\mathbb R^3$, so $v=0$, a contradiction.
Therefore $e_0^{\mathsf T}v\ne0$, and $\|\ell_n\|\,\|e_0\|/|\ell_ne_0|\to\|v\|/|e_0^{\mathsf T}v|<\infty$: the column $e_0$ is nonexceptional.

By Remark~\ref{rem:separated-spectrum}, the estimate of Definition~\ref{def:weak-pp} holds with exponent $g(M)-\varepsilon$ for every $\varepsilon>0$.
Since $U(0)=I$ and $U(n)^{-\mathsf T}e_0=e_0$, coboundary equivalence gives
$\Phi^\vee(n)e_0=\Phi_{\tilde T}^\vee(n)e_0$.
As $\rho_q\tilde Y_0^{\mathsf T}=A\ne0$, Corollary~\ref{cor:dual-rank-one} gives the limit $F/A=B(M)$ and
$\limsup_{n\to\infty}n^{-1}\log|B(M)-r_n|\le-g(M)+\varepsilon$.
Letting $\varepsilon\downarrow0$ proves the claimed error bound.\hfill$\square$

\subsection{Proofs of Lemma~\ref{lem:hh-nondeg} and Proposition~\ref{prop:hh-gap}}\label{app:hh-arith}

\begin{lemma}\label{lem:hh-roots}
If $M\ge7165$, then $\rho_1<\frac{16}{729}z^3(1+\frac89z+z^2)$ and $\frac{27}4>\rho_2>\frac{27}4-7.794229\sqrt z+\frac32z$.
\end{lemma}

\begin{proof}
Write $s=\sqrt z$, and $t_1$ and $m_2$ for the two bounds in the lemma, so $t_1=\frac{16}{729}z^3(1+\frac89z+z^2)$ and $m_2=\frac{27}4-7.794229s+\frac32s^2$. Then $0<s<\frac{119}{10^4}$ and clearly $0<t_1<m_2<\frac{27}4$. We have $\chi_z(0)=-z^3<0$ and $\chi_z(\frac{27}4)=-z(\frac{6561}{16}-\frac{81}4z+z^2)<0$.
Expanding, $\chi_z(t_1)=z^5\bigl(\frac{67}{243}+zR(z)\bigr)$, where $R$ is a polynomial of degree $9$ whose coefficients have
modulus $<1$, so $\frac{67}{243}+zR(z)\ge\frac{67}{243}-\frac{z}{1-z}>0$ in the stated range.
Next, $\chi_z(m_2)=s^2\bigl(c_0+c_1s+s^2(c_2+c_3s)+\frac18s^4\bigr)$, where
$c_0=\frac{154019907}{4\cdot10^{12}}$, $c_1=-\frac{44461719470989}{10^{18}}$, $c_2=\frac{10125017113323}{2\cdot10^{12}}$ and
$c_3=-\frac{23382687}{4\cdot10^6}$. Since $c_1,c_3<0$, and $c_0+c_1s$ and $c_2+c_3s$ are positive at $s=\frac{119}{10^4}$, we get $\chi_z(m_2)>0$. The sign pattern
$-,+,+,-$ of the monic cubic at $0<t_1<m_2<\frac{27}4$ places $\rho_1\in(0,t_1)$ and $\rho_2\in(m_2,\frac{27}4)$.
\end{proof}

\begin{proof}[Proof of Proposition~\ref{prop:hh-gap}]
By Lemma~\ref{lem:hh-roots}, $\rho_2=\frac{27}4+O(M^{-1/2})$. The roots of $\chi_z$ satisfy $\rho_1\rho_2\rho_3=z^3$ and
$\rho_2\rho_3=3z^2-\frac{81}2z+\frac{729}{16}-\rho_1(\rho_2+\rho_3)=\frac{729}{16}+O(z)$, so $\rho_1=\frac{16}{729}z^3(1+O(z))$.
Hence $g=\log(\rho_2/\rho_1)=3\log M+\log\frac{27}4+\log\frac{729}{16}+O(M^{-1/2})$, and subtracting
$\sigma=2\log M+\log16+2\log\frac{27}4+8$ gives the asymptotic formula, since $\frac{729}{16\cdot16\cdot\frac{27}4}=\frac{27}{64}$.

For the explicit threshold, let $M\ge7165$. Lemma~\ref{lem:hh-roots} gives
\[
 g-\sigma>\log M-\log16-8+\log\Bigl(\frac{27}4-\frac{7.794229}{\sqrt M}+\frac3{2M}\Bigr)-\log\Bigl(1+\frac8{9M}+\frac1{M^2}\Bigr).
\]
Each term on the right is increasing in $M$, and at $M=7165$ the right-hand side equals $8.91\ldots\cdot10^{-5}>0$.
\end{proof}


\begin{proof}[Proof of Lemma~\ref{lem:hh-nondeg}]
The rows $Y_n=x\,\Phi^\vee(n)$ solve $Y_{n+1}=Y_nT(n)^{-\mathsf T}$, so the numbers $x\,\Phi^\vee(n)e_0$ satisfy a common scalar recurrence of order $3$, obtained from the columns $e_0,\,T(n)^{-\mathsf T}e_0,\,T(n)^{-\mathsf T}T(n+1)^{-\mathsf T}e_0$. Their determinant is $\Delta(n)=c(n)\,p(n,z)$, where $c(n)\ne0$ for $n\ge0$ and every coefficient of the polynomial $p(n,z)$ in $n$ has the form $a z-b$ with $0<a<b$; hence $p(n,z)<0$ for $n\ge0$ and $0<z\le1$. The trailing coefficient of the recurrence is a nonzero multiple of $p(n+1,z)/p(n,z)$, so it does not vanish for $n\ge0$, and $\det[v_0,v_1,v_2]=\Delta(0)\ne0$. Lemma~\ref{lem:nonvanishing} with $x=u-B(M)w$ shows that $D_nB(M)-A_n$, a nonzero multiple of $x\,\Phi^\vee(n)e_0$, does not vanish at three consecutive indices.
\end{proof}

\paragraph{The approximation exponent.}
Once $B(M)$ is proved irrational, convergence implies that the reduced denominators $\widetilde D_n$ tend to infinity.
For every $\varepsilon>0$, the proved bounds give, eventually,
\[
 -1-\frac{\log|B(M)-r_n|}{\log\widetilde D_n}
 \ge -1+\frac{g(M)-\varepsilon}{\sigma(M)+\varepsilon},
\]
since $\widetilde D_n\le D_n$. Keeping only the first occurrence of each rational value leaves an infinite subsequence,
and its liminf is at least $g(M)/\sigma(M)-1$ by Definition~\ref{def:delta}. 


\subsection{Other Irrationality Criteria}\label{app:hh-other}

\subsubsection{The Fischler--Rivoal bound}

\begin{theorem}\label{thm:hh-fr}
$B(M)\notin\mathbb Q$ for every integer $M$ with $M-1>2^{2412}$.
\end{theorem}

\textbf{The criterion.} 
Theorem~2 in \cite{fischler2018rational} (and its erratum \cite{fischler2018erratum}) proves irrationality under the bound conditions. 
The constants are those of its proof, \S3.2, computed from the system data of \S2.1. In this subsection $Y,A,D,F_j,r,d,C,\mathcal D,d_n,a,b,c_1,c_2$
follow the notation of \cite{fischler2018rational}. Let $Y=(1,F_1,\ldots,F_r)^{\mathsf T}$ consist of
$G$-functions with rational Taylor coefficients, linearly independent over $\mathbb Q(x)$, with $Y'=AY$. Let
$D\in\mathbb Z[x]\setminus\{0\}$ satisfy $DA\in\operatorname{Mat}_{r+1}(\mathbb Z[x])$, $\deg D\le d$ and $\deg(DA)_{ij}\le d-1$. Let $C,\mathcal D\ge1$, and let
integers $1\le d_n\le\mathcal D^{n+1}$ satisfy $|f_{j,n}|\le C^{n+1}$ and $d_nf_{j,m}\in\mathbb Z$ for $m\le n$, where
$F_j=\sum f_{j,n}x^n$. Then $F_j(a/b)\notin\mathbb Q$ for positive integers $a,b$ with $b>(c_1a)^{c_2}$, where
\begin{equation}\label{eq:hh-frconstants}
 c_1=4H(D)(C\mathcal D)^{8rd+1}C,\qquad c_2=3(r+2),
\end{equation}
and $H(D)$ is the largest absolute coefficient of $D$. Section~3.1 of \cite{fischler2018rational} is stated for a general
independent system, and the parameter choice in \S3.2 uses only the data $(r,d,H(D),C,\mathcal D)$.

\textbf{The system.} Put $\mathcal F(x)=f_1(-x)$ and $\mathcal G(x)=f_2(-x)$, with $f_1,f_2$ as in Lemma~\ref{lem:hh-mobius}.
By that lemma, $B(M)=\mathcal G(1/k)$, and coefficient comparison gives, with $'=d/dx$,
\begin{equation}\label{eq:hh-firstorder}
 2x\mathcal F'+\mathcal F=\frac1{1+x},\qquad 2x\mathcal G'+\mathcal G=\frac{\mathcal F}{1+x}.
\end{equation}
For $Y=(1,\mathcal F,\mathcal G)^{\mathsf T}$ these give $D=2x(1+x)$ and
$DA=\bigl(\begin{smallmatrix}0&0&0\\1&-(1+x)&0\\0&1&-(1+x)\end{smallmatrix}\bigr)$, so $r=2$, $d=2$ and $H(D)=2$.

\textbf{Independence.} Set $t=1+x\to0^+$. Then
$\mathcal F(-1+t)=\frac{\operatorname{arctanh}\sqrt{1-t}}{\sqrt{1-t}}=-\frac12\log t+\log2+O(t|\log t|)$. By
\eqref{eq:hh-firstorder},
$\frac d{dt}\bigl(\sqrt{1-t}\,\mathcal G(-1+t)\bigr)=-\frac{\mathcal F(-1+t)}{2t\sqrt{1-t}}=\frac{\log t}{4t}-\frac{\log2}{2t}+O(|\log t|)$, hence
$\mathcal G(-1+t)=\frac18\log^2t+O(|\log t|)$. Suppose a nontrivial $\mathbb Q(x)$-relation among $1,\mathcal F,\mathcal G$ existed.
Expand its coefficients at $t=0$ and divide by the smallest power of $t$. The terms with nonzero constant
coefficient have orders $\log^2t$, $\log t$ and $1$, so the largest cannot cancel, while the remaining terms
tend to $0$. Hence $1,\mathcal F,\mathcal G$ are independent over $\mathbb Q(x)$.

\textbf{Growth constants.} The coefficients have modulus $\frac1{2n+1}$ and $\frac{O_n}{2n+1}\le1$, so $C=1$. Let
$d_n=L_{2n+3}^2$; it clears all coefficients of index $\le n$. The inequality $d_n\le64^{n+1}$ was checked for $n<600$ by an
exact integer computation. For $n\ge600$, with $\psi(v)=\sum_{p^a\le v}\log p$ and $\psi(v)<1.03883v$
\cite[Thm.~12, (3.35)]{rosser1962approximate}, we get $\log d_n\le2\psi(2n+3)<2.07766(2n+3)<(n+1)\log64$. Indeed, at $n=600$ the
last inequality reads $2499.42\ldots<2499.48\ldots$, and both sides are linear in $n$ with slopes $4.15532<\log64=4.15888\ldots$.
So $\mathcal D=64$. The system $Y'=AY$ and these bounds show that $\mathcal F$ and $\mathcal G$ are $G$-functions.

\textbf{Conclusion.} \eqref{eq:hh-frconstants} gives $c_1=4\cdot2\cdot64^{33}=2^3\cdot2^{198}=2^{201}$ and $c_2=12$, so $c_1^{c_2}=2^{2412}$.
Applying the criterion with $a=1$ and $b=k$ proves Theorem~\ref{thm:hh-fr}.

\subsubsection{The polylogarithm criterion}

\begin{theorem}\label{thm:hh-kaw}
$B(M)\notin\mathbb Q$ for every integer $M>2^{26}e^{16}$, by Corollary~2.3 of \cite{kawashima2026linear}.
\end{theorem}

For $M\ge5$, put $\theta=\arctan(1/\sqrt k)$ and $\omega=2/(1-i\sqrt k)=1-e^{-2i\theta}$, so $|\omega|=2/\sqrt M<1$.

\begin{lemma}\label{lem:hh-dilog}
$B(M)=\frac{\sqrt k}2\operatorname{Im}\operatorname{Li}_2(\omega)$, where $\operatorname{Li}_s(\omega)=\sum_{n\ge1}\omega^n/n^s$.
\end{lemma}
\begin{proof}
Integrating by parts and using $\int_0^\theta\log(2\sin u)\,du=-\frac12\operatorname{Cl}_2(2\theta)$ gives
$\int_0^\theta u\cot u\,du=\theta\log(2\sin\theta)+\frac12\operatorname{Cl}_2(2\theta)$, where
$\operatorname{Cl}_2(\vartheta)=\operatorname{Im}\operatorname{Li}_2(e^{i\vartheta})$ \cite{lewin1981polylogarithms}. Euler's reflection formula
$\operatorname{Li}_2(1-u)=\frac{\pi^2}6-\operatorname{Li}_2(u)-\log u\log(1-u)$ \cite{lewin1981polylogarithms}, with principal branches,
applies at $u=e^{-2i\theta}$: $1-u=\omega$, $\log u=-2i\theta$ and $\log\omega=\log(2\sin\theta)+i(\frac\pi2-\theta)$. This gives
$\operatorname{Im}\operatorname{Li}_2(\omega)=\operatorname{Cl}_2(2\theta)+2\theta\log(2\sin\theta)$. Now compare with \eqref{eq:hh-forms}.
\end{proof}

\begin{proof}[Proof of Theorem~\ref{thm:hh-kaw}]
Let $K=\mathbb Q(\sqrt{-k})$, with $\sqrt{-k}=i\sqrt k$. Since $\omega/(\omega-1)=\bar \omega$ and $\operatorname{Li}_2$ has real coefficients, Landen's
identity $\operatorname{Li}_2(\frac \omega{\omega-1})=-\operatorname{Li}_2(\omega)-\frac12\log^2(1-\omega)$ \cite{lewin1981polylogarithms} and
$\operatorname{Li}_1(\omega)=-\log(1-\omega)$ give $2i\operatorname{Im}\operatorname{Li}_2(\omega)=2\operatorname{Li}_2(\omega)+\frac12\operatorname{Li}_1(\omega)^2$.
So, by Lemma~\ref{lem:hh-dilog}, $B(M)=p/q$ would give the nontrivial $K$-linear relation
$2\operatorname{Li}_2(\omega)+\frac12\operatorname{Li}_1(\omega)^2-\frac{4p}{qk}\sqrt{-k}=0$.

Apply \cite[Corollary~2.3]{kawashima2026linear} at the complex place $v_0$ of $K$, where the normalized absolute value is the usual
one. As there, $h$ is the absolute logarithmic Weil height and $h_{v_0}$ its local part at $v_0$. Take $m=1$ and $r=2$, so the number written $M$ there is $(m+1)^r-1=3$. Take $\alpha=2$ and $\beta=1-\sqrt{-k}$, so that
$\alpha/\beta=\omega$. Since $\beta$ is an algebraic integer with $|\beta|=\sqrt M>2$, we have $h(\beta)=h_{v_0}(\beta)=\frac12\log M$ and
$h(\alpha)=h_{v_0}(\alpha)=\log2$. The quantity (1) of \cite{kawashima2026linear} is then
\[
 V_{v_0}(\alpha,\beta)=4h(\beta)-h(\alpha)-3\bigl(h(\beta)+2h(\alpha)\bigr)-(3\log2+3\log2+2+6)=\tfrac12\log M-13\log2-8 ,
\]
which is positive for $M>2^{26}e^{16}$. The corollary makes $1,\operatorname{Li}_1(\omega),\operatorname{Li}_2(\omega),\operatorname{Li}_1(\omega)^2$
linearly independent over $K$, a contradiction.
\end{proof}

\subsubsection{Summary of the bounds}

The general criterion of Fischler and Rivoal \cite{fischler2018rational} gives $B(M)\notin\mathbb Q$ for
$M-1>2^{2412}\approx10^{726}$ (Theorem~\ref{thm:hh-fr}). Kawashima's polylogarithm criterion \cite{kawashima2026linear}
gives $M>2^{26}e^{16}\approx5.96\cdot10^{14}$ (Theorem~\ref{thm:hh-kaw}). Theorem~\ref{thm:hh-main} therefore lowers
the threshold by about eleven orders of magnitude, and it does so with explicit approximating sequences.
\subsection{Experimental Exponents}\label{app:hh-exp}

Beyond the proven theorem for a general $M$, we discovered experimentally an apparent pattern in the irrationality exponent for different values of $M$.
We sampled 700 integers $432\le M\le65536$: 100 odd, 100 with $M\equiv2\pmod4$, and 500 with $4\mid M$. 
For each, the exponent $\delta_M(n)=-1-\log|B(M)-r_n|/\log\widetilde D_n$, with $\widetilde D_n$ the reduced denominator of $r_n=A_n/D_n$, was evaluated using Arb enclosures for the error at
$n=1000,1200,\dots,2000$ and extrapolated as $\delta(M)+a\log n/n$. Points show $\delta(M)$, curves are quadratic fits in
$\log_2M$, and crosses mark their zeros near $458$ ($4\mid M$), $1860$ ($M\equiv2\bmod4$) and $7197$ (odd $M$). 
For odd $M$ the observed crossing lies close to the proved threshold $7165$. 
For even $M$ the reduced denominators are observed to be smaller by factors of about $4^n$ ($M\equiv2\bmod4$) and $16^n$ ($4\mid M$).
This suggests that an argument exploiting this $2$-adic cancellation could reach thresholds near $1.9\cdot10^3$ and $4.6\cdot10^2$, respectively.

\begin{figure}[H]
\centering
\includegraphics[width=0.8\linewidth]{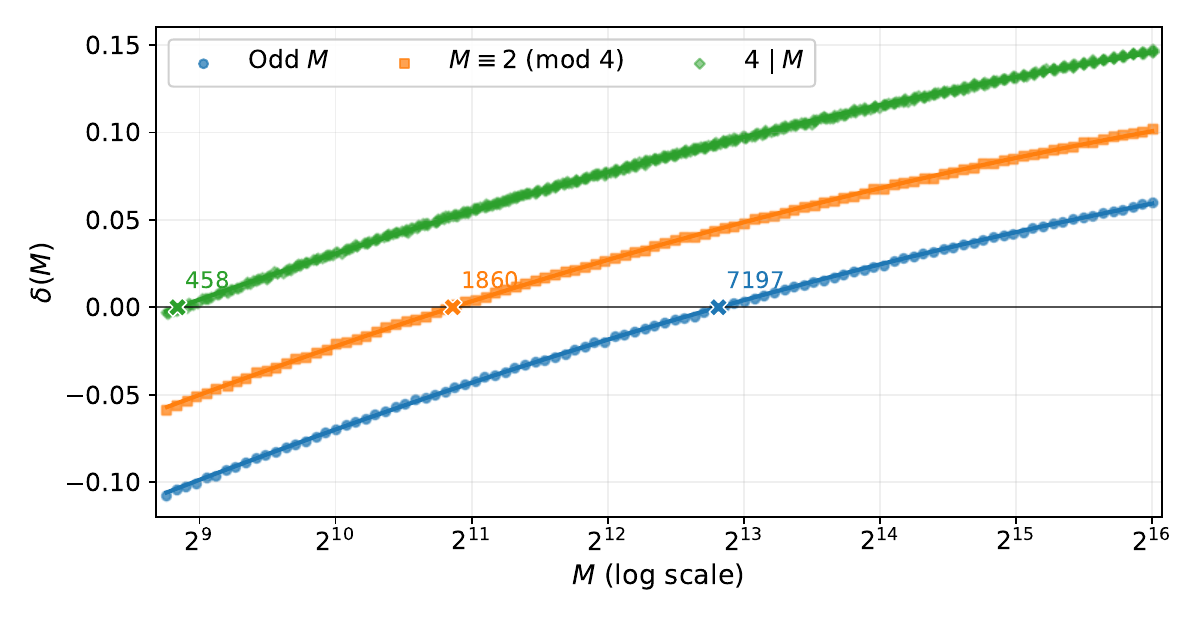}
\caption{\emph{Experimentally discovered irrationality exponents of $B(M)$ exhibit $2$-adic cancellation.} }

\label{fig:hh-exp}
\end{figure}

\end{document}